%% file: 2021_LBM-MR_and_Eq_Eq.tex
\documentclass[a4paper]
{cedram-smai-jcm}
\usepackage[utf8]{inputenc}
\usepackage{amsmath}
\usepackage{amssymb}
\usepackage{amsthm}
\usepackage{bm}
\usepackage{graphicx}
\usepackage{stmaryrd}
\usepackage{xcolor}
\usepackage{placeins}
\usepackage{booktabs}
\usepackage{hyperref}

\newcommand{\distribution}{f}
\newcommand{\latticevelocity}{\lambda}
\newcommand{\spatialdimensionality}{d}
\newcommand{\momentum}{m}
\newcommand{\average}[1]{\overline{#1}}
\newcommand{\vectorial}[1]{\bm{#1}}
\newcommand{\operatorial}[1]{\bm{#1}}
\newcommand{\levelletter}{\ell}
\newcommand{\maxlevel}{\overline{L}}
\newcommand{\minlevel}{\underline{L}}
\newcommand{\indexletter}{k}
\newcommand{\populationindex}{\alpha}
\newcommand{\velocityletter}{e}
\newcommand{\normalizedvelocityletter}{c}
\newcommand{\superscript}[2]{#1^{#2}}
\newcommand{\subscript}[2]{#1_{#2}}
\newcommand{\collided}{\star}
\newcommand{\cellletter}{C}
\newcommand{\timestep}{\Delta t}
\newcommand{\spacestep}{\Delta x}
\newcommand{\leveldifference}{\Delta \levelletter}
\newcommand{\adaptiveroundbrackets}[1]{\left ( #1 \right )}
\newcommand{\predictionstencil}{\gamma}
\newcommand{\velocitynumber}{q}
\usepackage{scalerel,stackengine}
\stackMath
\newcommand\reallywidehat[1]{%
\savestack{\tmpbox}{\stretchto{%
  \scaleto{%
    \scalerel*[\widthof{\ensuremath{#1}}]{\kern-.6pt\bigwedge\kern-.6pt}%
    {\rule[-\textheight/2]{1ex}{\textheight}}%WIDTH-LIMITED BIG WEDGE
  }{\textheight}% 
}{0.5ex}}%
\stackon[1pt]{#1}{\tmpbox}%
}
\newcommand\reallywidedoublehat[1]{%
\savestack{\tmpbox}{\stretchto{%
  \scaleto{%
    \scalerel*[\widthof{\ensuremath{#1}}]{\kern-.6pt\bigwedge\kern-.6pt}%
    {\rule[-\textheight/2]{1ex}{\textheight}}%WIDTH-LIMITED BIG WEDGE
  }{\textheight}% 
}{0.5ex}}%
\stackon[-0.3mm]{\stackon[1pt]{#1}{\tmpbox}}{\tmpbox}%
}
\newcommand{\predicted}[3]{\reallywidehat{#1} \subscript{\superscript{\vphantom{#1}}{#2}}{#3}}
\newcommand{\reconstructed}[3]{\reallywidedoublehat{#1} \subscript{\superscript{\vphantom{#1}}{#2}}{#3}}

\title[analysis of the LBM-MR schemes \emph{via} the equivalent equations]{High accuracy analysis of adaptive multiresolution-based lattice Boltzmann schemes \emph{via} the equivalent equations}

\author[T. Bellotti]{\firstname{Thomas} \lastname{Bellotti}}
\address{CMAP, CNRS, Ecole polytechnique, Institut Polytechnique de Paris, 91128 Palaiseau Cedex, France.}
\email{thomas.bellotti@polytechnique.edu}

\author[L. Gouarin]{\firstname{Lo\"{\i}c} \lastname{Gouarin}}
\address{CMAP, CNRS, Ecole polytechnique, Institut Polytechnique de Paris, 91128 Palaiseau Cedex, France.}
\email{loic.gouarin@polytechnique.edu}

\author[B. Graille]{\firstname{Benjamin} \lastname{Graille}}
\address{Université Paris-Saclay, CNRS, Laboratoire de mathématiques d’Orsay, 91405, Orsay, France.}
\email{benjamin.graille@universite-paris-saclay.fr}

\author[M. Massot]{\firstname{Marc} \lastname{Massot}}
\address{CMAP, CNRS, Ecole polytechnique, Institut Polytechnique de Paris, 91128 Palaiseau Cedex, France.}
\email{marc.massot@polytechnique.edu}



\begin{document}

% Abstract
\begin{abstract}
    Multiresolution provides a fundamental tool based on the wavelet theory to build adaptive numerical schemes for Partial Differential Equations and time-adaptive meshes, allowing for error control.
    We have introduced this strategy before to construct adaptive lattice Boltzmann methods with this interesting feature.
    Furthermore, these schemes allow for an effective memory compression of the solution when spatially localized phenomena -- such as shocks or fronts -- are involved, to rely on the original scheme without any manipulation at the finest level of grid and to reach a high level of accuracy on the solution.
    Nevertheless, the peculiar way of modeling the desired physical phenomena in the lattice Boltzmann schemes calls, besides the possibility of controlling the error introduced by the mesh adaptation, for a deeper and more precise understanding of how mesh adaptation alters the physics approximated by the numerical strategy. 
    In this contribution, this issue is studied by performing the equivalent equations analysis of the adaptative method after writing the scheme under an adapted formalism. 
    It provides an essential tool to master the perturbations introduced by the adaptive numerical strategy, which can thus be devised to preserve the desired features of the reference scheme at a high order of accuracy.
    The theoretical considerations are corroborated by numerical experiments in both the 1D and 2D context, showing the relevance of the analysis. In particular, we show that our numerical method outperforms traditional approaches, whether or not the solution of the reference scheme converges to the solution of the target equation.
    Furthermore, we discuss the influence of various collision strategies for non-linear problems, showing that they have only a marginal impact on the quality of the solution, thus further assessing the proposed strategy.
\end{abstract}

\maketitle
%\tableofcontents

\section{Introduction}

    The lattice Boltzmann schemes are used to investigate many interesting problems for which the solutions have variability unevenly distributed in space.
    Therefore, one can take advantage of this property to reduce both the memory footprint -- especially when deployed to solve large problems -- and the computational cost of each iteration by selectively adapting the computational mesh.
    In our previous contributions \cite{bellotti2021sisc, bellotti2021multidimensional}, we have introduced and tested a novel way of implementing lattice Boltzmann methods to be deployed on spatially adapted grids generated by multiresolution analysis (called LBM-MR method).
    The collision phase -- due to its inherent local nature -- was performed only on the leaves of the corresponding adaptive graded tree where the data are directly available by using the equilibria of the reference lattice Boltzmann scheme.
    The advection phase, which is linear but non-local, was performed with a technique closely related to the Corner Transport Upwind scheme (CTU), see Colella \cite{colella1990multidimensional}.
    Using the multiresolution formalism, we first build a piece-wise constant reconstruction of the particle distribution functions at the finest level of resolution, then advect the reconstruction exactly, and finally project the new functions on the adapted grid.
    The finest level of the resolution dictates the global time-step across all the mesh levels, so no adaptation of the speed of sound is needed. This approach to time discretization has also been used by Fakhari \emph{et al.} \cite{fakhari2014finite, fakhari2015numerics, fakhari2016mass} in combination with the AMR (Adaptive Mesh Refinement) technique.
    The method reaches the following goals and has the following peculiarities: it reduces the memory trace of the procedure by employing and storing the distributions only on an adaptive hybrid partition of the domain made up of leaves of an underlying graded tree structure; being fully adaptive, thus working only on the compressed representation of the mesh, it provides potential reductions of the computational time of the method; the use of the multiresolution analysis, decomposing the distributions on a wavelet basis and using a functional analysis theoretical background \cite{devore1984maximal}, ensures -- under reasonable assumptions -- that we can control the error of the adaptive method with respect to the reference method using a unique threshold small parameter; finally, the original lattice Boltzmann scheme on the uniform grid does not need to be modified to handle non-uniform lattices.
    
    Even though the previous method allows to bound the difference between the solution of the adaptive method and that of the reference method and has proved to 
    be highly accurate,  it does not say much about the real physics behind this new lattice Boltzmann method. 
    The question of determining the physical model approximated by the lattice Boltzmann method has to be discussed with particular care, especially when one wants to introduce diffusion terms with a specific structure and when a fine control on the stability constraints is needed.
    For example, when using an acoustic scaling, the diffusive terms do not show up at leading order.
    Therefore, in order to build reliable methods, we have to ensure that our adaptive approach does not alter these important contributions.
    
    Until nowadays, two main formal approaches -- which consistently yield the same results -- are devised to analyze the physics reproduced by lattice Boltzmann methods: the Chapman-Enskog expansion \cite{chapman1990mathematical} and the Taylor expansion method \cite{dubois2008equivalent}.
    The former approach is based on a Hilbert-like asymptotic expansion, whereas the latter follows the tracks of the consistency analysis for Finite Difference methods based on Taylor expansions.
    These methods remain formal because they assume that the derivatives of negligible terms are also negligible.

    In this contribution, we clarify how the quality of the utilized multiresolution analysis impacts the physical model approximated by the adaptive numerical scheme.
    More precisely, we show that the central tool to perform the multiresolution analysis, the so-called ``prediction operator''\footnote{Which can be simply seen as an interpolation operator in this introduction.}, must be accurate enough so that the adaptative method does not significantly alter the behavior of the numerical scheme performed on the finest level of grid. Otherwise, large deviations from the behavior of the reference scheme are theoretically expected and numerically observed.
    This pertains both to the phenomena present in the macroscopic target model (transport and diffusion, for example) and to higher-order terms that play a role in determining the stability of the numerical method. The latter were generally ignored by the analyses available in the literature but may lead to unexpected behaviors of the adaptive strategy.
    The analysis is conducted on locally uniform parts of the mesh at some coarser level of resolution than the finest available mesh size. We then show, using Taylor expansions, that the adaptive scheme acting on such group of coarse cells does not perturb the original lattice Boltzmann algorithm up to  a certain order in the space step. This accuracy result holds uniformly with respect to the local level of the grid, thus can be applied globally to any adapted mesh.
    
    %{\bf Marc : dans ce paragraphe on passe un peu trop vite sur l'originalite des contributions : on ne parle pas des divers niveaux de grille et de ce que l'on fait + on ne parle pas d'une etude detaillee sur la demonstration numerique des resultats theoriques - seul endroit a reprendre}.
    %In this contribution, {\bf we clarify how the several possible choices of multiresolution analysis - mal dit} have an impact on the physical model approximated by the adaptative scheme.
    %More precisely, we show that the central tool to perform the multiresolution analysis, the so-called ``prediction operator''\footnote{Which can be simply seen as an interpolation operator in this introduction.}, must be accurate enough so that the adaptative method does not significantly alter the behavior of the numerical scheme. Otherwise, large deviations from the behavior of the reference scheme are theoretically expected and numerically observed.
    %This pertains both to the phenomena present in the macroscopic target model (transport and diffusion, for example) and to higher-order terms that play a role in determining the stability of the numerical method. The latter were generally ignored by the analyses available in the literature but may lead to unexpected behaviors of the adaptive strategy.
    
    Thanks to the particular structure of the multiresolution analysis we have utilized, the extension to the multidimensional case is straightforward.
    
    The theoretical study is thoroughly assessed numerically both in one and two spatial dimensions and for linear and non-linear problems, showing very good agreement between the real and expected behavior of the adaptive scheme. 
    This confirms the predictive power of the analysis based on the equivalent equations.
    Moreover, we present a concise numerical study on the effect of different strategies for the treatment of the collision operator on the reliability of the adaptive method. 
    This study confirms the relevance of the local approach, introduced in the previous works \cite{bellotti2021sisc, bellotti2021multidimensional}, in which the collision phase is performed on the leaves. This yields accuracy levels comparable to those of more computationally expensive strategies, which are briefly introduced in this work.
    
    The paper is structured as follows: in Section \ref{sec:ReferenceLBMMethod}, we present a quick recap on the multiple-relaxation-times lattice Boltzmann schemes, setting the ground for Section \ref{sec:Geometry} which introduces the adaptive spatial discretization and the adaptive LBM-MR method.
    Particular care is devoted to the presentation of the two basic bricks of the method (called ``prediction'' and ``reconstruction'' operators) because of their preeminent role in the subsequent analysis.
    With these ideas in mind, we can go to the central sections of the work, namely Section \ref{sec:Equivalent1D} and \ref{sec:Equivalent2D}, where the ``reconstruction flattening'' allows to deploy the equivalent equations technique to the new LBM-MR method, by rewriting the reconstruction operator, usually seen in a recursive fashion, in a more friendly ``flattened'' way.
    This approach is used to analyze two of the most simple adaptive LBM-MR techniques as well as a Lax-Wendroff scheme present in the literature \cite{fakhari2014finite, fakhari2015numerics, fakhari2016mass} for comparison purposes.
    This theoretical analysis is corroborated by 1D and 2D simulations presented in Section \ref{sec:NumericalSimulations}.
    The final conclusions are drawn in Section \ref{sec:Conclusions}.
    
\section{The reference lattice Boltzmann scheme}\label{sec:ReferenceLBMMethod}

    The lattice Boltzmann methods are a class of numerical schemes used to solve numerous problems in applied mathematics and fluid mechanics. 
    Their fundamental building blocks are a particular link between the temporal and the spatial discretizations and a finite family of discrete velocities, a local collision phase and a linear stream phase.
    The aim of this section is to provide the basic ideas and notations about these methods as we shall employ them. Still, we do not aim at encompassing all the possibles shades and flavors of the lattice Boltzmann method: we refer the interested reader to the book \cite{kruger2017lattice} for more information.

\subsection{Discretization}

    The lattice Boltzmann method is based on a regular lattice of fixed step $\spacestep > 0$.
    The time discretization is performed by considering a uniform time discretization of step
    \begin{equation}\label{eq:TimeStepDefinition}
        \Delta t = \frac{\spacestep}{\latticevelocity},
    \end{equation}
    where $\latticevelocity > 0$ is the so-called ``lattice velocity''.
    Another possible choice is the so-called ``parabolic scaling'' $\timestep \sim \spacestep^2$: the following analysis is still valid for this choice.
    Then, one introduces $\velocitynumber \in \mathbb{N^{\star}}$ velocities $(\vectorial{\velocityletter}_{\populationindex})_{\populationindex = 0}^{\populationindex = \velocitynumber - 1} \subset \latticevelocity \mathbb{Z}^{\spatialdimensionality}$ compatible with the lattice, namely multiple of the lattice velocity $\latticevelocity$. We write $\vectorial{\normalizedvelocityletter}_{\populationindex} := \vectorial{\velocityletter}_{\populationindex}/\latticevelocity \in \mathbb{Z}^{\spatialdimensionality}$, $\populationindex = 0, \dots, \velocitynumber-1$, which are the dimensionless velocities.
    The distribution function of the population moving with velocity $\vectorial{\velocityletter}_{\populationindex}$ shall be denoted $\distribution^{\populationindex}$.
    It is customary to indicate schemes with the notation D$d$Q$q$, where $\spatialdimensionality$ is the spatial dimensionality and $q$ is the number of discrete velocities.
    
    \subsection{Collide-and-stream}
    
    We now assume that $t$ is the discrete time at which we know the solution and that we want to update it towards time $t + \timestep$, where $\timestep$ is given by \eqref{eq:TimeStepDefinition}.
    Any lattice Boltzmann method can be resumed as a collide-and-stream algorithm made up of the following steps:
    \begin{itemize}
            \item The collision phase is performed locally at each point of the lattice.
            To employ multiple relaxation times, it is common to set it in the space of the moments to obtain a diagonal collision matrix \cite{lallemand2000theory}.
            It reads
            \begin{align}
                \vectorial{\momentum}(t, \vectorial{x}) &= \operatorial{M} \vectorial{\distribution}(t, \vectorial{x}), \nonumber \\
                \vectorial{\distribution}^{\collided} (t, \vectorial{x}) &= \operatorial{M}^{-1} \adaptiveroundbrackets{(\operatorial{I} - \operatorial{S})  \vectorial{\momentum}(t, \vectorial{x}) + \operatorial{S} \vectorial{\momentum}^{\text{eq}} (\momentum^{0}(t, \vectorial{x}), \dots, \momentum^{\velocitynumber_{c}-1}(t, \vectorial{x}))},\label{eq:ReferenceCollision}
            \end{align}
            where the star $\star$ denotes any post-collisional state.
            Vectorial notations are used: $\vectorial{\distribution} = (\distribution^{\populationindex})_{\populationindex=0}^{\populationindex=q-1}$ for the particle distribution functions and $\vectorial{\momentum} = (\momentum^{\populationindex})_{\populationindex=0}^{\populationindex=q-1}$ for the moments vector.
            The matrix $\operatorial{M} \in \mathbb{R}^{\velocitynumber \times \velocitynumber}$ is the invertible matrix that gives the moments from the particle distribution functions. The collide operator then behaves as a diagonal relaxation on the moments vector associated with the singular matrix $\operatorial{S} \in \mathbb{R}^{\velocitynumber \times \velocitynumber}$. The rank of the matrix $\operatorial{S}$ is $\velocitynumber - \velocitynumber_c$, where $\velocitynumber_c$ is the number of conserved moments.
            The vector of the moments at equilibrium, function of the conserved moments\footnote{Those with relaxation parameter equal to zero.}, is indicated by $\vectorial{\momentum}^{\text{eq}}$.
            \item The stream phase is non-local but linear and corresponds to a shift of the data along the characteristics of each velocity field. It reads
            \begin{equation}\label{eq:referenceSchemeStream}
                \distribution^{\populationindex}(t + \timestep, \vectorial{x}) = \distribution^{\populationindex, \collided}(t, \vectorial{x} - \vectorial{\normalizedvelocityletter}_{\populationindex}\spacestep), \qquad \populationindex = 0, \dots, \velocitynumber - 1.
            \end{equation}
    \end{itemize}
    
    In order to recover the desired macroscopic equations on the conserved moments, one can act on the choice of relaxation matrix $\operatorial{S}$ and on the equilibria $\vectorial{\momentum}^{\text{eq}}$.

\section{Adaptive space discretization and adaptive lattice Boltzmann scheme}\label{sec:Geometry}

    The lattice Boltzmann methods introduced in the previous Section are defined on regular uniform lattices.
    However, in practical applications, especially when localized structures such as propagating fronts are present, a fine discretization of the whole spatial domain is extremely demanding in terms of memory and computational time. This calls for the adoption of a spatial discretization with different levels of refinement.
    In this Section, we summarize our adaptive approach presented in \cite{bellotti2021sisc, bellotti2021multidimensional} to transform a method acting on a uniform mesh at the finest available scale into another working on a spatially adaptive grid spanning many spatial scales.
    The adaptive grid is dynamically updated at each time step using a regularity evaluation through the multiresolution analysis, ensuring error control. Then, multiresolution is used to evaluate the lattice Boltzmann scheme as if we were on the finest scale, which is the one of the reference scheme presented in the previous section.
    Particular attention is paid to the way of devising the stream phase, because of its pivotal role in the forthcoming analysis.
    
    \subsection{Spatial discretization}

    \begin{figure}
       \begin{center}
            \includegraphics[width=0.5\textwidth]{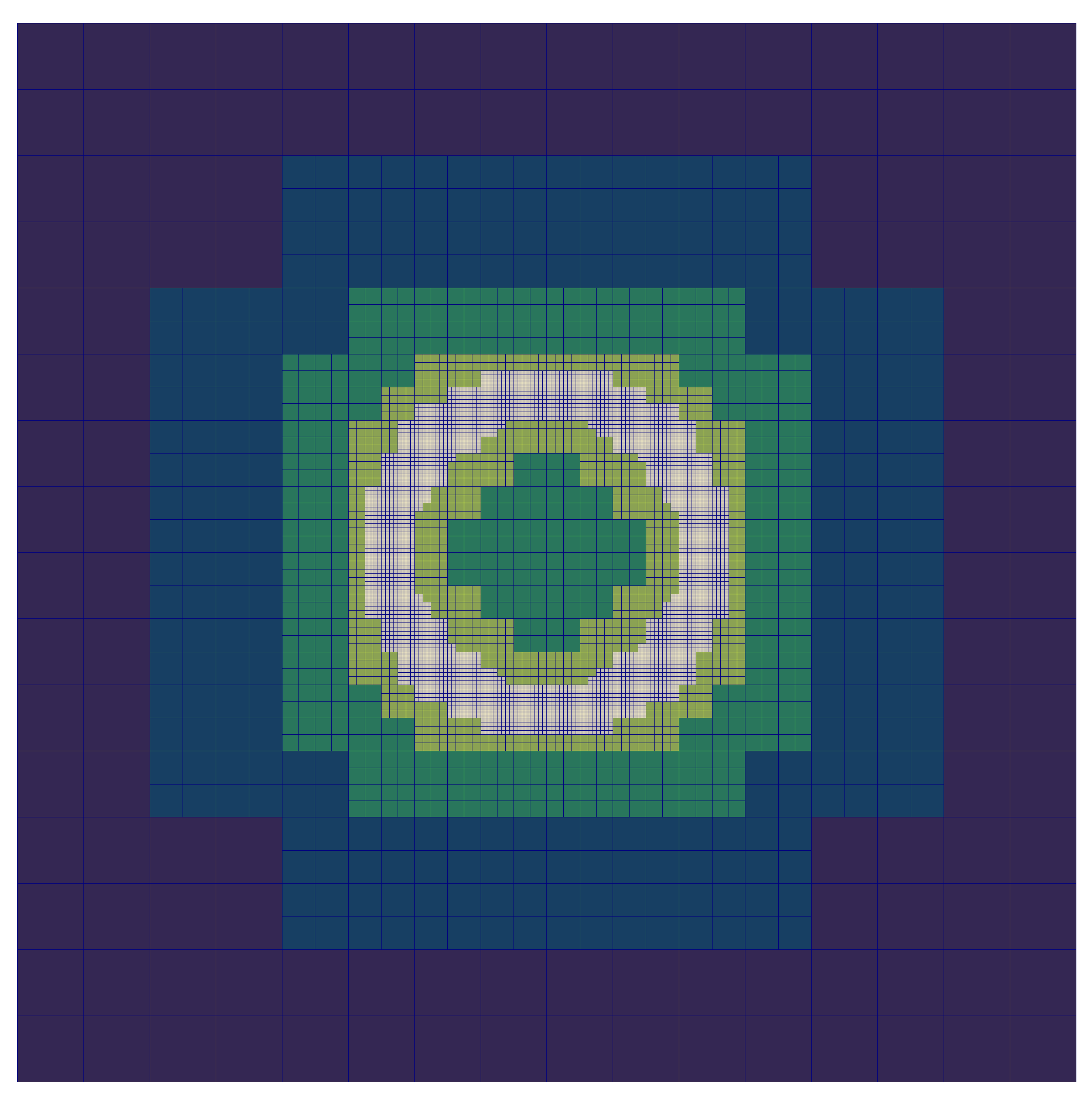}
        \end{center}\caption{\label{fig:hybrid_mesh}Example of hybrid mesh where different colors represent different levels of resolution spanning $\levelletter = \minlevel, \dots, \maxlevel$.}
    \end{figure} 
    
    The starting point to design the adaptive LBM-MR method is the choice of a spatial discretization, which is performed using dyadic Cartesian grids with possibly different levels of resolution.
    Consider, for the sake of a simpler presentation, a bounded domain $\Omega = [0, 1]^{\spatialdimensionality}$, typically for $\spatialdimensionality = 1, 2, 3$. Following our previous contributions \cite{bellotti2021sisc, bellotti2021multidimensional}, we can build a hybrid partition of such a domain formed by cells at different levels of resolution between $\minlevel$ and $\maxlevel$. An example is given in Figure \ref{fig:hybrid_mesh}. A cell of the level $\levelletter$, $\levelletter = \minlevel, \dots, \maxlevel$, is given by
    \begin{equation*}
        \cellletter_{\levelletter, \vectorial{\indexletter}} = \prod_{a = 1}^{\spatialdimensionality} [2^{-\levelletter} \indexletter_{a}, 2^{-\levelletter} (\indexletter_{a} + 1)],
        \qquad \vectorial{\indexletter} \in \{0, \dots, 2^{\leveldifference} - 1 \}^{\spatialdimensionality}.
    \end{equation*}
    % for $\levelletter = \minlevel, \dots, \maxlevel$ and $\vectorial{\indexletter} \in \{0, \dots, 2^{\leveldifference} - 1 \}^{\spatialdimensionality}$. 
    For the sake of notation, we have introduced $\leveldifference$ as the distance between the current level $\levelletter$ and the finest level $\maxlevel$ given by $\leveldifference = \maxlevel - \levelletter$.
    The center of the cell $\cellletter_{\levelletter, \vectorial{\indexletter}}$ is given by $\vectorial{x}_{\levelletter, \vectorial{\indexletter}} := 2^{-\levelletter} (\vectorial{\indexletter} + 1/2)$ and its Lebesgue measure is $(\spacestep_{\levelletter})^{\spatialdimensionality}$, where $\spacestep_{\levelletter} = 2^{\leveldifference} \spacestep$ with $\spacestep = 2^{-\maxlevel}$ that corresponds to the space-step of the finest grid.
    
    Remark that by \eqref{eq:TimeStepDefinition} the time-step is constant across all the levels and is dictated by the finest resolution allowed for the hybrid grid, which corresponds to the resolution of reference scheme on the uniform mesh.
    This approach complies with \cite{fakhari2014finite, fakhari2015numerics, fakhari2016mass} and with our previous works \cite{bellotti2021sisc, bellotti2021multidimensional}.
    % This is not the only possible choice to use lattice Boltzmann methods on non-uniform grids. Indeed, the other predominant viewpoint on the issue -- based on \cite{filippova1998grid} -- is to take one different time-step for each level of resolution.
    Other choices are possible to build lattice Boltzmann methods on non-uniform grids: the predominant one consists in taking a different time-step for each level of resolution \cite{filippova1998grid, dupuis2003theory, rohde2006}.
    The main drawbacks of this approach are that one generally needs to modify the scheme to be consistent with the target model and that time interpolations with storage of the solution at the previous time steps are necessary. 
    
    In this contribution, the way of constructing the hybrid partition of $\Omega$ at each time step is of little interest for the theoretical analysis we aim at performing. We shall come back on this point later.
    The interested reader may refer to our previous works \cite{bellotti2021sisc} and \cite{bellotti2021multidimensional}.
    Consequently, the meshes considered in this paper are mostly uniform but made up of cells at a smaller level of refinement  than the maximum allowed level $\maxlevel$, which still dictates the time-step of the method \eqref{eq:TimeStepDefinition}.
    We just mention that the mesh is constructed using multiresolution analysis, which allows, given a threshold parameter $0 < \epsilon \ll 1$, to control the quality of the solution with respect to the reference scheme on the uniform mesh at the finest level $\maxlevel$.
    We sometimes refer to the cells yielding this hybrid partition of the domain as ``leaves'' because they can be seen as such in a binary tree (for $\spatialdimensionality = 1$), a quadtree (for $\spatialdimensionality = 2$) or an octree (for $\spatialdimensionality = 3$).
    
    In the sequel, the quantities that can be interpreted as mean values on the corresponding cell $\cellletter_{\levelletter, \vectorial{\indexletter}}$ shall be denoted with a bar.
    The quantities without a bar are to be intended as point-valued quantities and are most of the time taken at the cell center $\vectorial{x}_{\levelletter, \vectorial{\indexletter}}$.

\subsection{The adaptive LBM-MR method}\label{sec:AdaptiveMethod}

    Based on the space/time discretizations introduced in the previous section, we can recap the basic ingredients used to build the adaptive LBM-MR method.
    These are the prediction operator and the reconstruction operator. The latter is generated by the recursive application of the former. 
    The special care devoted to the design of these operators is the key to achieve the error control. 
    Then, we adapt the two basic operations making up the lattice Boltzmann scheme, namely the collision and the stream phase. 
    The former is applied locally on each cell whereas the latter relies on the reconstruction of incoming and outgoing pseudo-fluxes only close to the edges of each cell, thus used a modest number of times. This yields a significant reduction of the number of performed operations on adapted grids. Moreover, we are able to ``break'' the recursivity of the reconstruction operator as illustrated in Section \ref{sec:ReconstructionFlattening}, allowing both for a cost reduction of the method and for the forthcoming theoretical study.

    \subsubsection{Prediction and reconstruction operators}\label{sec:PredictionandReconstruction}
    
    The prediction operator is an important ingredient both of adaptive multiresolution \cite{cohen2003fully} and of our adaptive LBM method \cite{bellotti2021multidimensional}.
    In many applications \cite{bihari1997multiresolution}, this operator for $\spatialdimensionality \geq 2$ is essentially built from the one for $\spatialdimensionality = 1$.
        For this reason, we present the case $\spatialdimensionality = 1$ and then we briefly sketch how the multidimensional extension is done.
        The predicted value -- indicated with a hat -- over a cell $\cellletter_{\levelletter + 1, 2\indexletter + \delta}$ is given by
        \begin{equation}\label{eq:PredictionOperator}
            \predicted{\average{\distribution}}{\populationindex}{\levelletter + 1, 2\indexletter + \delta} = \average{\distribution}_{\levelletter, \indexletter}^{\populationindex} + (-1)^{\delta} Q_{1}^{\predictionstencil}(\indexletter; \average{\vectorial{\distribution}}_{\levelletter}), \quad \text{with} \quad Q_{1}^{\predictionstencil}(\indexletter; \average{\vectorial{\distribution}}_{\levelletter}) = \sum_{\pi = 1}^{\predictionstencil} w_{\pi} \adaptiveroundbrackets{\average{\distribution}_{\levelletter, \indexletter + \pi}^{\populationindex} - \average{\distribution}_{\levelletter, \indexletter - \pi}^{\populationindex}},
        \end{equation}
        for $\delta  = 0, 1$, where the weights $(w_{\pi})_{\pi = 1}^{\pi = \predictionstencil}$ are given in \cite{bellotti2021sisc}. The parameter $\predictionstencil \in \mathbb{N}$ is the size of the prediction stencil. 
        The predicted value is a tentative one built using information on a coarser level of resolution.
        The weights are constructed as follows. 
        The local reconstruction polynomial of degree $2\predictionstencil$ is written in the canonical basis
        $$\pi^{\populationindex}_{\levelletter, \indexletter}(x) = \sum_{m = 0}^{m = 2\predictionstencil} A_{\levelletter, \indexletter}^{\populationindex, m} x^m,$$
        where the coefficients $(A_{\levelletter, \indexletter}^{\populationindex, m})_{m=0}^{m = 2\predictionstencil}$ are obtained by enforcing that the following linear constraints hold
        \begin{equation}\label{eq:systemReconstructionPolynomial}
            \frac{1}{\spacestep_{\levelletter}} \int_{x_{\levelletter, \indexletter} + \spacestep_{\levelletter}(\delta - 1/2)}^{x_{\levelletter, \indexletter} + \spacestep_{\levelletter}(\delta + 1/2)} \pi_{\levelletter, \indexletter}^{\populationindex}(x) \text{d}x = \average{\distribution}^{\populationindex}_{\levelletter, \indexletter + \delta}, \qquad \delta = -\predictionstencil, \dots, 0, \dots, \predictionstencil.
        \end{equation}
        This yields a linear system with invertible matrix $\operatorial{T} \in \mathbb{R}^{(2\predictionstencil+1)\times (2\predictionstencil+1)}$
        \begin{equation}\label{eq:SystemPrediction}
            \operatorial{T} (A_{\levelletter, \indexletter}^{\populationindex, m})_{m=0}^{m = 2\predictionstencil} = (\average{\distribution}_{\levelletter, \indexletter + \delta}^{\populationindex})_{\delta = -\predictionstencil}^{\delta = +\predictionstencil}.
        \end{equation}
        Once the coefficients $(A_{\levelletter, \indexletter}^{\populationindex, m})_{m=0}^{m = 2\predictionstencil}$ are known, the prediction operator and the corresponding weights are obtained by averaging the reconstruction polynomial
        \begin{equation*}
            \predicted{\average{\distribution}}{\populationindex}{\levelletter + 1, 2\indexletter + \delta} = \frac{1}{\spacestep_{\levelletter + 1}} \int_{x_{\levelletter, \indexletter} + \spacestep_{\levelletter+1}(\delta - 1)}^{x_{\levelletter, \indexletter} + \spacestep_{\levelletter+1}\delta} \pi^{\populationindex}_{\levelletter, \indexletter}(x) \text{d}x, \qquad \delta = 0, 1.
        \end{equation*}
        
        In this work, we emphasize the difference between the case $\predictionstencil = 0$, called Haar case\footnote{Because one can show that this approach can be linked with the theory of the Haar wavelet.} or direct evaluation, and the case $\predictionstencil = 1$ (with $w_1 = -1/8$).
        
        \begin{remark}[Polynomial exactness]\label{rem:polynomialExactness}
            The prediction operator of stencil width $\predictionstencil$ has been built in order to exactly recover the average on the cell $\cellletter_{\levelletter + 1, 2\indexletter + \delta}$ when the function $\distribution^{\populationindex}$ is polynomial of degree at most $2\predictionstencil + 1$.
            However, as we shall see, it would be dangerous to conclude on the effect of the prediction operator on the equivalent equations just using this statement.
        \end{remark}
        
        The multidimensional generalization is carried out following the path of Bihari and Harten \cite{bihari1997multiresolution} in a tensor product fashion. It is useful to present it because it provides us with a formalism to speed up proofs in the multidimensional framework.
        For $\spatialdimensionality = 2$, one considers a reconstruction polynomial $\pi_{\levelletter, \vectorial{\indexletter}}^{\populationindex}(\vectorial{x})$ made up only of terms belonging to the product of two uni-dimensional polynomials  $\pi_{\levelletter, \indexletter_x}^{\populationindex}(x)$ and $\pi_{\levelletter, \indexletter_y}^{\populationindex}(y)$, one for each Cartesian direction. 
        This yields
        \begin{equation*}
            \pi_{\levelletter, \vectorial{\indexletter}}^{\populationindex}(\vectorial{x}) = \pi_{\levelletter, \vectorial{\indexletter}}^{\populationindex}(x, y) = \sum_{m = 0}^{2\predictionstencil}\sum_{n = 0}^{2\predictionstencil} A_{\levelletter, \vectorial{\indexletter}}^{\populationindex, m, n} x^m y^n.
        \end{equation*}
        It can be easily shown that, thanks to the tensor-product structure, \eqref{eq:SystemPrediction} becomes
        \begin{equation*}
           (\operatorial{T} \otimes \operatorial{T}) (A_{\levelletter, \indexletter}^{\populationindex, m, n})_{m, n=0}^{m, n = 2\predictionstencil} = (\average{\distribution}_{\levelletter, \vectorial{\indexletter} + \vectorial{\delta}}^{\populationindex})_{\vectorial{\delta} \in  \{ -\predictionstencil, \dots, +\predictionstencil \}^{2}},
        \end{equation*}
        where $\otimes$ is the Kronecker product of matrices.
        A useful property of the Kronecker product is that $(\operatorial{T} \otimes \operatorial{T})^{-1} = \operatorial{T}^{-1} \otimes \operatorial{T} ^{-1}$, therefore we can inverse the tensor product by inverting each of its terms.
        
        %\subsubsection{Reconstruction operator}
        Starting from the prediction operator, we can construct the reconstruction operator, denoted by the double hat $\reconstructed{~~~~}{ }{ }$.
        This is used to recover information at the finest level $\maxlevel$ using a recursive application of the prediction operator until reaching information stored on the considered level $\levelletter$.

    \subsubsection{Collide-and-stream}

        Once we have a reconstruction operator generated by a prediction operator for some given stencil $\predictionstencil \in \mathbb{N}$, we are ready to devise a LBM-MR method as shown in \cite{bellotti2021sisc, bellotti2021multidimensional}.
        Like the standard lattice Boltzmann method, it is made up of two different steps:

        \begin{itemize}
            \item The collision phase is performed on each cell $\cellletter_{\levelletter, \vectorial{\indexletter}}$, $\levelletter = \minlevel, \dots, \maxlevel$, $\vectorial{\indexletter}\in\lbrace 0,\ldots, 2^{\leveldifference} - 1 \rbrace^{\spatialdimensionality}$, of the hybrid mesh at time $t$ and is completely local.
            It reads
            \begin{align}
                \average{\vectorial{\momentum}}_{\levelletter, \vectorial{\indexletter}}(t) &= \operatorial{M} \average{\vectorial{\distribution}}_{\levelletter, \vectorial{\indexletter}}(t), \nonumber \\
                \average{\vectorial{\distribution}}_{\levelletter, \vectorial{\indexletter}}^{\collided} (t) &= \operatorial{M}^{-1} \adaptiveroundbrackets{(\operatorial{I} - \operatorial{S})  \average{\vectorial{\momentum}}_{\levelletter, \vectorial{\indexletter}}(t) + \operatorial{S} \vectorial{\momentum}^{\text{eq}} (\average{\momentum}_{\levelletter, \vectorial{\indexletter}}^{0}(t), \dots, \average{\momentum}_{\levelletter, \vectorial{\indexletter}}^{\velocitynumber_c - 1}(t))}. \label{eq:CollisionPhase}
            \end{align}
            The equilibria and the relaxation matrix are the same than for the reference method, see \eqref{eq:ReferenceCollision}.
            \item The stream phase at velocity $\vectorial{\velocityletter}_{\populationindex}$, for $\populationindex = 0, \dots, \velocitynumber - 1$, is linear but non-local.
            It reads
            \begin{equation}\label{eq:StreamPhase}
                \average{\distribution}^{\populationindex}_{\levelletter, \vectorial{\indexletter}} (t+\timestep) = \average{\distribution}^{\populationindex, \collided}_{\levelletter, \vectorial{\indexletter}} (t) + \frac{1}{2^{d\leveldifference}} \adaptiveroundbrackets{\sum_{\overline{\vectorial{\indexletter}} \in \mathcal{E}_{\levelletter, \vectorial{\indexletter}}^{\populationindex}} \reconstructed{\average{\distribution}}{\populationindex, \collided}{\maxlevel, \overline{\vectorial{\indexletter}}} (t) - \sum_{\overline{\vectorial{\indexletter}} \in \mathcal{A}_{\levelletter, \vectorial{\indexletter}}^{\populationindex}} \reconstructed{\average{\distribution}}{\populationindex, \collided}{\maxlevel, \overline{\vectorial{\indexletter}}} (t)},
            \end{equation}
            where we have taken
            \begin{gather*}
                \mathcal{B}_{\levelletter, \vectorial{\indexletter}} = \{\vectorial{\indexletter}2^{\leveldifference} + \vectorial{\delta} ~ : ~ \vectorial{\delta} \in \{0, \dots, 2^{\leveldifference} - 1 \}^{\spatialdimensionality}\}, \\
                \mathcal{E}_{\levelletter, \vectorial{\indexletter}}^{\populationindex} = (\mathcal{B}_{\levelletter, \vectorial{\indexletter}} - \vectorial{\normalizedvelocityletter}_{\populationindex}) \smallsetminus \mathcal{B}_{\levelletter, \vectorial{\indexletter}}, \qquad \mathcal{A}_{\levelletter, \vectorial{\indexletter}}^{\populationindex} =  \mathcal{B}_{\levelletter, \vectorial{\indexletter}} \smallsetminus (\mathcal{B}_{\levelletter, \vectorial{\indexletter}} - \vectorial{\normalizedvelocityletter}_{\populationindex}).
            \end{gather*}
            The idea behind the formula is to build, \emph{via} the reconstruction operator $\reconstructed{~~}{}{}$, the piece-wise constant solution on the uniform mesh at the finest level $\maxlevel$. This reconstruction is exactly advected by the constant velocity field $\vectorial{\velocityletter}_{\populationindex}$ and then averaged on the cells of the hybrid mesh.
            This procedure is really close to the so-called Corner Transport Upwind scheme \cite{colella1990multidimensional}.
            Cells indexed by elements of $\mathcal{E}_{\levelletter, \vectorial{\indexletter}}^{\populationindex}$ can be seen to yield an incoming pseudo-flux into the current cell $\cellletter_{\levelletter, \vectorial{\indexletter}}$, whereas those linked with $\mathcal{A}_{\levelletter, \vectorial{\indexletter}}^{\populationindex}$ give an outgoing pseudo-flux. An example is given in Figure \ref{fig:setsstream}.
            
                    \begin{figure}[h]
        \begin{center}
            \def\svgwidth{0.4\textwidth}
            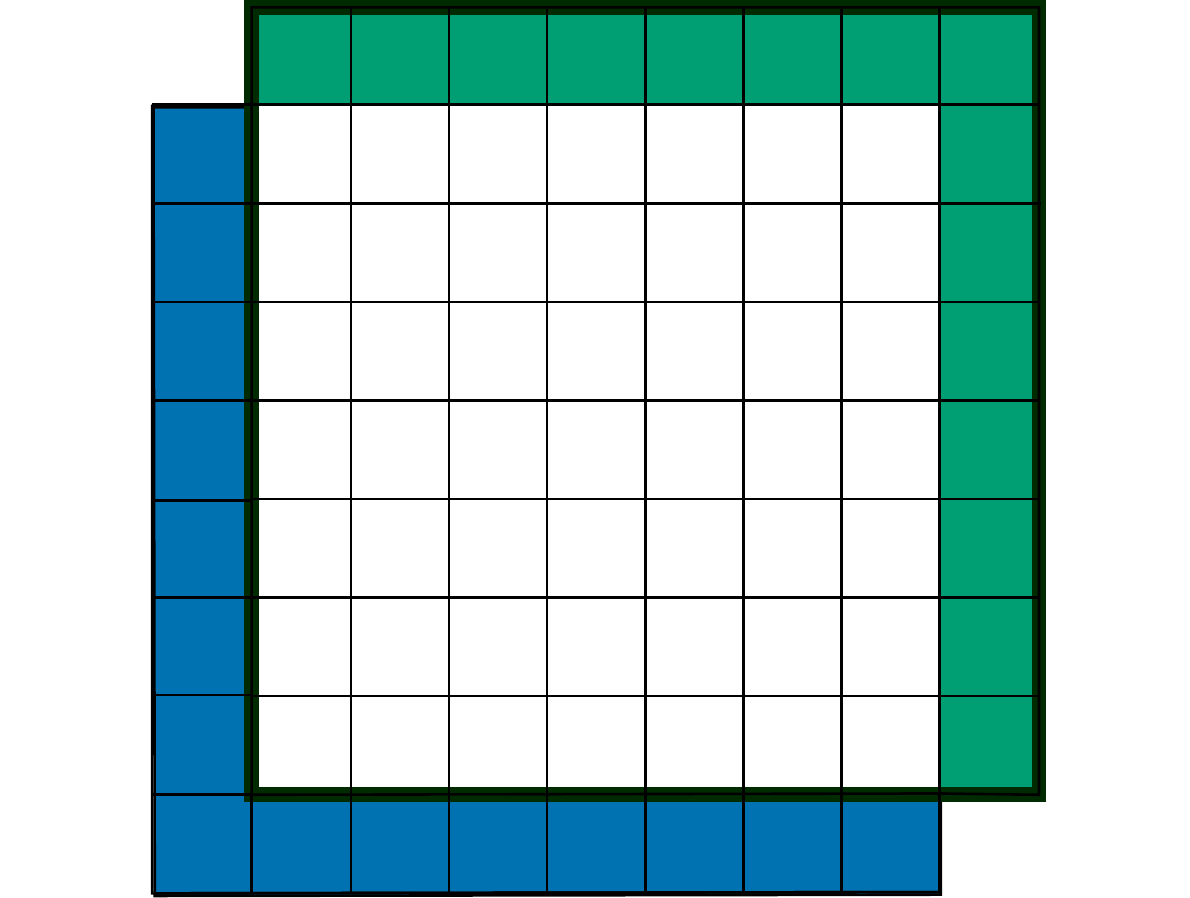
        \end{center}
        \caption{\label{fig:setsstream}Example of sets involved in the stream phase for $\spatialdimensionality = 2$, velocity $\vectorial{\normalizedvelocityletter}_{\populationindex} = (1, 1)$ and $\leveldifference = 3$. The cells yielding an incoming pseudo-flux are colored in blue whereas those giving an out-going pseudo-flux are highlighted in green. The perimeter of the cell $\cellletter_{\levelletter, \vectorial{\indexletter}}$ is traced with a thicker black line.}
        \end{figure}
        \end{itemize}
        
    \subsection{Mesh adaptation}
    
        Though the mesh adaptation algorithm is not the focus of this paper (see \cite{bellotti2021sisc, bellotti2021multidimensional} for the details), we briefly describe it. 
        %We shall provide a small demonstration on the pertinence of such a procedure in Section \ref{sec:ViscousBurgersEquation}.
        At each time-step, a cell of the hybrid mesh can be eliminated from the structure if its associated detail, namely the difference between the actual stored average and the predicted average \emph{via} \eqref{eq:PredictionOperator}, is smaller than a given level-dependent threshold.
        Conversely, if the detail on the considered cell is large, we refine it in order to ensure that abrupt changes in the solution are correctly followed by the method.
        Then, the adaptation strategy reads
        \begin{equation}\label{eq:DetailInequality}
        \begin{aligned}
            \text{Coarsen} \quad \cellletter_{\levelletter, \vectorial{\indexletter}} \quad \text{if} \quad \max_{\populationindex} \adaptiveroundbrackets{|\predicted{\average{\distribution}}{\populationindex}{\levelletter, \vectorial{\indexletter}} - \average{\distribution}^{\populationindex}_{\levelletter, \vectorial{\indexletter}}|} & \leq \epsilon \, 2^{-\spatialdimensionality \leveldifference}, \\
            \text{Refine} \quad \cellletter_{\levelletter, \vectorial{\indexletter}} \quad \text{if} \quad \max_{\populationindex} \adaptiveroundbrackets{|\predicted{\average{\distribution}}{\populationindex}{\levelletter, \vectorial{\indexletter}} - \average{\distribution}^{\populationindex}_{\levelletter, \vectorial{\indexletter}}|} &\geq \epsilon \, 2^{-\spatialdimensionality (\leveldifference-1) + \overline{\mu}}, 
        \end{aligned}
        \end{equation} 
        where $\overline{\mu} \geq 0$ is a free parameter to be choosen linked to the expected regularity of the solution (see~\cite{bellotti2021sisc}).
        The process is repeated until no modification of the mesh are made.
        Indeed, the example of hybrid grid given in Figure \ref{fig:hybrid_mesh} has been built using this very algorithm.
        
        The interest of having used a rather expensive and complex operator $\reconstructed{~~~~}{ }{ }$ to build the adaptive stream phase \eqref{eq:StreamPhase} is that it ensures that if the grid is adapted at each time step using \eqref{eq:DetailInequality} we are able to control the additional error introduced by the adaptive strategy with $\epsilon$ for any lattice Boltzmann scheme and without having to modify it.
        To the best of our knowledge, this is unprecedented in the literature and cannot be achieved with other mesh adaptation strategies such as the cell-based AMR \cite{rohde2006, eitel2013, fakhari2014finite, fakhari2015numerics, fakhari2016mass}, which relies on strongly problem-dependent heuristic criteria.
        
        \begin{remark}[On the use of uniform grids in the paper]\label{rem:UniformMesh}
            The analysis that we shall develop in Section \ref{sec:Equivalent1D} and \ref{sec:Equivalent2D}
            and most of the numerical tests of Section \ref{sec:NumericalSimulations} are conducted, as previously stressed, on uniform meshes at some level of refinement $\levelletter = \minlevel, \dots, \maxlevel$.
            Still, this analysis is also relevant in the case of an adapted mesh for the following reasons:
            \begin{itemize}
             \item Given a leaf $\cellletter_{\levelletter, \vectorial{\indexletter}}$ at some level $\levelletter$, it is surrounded by enough cells (both leaves or halo cells) at the same level of refinement. Therefore the mesh can be considered to be locally uniform \cite{cohen2003fully}, which perfectly fits the local character of the analysis.
             \item The theoretical analysis is based on the assumption that the distributions $\distribution^{\populationindex}$ are smooth on the whole domain at every considered time.
             Thus, once one fixes a small but finite tolerance $\epsilon$ and the authorized level range $\maxlevel - \minlevel \geq 0$, letting $\maxlevel$ increase (imagine $\maxlevel \to +\infty$, thus $\spacestep \to 0$), even the uniform mesh at level $\minlevel$ will allow to control errors by $\epsilon$ \emph{via} multiresolution.
            \end{itemize}
        \end{remark}

        Before switching to the analysis through the theory of equivalent equations, let us come back to several interesting variants for treating both collision and stream phases.
        
    \section{Alternative treatments of the collision and the stream phases}
    
    Clearly, the LBM-MR scheme presented in the previous section is not the only possible strategy to adapt the reference numerical method.
    Therefore, the aim of this section is to present and discuss some alternative ideas, which will also be analyzed in the following in order to shed some light on how the various strategies behave.
        
    \subsection{Enhanced collision treatments}\label{sec:DifferentCollisionStrategies}
        
        In the unidimensional case \cite{bellotti2021sisc}, we studied the effect of considering the collision operator only on the leaves of the adaptive mesh \eqref{eq:CollisionPhase} in terms of multiresolution and the possibility of recovering a control on the additional error.
        When the equilibria are linear functions of the conserved moments, this does not have any effect on the outcome of the method. In the non-linear case, the control on the additional error is still possible except in marginal pathological cases.
        %However, the choice of how performing the collision phase can have an important role and thus it is interesting to study it.
        Even if the scope of the present work is not to fully discuss the different approaches to build the collision operator, we present two strategies that are eventually tested to build more reliable -- though costly -- collision phases.
        In the sequel, we shall call LBM-MR-LC scheme the one where the collision is done locally at the level of the leaves \eqref{eq:CollisionPhase}, to distinguish it from the other approaches.
    
        \subsubsection{Reconstructed collision}
            The first alternative approach we present, introduced in our previous work \cite{bellotti2021sisc}, is called ``reconstructed collision''. The resulting scheme is called LBM-MR-RC.
            The collision phase is done on each cell $\cellletter_{\levelletter, \vectorial{\indexletter}}$ of the hybrid mesh at time $t$ reconstructing the information needed in the (possibly) non-linear equilibria.
            This is (the change of basis is understood)
            \begin{align}
                %\average{\vectorial{\momentum}}_{\levelletter, \vectorial{\indexletter}}(t) &= \operatorial{M} \average{\vectorial{\distribution}}_{\levelletter, \vectorial{\indexletter}}(t), \nonumber \\
                \average{\vectorial{\distribution}}_{\levelletter, \vectorial{\indexletter}}^{\collided} (t) &= \operatorial{M}^{-1} 
                % \adaptiveroundbrackets{
                    \biggl( (\operatorial{I} - \operatorial{S})  \average{\vectorial{\momentum}}_{\levelletter, \vectorial{\indexletter}}(t) + \frac{1}{2^{\spatialdimensionality \leveldifference}} \operatorial{S} \sum_{\overline{\vectorial{\indexletter}} \in \mathcal{B}_{\levelletter, \vectorial{\indexletter}}}  \vectorial{\momentum}^{\text{eq}} (\reconstructed{\average{\momentum}}{0}{\maxlevel, \overline{\vectorial{\indexletter}}}(t), \dots, \reconstructed{\average{\momentum}}{\velocitynumber_c - 1}{\maxlevel, \overline{\vectorial{\indexletter}}}(t))
                    \biggr)
                    % }
                    . \label{eq:CollisionPhaseReconstructed}
            \end{align}
            The idea behind is to make the following approximation
            \begin{align*}
                \frac{1}{|\cellletter_{\levelletter, \vectorial{\indexletter}}|} \int_{\cellletter_{\levelletter, \vectorial{\indexletter}}}  \vectorial{\momentum}^{\text{eq}} (\momentum^{0}(t, \vectorial{x}), \dots, \momentum^{\velocitynumber_c-1}(t, \vectorial{x})) \text{d}\vectorial{x} &\simeq \frac{1}{|\cellletter_{\levelletter, \vectorial{\indexletter}}|} \sum_{\overline{\vectorial{\indexletter}} \in \mathcal{B}_{\levelletter, \vectorial{\indexletter}}} |\cellletter_{\maxlevel, \overline{\vectorial{\indexletter}}}| \vectorial{\momentum}^{\text{eq}} (\reconstructed{\average{\momentum}}{0}{\maxlevel, \overline{\vectorial{\indexletter}}}(t), \dots, \reconstructed{\average{\momentum}}{\velocitynumber_c - 1}{\maxlevel, \overline{\vectorial{\indexletter}}}(t)), \\
                &= \frac{1}{2^{\spatialdimensionality \leveldifference}} \sum_{\overline{\vectorial{\indexletter}} \in \mathcal{B}_{\levelletter, \vectorial{\indexletter}}} \vectorial{\momentum}^{\text{eq}} (\reconstructed{\average{\momentum}}{0}{\maxlevel, \overline{\vectorial{\indexletter}}}(t), \dots, \reconstructed{\average{\momentum}}{\velocitynumber_c - 1}{\maxlevel, \overline{\vectorial{\indexletter}}}(t)).
            \end{align*}
            It can be shown that this way of proceeding \eqref{eq:CollisionPhaseReconstructed} coincides with \eqref{eq:CollisionPhase} whenever the moments at the equilibrium are linear functions, because the information added by the use of the reconstruction operator disappears once we project back onto the current level of resolution $\levelletter$.
            However, this is false for non-linear equilibria.
            The notable disadvantage of this approach is that the reconstruction is needed on every cell belonging to $\mathcal{B}_{\levelletter, \vectorial{\indexletter}}$ and not only on those close to the cell edge as for the stream phase \eqref{eq:StreamPhase}.
            This is what the next strategy tries to alleviate.

        \subsubsection{Predict-and-quadrate}
            
            This approach closely follows that of \cite{hovhannisyan2010} for treating the forcing terms in Finite Volume schemes and shall be called LBM-MR-PQC.
            Reminding us of Section \ref{sec:PredictionandReconstruction}, consider the polynomials interpolating the $\populationindex$-th momentum around the cell $\cellletter_{\levelletter, \vectorial{\indexletter}}$ at time $t$ given by $\pi_{\levelletter, \vectorial{\indexletter}}^{\populationindex}(t, \vectorial{x}) = \operatorial{M} \vectorial{\pi}_{\levelletter, \vectorial{\indexletter}}(t, \vectorial{x})|_{\populationindex}$.
            Having this continuous approximation of the solution, one performs the following steps
            \begin{align*}
                \frac{1}{|\cellletter_{\levelletter, \vectorial{\indexletter}}|} \int_{\cellletter_{\levelletter, \vectorial{\indexletter}}}  \vectorial{\momentum}^{\text{eq}} (\momentum^{0}(t, \vectorial{x}), \dots, \momentum^{\velocitynumber_c-1}(t, \vectorial{x})) \text{d}\vectorial{x} &\simeq \frac{1}{|\cellletter_{\levelletter, \vectorial{\indexletter}}|} \int_{\cellletter_{\levelletter, \vectorial{\indexletter}}}  \vectorial{\momentum}^{\text{eq}} (\pi_{\levelletter, \vectorial{\indexletter}}^{0}(t, \vectorial{x}), \dots, \pi_{\levelletter, \vectorial{\indexletter}}^{\velocitynumber_c-1}(t, \vectorial{x})) \text{d}\vectorial{x}, \\
                &\simeq \frac{1}{|\cellletter_{\levelletter, \vectorial{\indexletter}}|} \sum_{i=1}^N \tilde{w}_i \vectorial{\momentum}^{\text{eq}} (\pi_{\levelletter, \vectorial{\indexletter}}^{\populationindex}(t, \tilde{\vectorial{x}}_i), \dots, \pi_{\levelletter, \vectorial{\indexletter}}^{\velocitynumber_c-1}(t, \vectorial{\tilde{x}_i})),
            \end{align*}
            where one employs a quadrature formula with $N \in \mathbb{N}^{\star}$ weights $(\tilde{w}_i)_{i = 1}^{i=N}$ and quadrature points $(\tilde{\vectorial{x}}_i)_{i = 1}^{i=N}$. 
            Therefore, the collision phase reads
            \begin{align}
                %\average{\vectorial{\momentum}}_{\levelletter, \vectorial{\indexletter}}(t) &= \operatorial{M} \average{\vectorial{\distribution}}_{\levelletter, \vectorial{\indexletter}}(t), \nonumber \\
                \average{\vectorial{\distribution}}_{\levelletter, \vectorial{\indexletter}}^{\collided} (t) &= \operatorial{M}^{-1} 
                % \adaptiveroundbrackets{
                    \biggl( (\operatorial{I} - \operatorial{S})  \average{\vectorial{\momentum}}_{\levelletter, \vectorial{\indexletter}}(t) + \frac{1}{|\cellletter_{\levelletter, \vectorial{\indexletter}}|} \operatorial{S} \sum_{i=1}^N \tilde{w}_i \vectorial{\momentum}^{\text{eq}} (\pi_{\levelletter, \vectorial{\indexletter}}^{0}(t, \tilde{\vectorial{x}}_i), \dots, \pi_{\levelletter, \vectorial{\indexletter}}^{\velocitynumber_c-1}(t, \tilde{\vectorial{x}}_i))
                    % }
                    \biggr). \label{eq:CollisionPhaseQuadrature}
            \end{align}
            
            This procedure is certainly less computationally expensive than that of the LBM-MR-RC scheme \eqref{eq:CollisionPhaseReconstructed} because one utilizes only one stage of prediction and the number $N$ is generally not so large.
            It relies on the fact that the solution is expected to locally behave like a low degree polynomial, which is transformed by the equilibria into another non-linear function and hoping that the quadrature formula is accurate enough to approximate the integral over the considered cell.

    \subsection{Lax-Wendroff stream phase}
    
        In our multiresolution framework, constructing the adaptive stream phase by using the reconstruction operator was a compulsory choice since we wish to recover a control on the numerical error of the adaptive strategy.
        However, for methods based on the heuristic AMR, simpler and more physically funded approaches are possible.
        Indeed, before going on with the analysis of our adaptive method, we present, for comparison purposes, the so-called Lax-Wendroff method proposed by \cite{fakhari2014finite, fakhari2015numerics, fakhari2016mass} for a D2Q9 scheme to be used on adaptive grid constructed \emph{via} an AMR algorithm. We may imagine that in our framework the mesh adaptation could be done using multiresolution as presented in the previous section.
        The collision phase is the same than \eqref{eq:CollisionPhase}. 
        What changes is the stream phase coming under the form
        \begin{equation}\label{eq:LaxWendroff}
            \average{\distribution}^{\populationindex}_{\levelletter, \vectorial{\indexletter}}(t + \timestep) = \adaptiveroundbrackets{1 - \frac{1}{4^{\leveldifference}}}\hspace{-0.5ex} \average{\distribution}^{\populationindex, \collided}_{\levelletter, \vectorial{\indexletter}}(t) + \frac{1}{2^{\leveldifference + 1}} \adaptiveroundbrackets{1 + \frac{1}{2^{\leveldifference}}} \hspace{-0.5ex} \average{\distribution}^{\populationindex, \collided}_{\levelletter, \vectorial{\indexletter} - \vectorial{\normalizedvelocityletter}_{\populationindex}/|\vectorial{\normalizedvelocityletter}_{\populationindex}|_2}(t) - \frac{1}{2^{\leveldifference + 1}} \adaptiveroundbrackets{1 - \frac{1}{2^{\leveldifference}}} \hspace{-0.5ex} \average{\distribution}^{\populationindex, \collided}_{\levelletter, \vectorial{\indexletter} + \vectorial{\normalizedvelocityletter}_{\populationindex}/|\vectorial{\normalizedvelocityletter}_{\populationindex}|_2}(t).
        \end{equation}
        
        At least for $\spatialdimensionality = 1$, we can show that this scheme has some link with the prediction operator \eqref{eq:PredictionOperator} but is not a multiresolution scheme, namely that it does not involve the reconstruction operator. This is the meaning of the following.
        
        \begin{proposition}\label{prop:LWisNotMR}
            Let $\spatialdimensionality = 1$. The Lax-Wendroff scheme given by \eqref{eq:LaxWendroff} is obtained using the local reconstruction polynomial with coefficients given by \eqref{eq:systemReconstructionPolynomial} for $\predictionstencil = 1$ around the considered cell.
        For this reason, since the reconstruction polynomial is not used in a recursive manner, it is not a multiresolution scheme.

    \end{proposition}
        \begin{proof}
            Set $\predictionstencil = 1$. The time variable and the collision are understood in this proof not to overcharge the notations. After some computations, we obtain that the reconstruction polynomial comes under the form
            \begin{equation*}
                \pi_{\levelletter, \indexletter}^{\populationindex}(x) = \tilde{A}_{\levelletter, \indexletter}^{\populationindex, 2} \adaptiveroundbrackets{\frac{x - x_{\levelletter, \indexletter}}{\spacestep_{\levelletter}}}^2 + \tilde{A}_{\levelletter, \indexletter}^{\populationindex, 1} \adaptiveroundbrackets{\frac{x - x_{\levelletter, \indexletter}}{\spacestep_{\levelletter}}} + \tilde{A}_{\levelletter, \indexletter}^{\populationindex, 0},
            \end{equation*}
            with
            \begin{equation*}\left\lbrace
                \begin{aligned}
                    \tilde{A}_{\levelletter, \indexletter}^{\populationindex, 2} &= \tfrac{1}{2} \average{\distribution}_{\levelletter, \indexletter - 1}^{\populationindex} - \average{\distribution}_{\levelletter, \indexletter}^{\populationindex} + \tfrac{1}{2} \average{\distribution}_{\levelletter, \indexletter + 1}^{\populationindex}, \\
                    \tilde{A}_{\levelletter, \indexletter}^{\populationindex, 1} &= -\tfrac{1}{2} \average{\distribution}_{\levelletter, \indexletter - 1}^{\populationindex} + \tfrac{1}{2} \average{\distribution}_{\levelletter, \indexletter + 1}^{\populationindex}, \\
                    \tilde{A}_{\levelletter, \indexletter}^{\populationindex, 0} &= -\tfrac{1}{24} \average{\distribution}_{\levelletter, \indexletter - 1}^{\populationindex} + \tfrac{13}{12} \average{\distribution}_{\levelletter, \indexletter}^{\populationindex} - \tfrac{1}{24} \average{\distribution}_{\levelletter, \indexletter + 1}^{\populationindex}.
                \end{aligned}
                \right.
            \end{equation*}
            According to \eqref{eq:StreamPhase} but without employing the multiresolution reconstruction, we approximate
            \begin{equation*}
                \reconstructed{\average{\distribution}}{\populationindex}{\maxlevel, \overline{{\indexletter}}} \simeq \frac{1}{\spacestep} \int_{C_{\maxlevel, \overline{\indexletter}}} \pi_{\levelletter, \indexletter}^{\populationindex}(x)\text{d}x, \qquad \overline{\indexletter} \in \mathcal{E}_{\levelletter, {\indexletter}}^{\populationindex} \cup \mathcal{A}_{\levelletter, {\indexletter}}^{\populationindex}.
            \end{equation*}
            Now let $\overline{\indexletter} \in \mathcal{E}_{\levelletter, {\indexletter}}^{\populationindex} \cup \mathcal{A}_{\levelletter, {\indexletter}}^{\populationindex}$, this gives
            \begin{align*}
                \frac{1}{\spacestep} \int_{C_{\maxlevel, \overline{\indexletter}}} \pi_{\levelletter, \indexletter}^{\populationindex}(x)\text{d}x &= \frac{1}{\spacestep} \int_{x_{\maxlevel, \overline{\indexletter}} - \spacestep/2}^{x_{\maxlevel, \overline{\indexletter}} + \spacestep/2} \pi_{\levelletter, \indexletter}^{\populationindex}(x)\text{d}x \\
                &= 2^{\leveldifference} \int_{\frac{x_{\maxlevel, \overline{\indexletter}} - x_{\levelletter, \indexletter}}{\spacestep_{\levelletter}} - \frac{1}{2^{\leveldifference + 1}}}^{\frac{x_{\maxlevel, \overline{\indexletter}} - x_{\levelletter, \indexletter}}{\spacestep_{\levelletter}} + \frac{1}{2^{\leveldifference + 1}}} (\tilde{A}_{\levelletter, \indexletter}^{\populationindex, 2}\xi^2 + \tilde{A}_{\levelletter, \indexletter}^{\populationindex, 1} \xi + \tilde{A}_{\levelletter, \indexletter}^{\populationindex, 0}) \text{d}\xi, \\
                &= 2^{\leveldifference}  \left [ \tilde{A}_{\levelletter, \indexletter}^{\populationindex, 2}\frac{\xi^3}{3} + \tilde{A}_{\levelletter, \indexletter}^{\populationindex, 1} \frac{\xi^2}{2} + \tilde{A}_{\levelletter, \indexletter}^{\populationindex, 0} \xi \right ]_{\xi = \frac{x_{\maxlevel, \overline{\indexletter}} - x_{\levelletter, \indexletter}}{\spacestep_{\levelletter}} - \frac{1}{2^{\leveldifference + 1}}}^{\xi = \frac{x_{\maxlevel, \overline{\indexletter}} - x_{\levelletter, \indexletter}}{\spacestep_{\levelletter}} + \frac{1}{2^{\leveldifference + 1}}}, \\
                &= 2^{\leveldifference}  \tilde{A}_{\levelletter, \indexletter}^{\populationindex, 2} \left [ \adaptiveroundbrackets{\frac{x_{\maxlevel, \overline{\indexletter}} - x_{\levelletter, \indexletter}}{\spacestep_{\levelletter}}}^2 \frac{1}{2^{\leveldifference}} + \frac{1}{12}  \frac{1}{2^{3\leveldifference}} \right ] + \tilde{A}_{\levelletter, \indexletter}^{\populationindex, 1} \adaptiveroundbrackets{\frac{x_{\maxlevel, \overline{\indexletter}} - x_{\levelletter, \indexletter}}{\spacestep_{\levelletter}}}  + \tilde{A}_{\levelletter, \indexletter}^{\populationindex, 0}.
            \end{align*}
            Let us consider $\normalizedvelocityletter_{\populationindex} = 1$ to illustrate. In this case
            \begin{align*}
                \overline{k} \in \mathcal{E}_{\levelletter, {\indexletter}}^{\populationindex}, \qquad \frac{x_{\maxlevel, \overline{\indexletter}} - x_{\levelletter, \indexletter}}{\spacestep_{\levelletter}} &= - \frac{\spacestep_{\levelletter} + \spacestep}{2 \spacestep_{\levelletter}} = -\frac{1}{2}\adaptiveroundbrackets{1+\frac{1}{2^{\leveldifference}}}, \\
                \overline{k} \in \mathcal{A}_{\levelletter, {\indexletter}}^{\populationindex}, \qquad \frac{x_{\maxlevel, \overline{\indexletter}} - x_{\levelletter, \indexletter}}{\spacestep_{\levelletter}} &= \frac{\spacestep_{\levelletter} - \spacestep}{2 \spacestep_{\levelletter}} = \frac{1}{2}\adaptiveroundbrackets{1-\frac{1}{2^{\leveldifference}}}, 
            \end{align*}
            therefore 
            \begin{multline*}
                \sum_{\overline{\vectorial{\indexletter}} \in \mathcal{E}_{\levelletter, \vectorial{\indexletter}}^{\populationindex}} \reconstructed{\average{\distribution}}{\populationindex, \collided}{\maxlevel, \overline{\vectorial{\indexletter}}} - \sum_{\overline{\vectorial{\indexletter}} \in \mathcal{A}_{\levelletter, \vectorial{\indexletter}}^{\populationindex}} \reconstructed{\average{\distribution}}{\populationindex, \collided}{\maxlevel, \overline{\vectorial{\indexletter}}} \\
                = \tilde{A}_{\levelletter, \indexletter}^{\populationindex, 2} \left \{ \left [ -\frac{1}{2}\adaptiveroundbrackets{1+\frac{1}{2^{\leveldifference}}}\right ]^2 - \left [ \frac{1}{2}\adaptiveroundbrackets{1-\frac{1}{2^{\leveldifference}}}\right ]^2\right \} 
                + \frac{\tilde{A}_{\levelletter, \indexletter}^{\populationindex, 1}}{2} \left [ -\adaptiveroundbrackets{1+\frac{1}{2^{\leveldifference}}} - \adaptiveroundbrackets{1-\frac{1}{2^{\leveldifference}}} \right ] \\
                = \frac{\tilde{A}_{\levelletter, \indexletter}^{\populationindex, 2} }{2^{\leveldifference}} - \tilde{A}_{\levelletter, \indexletter}^{\populationindex, 1} = \frac{1}{2^{\leveldifference}} \adaptiveroundbrackets{\dfrac{1}{2} \average{\distribution}_{\levelletter, \indexletter - 1}^{\populationindex} - \average{\distribution}_{\levelletter, \indexletter}^{\populationindex} + \dfrac{1}{2} \average{\distribution}_{\levelletter, \indexletter + 1}^{\populationindex}} - \adaptiveroundbrackets{-\dfrac{1}{2} \average{\distribution}_{\levelletter, \indexletter - 1}^{\populationindex} + \dfrac{1}{2} \average{\distribution}_{\levelletter, \indexletter + 1}^{\populationindex}} \\
                = \frac{1}{2} \adaptiveroundbrackets{1 + \frac{1}{2^{\leveldifference}}} \average{\distribution}_{\levelletter, \indexletter - 1}^{\populationindex} -\frac{1}{2^{\leveldifference}} \average{\distribution}_{\levelletter, \indexletter}^{\populationindex} - \frac{1}{2} \adaptiveroundbrackets{1 - \frac{1}{2^{\leveldifference}}} \average{\distribution}_{\levelletter, \indexletter + 1}^{\populationindex},
            \end{multline*}
            which achieves the proof.
        \end{proof}

        One might ask whether it could be possible to see \eqref{eq:LaxWendroff} as a multiresolution scheme by using prediction operators different from \eqref{eq:PredictionOperator}.
        The path would be to consider prediction operators based on two values, which are employed in the point-wise multiresolution analysis \cite{harti1993discrete} but are not suitable to be used with volumetric representations. Indeed, in the symmetric case, the prediction operator would be
        \begin{equation*}
            \predicted{\average{\distribution}}{\populationindex}{\levelletter + 1, 2\indexletter + \delta} = \adaptiveroundbrackets{\frac{1}{2} + \frac{(-1)^{\delta}}{8}} \average{\distribution}_{\levelletter, \indexletter - 1}^{\populationindex} + \adaptiveroundbrackets{\frac{1}{2} - \frac{(-1)^{\delta}}{8}} \average{\distribution}_{\levelletter, \indexletter + 1}^{\populationindex},
        \end{equation*}
        which on one side has the unwanted property of being non-volume preserving: $(\predicted{\average{\distribution}}{\populationindex}{\levelletter + 1, 2\indexletter} + \predicted{\average{\distribution}}{\populationindex}{\levelletter + 1, 2\indexletter + 1})/2 = (\average{\distribution}_{\levelletter, \indexletter - 1}^{\populationindex} + \average{\distribution}_{\levelletter, \indexletter + 1}^{\populationindex})/2 \neq \average{\distribution}_{\levelletter, \indexletter}^{\populationindex}$
        and on the other side cannot generate the Lax-Wendroff scheme by its cascade application.
        %it is easy to see that for $\normalizedvelocityletter_{\populationindex} = \pm 1$ we have $C_{\leveldifference, \pm 2}^{\populationindex} \neq 0$ in general.
        On the other hand, if we imagine to use upwind prediction operators (thus using the value on the cell $\cellletter_{\levelletter, \indexletter}$), it can be shown that such operator degenerates into the Haar prediction $\predictionstencil = 0$, thus we cannot recover the Lax-Wendroff scheme neither.

\section{Equivalent equations analysis in 1D}\label{sec:Equivalent1D}

    We now enter the core section of the present contribution, where we make the link between the analysis of the lattice Boltzmann schemes \emph{via} the equivalent equations \cite{dubois2008equivalent} and our adaptive LBM-MR method \eqref{eq:CollisionPhase}, \eqref{eq:StreamPhase}, for $\spatialdimensionality = 1$.
    The aim is to find the maximum order of accuracy of our adaptive strategies according to the size of the prediction stencil $\predictionstencil$.
    As one might expect, the larger $\predictionstencil$, the better the behavior of the adaptive scheme. This high accuracy is here quantified in terms of numerical analysis.
    This analysis pertains to the way of performing the stream phase and does not take the different models for the collision phase into account.
    This assumption is rigorously justified as long as the equilibria are linear functions but we shall numerically verify that the results of the present study apply to more complicated non-linear situations.

    \subsection{Target expansion}
    
    As it is common in the volumetric context (see \cite{leveque2002finite}), we adopt the point of view of Finite Differences, where the discrete values are point-values at the cells centers $\vectorial{x}_{\levelletter, \vectorial{\indexletter}}$ rather that averages.
    This allows us to study the scheme from the point of view of the equivalent equations, by assuming that the underlying distribution functions are smooth in space and time.
    When considered the current cell is at the finest level of resolution $\maxlevel$, the stream phase  \eqref{eq:StreamPhase} coincides with \eqref{eq:referenceSchemeStream}, yielding what we call ``reference scheme''.
    This is simply a translation of the datum along the characteristics of the velocity field, since we are working at Courant number equal to one\footnote{Indeed, observe that $1/2^{\leveldifference}$ is the local Courant number at a given resolution $\levelletter$.}: $\average{\distribution}^{\populationindex}_{\maxlevel, \vectorial{\indexletter}} (t+\timestep) = \average{\distribution}^{\populationindex, \collided}_{\maxlevel, \vectorial{\indexletter} - \vectorial{\normalizedvelocityletter}_{\populationindex}} (t)$, which written in a Finite Difference formalism and assuming to deal with $\spatialdimensionality = 1$, reads
    \begin{equation*}
        \distribution^{\populationindex} (t + \timestep, x_{\maxlevel, \indexletter}) = \distribution^{\populationindex, \collided} (t, x_{\maxlevel, \indexletter - \normalizedvelocityletter_{\populationindex}}) =  \distribution^{\populationindex, \collided} (t, x_{\maxlevel, \indexletter} - \normalizedvelocityletter_{\populationindex} \spacestep).
    \end{equation*}
    Thus we can apply a Taylor expansion to both sides of the equation, yielding
    \begin{align}\label{eq:TaylorExpansion}
        \sum_{s = 0}^{+\infty} \frac{\timestep^s}{s!} \partial_{t}^s \distribution^{\populationindex}(t, x_{\maxlevel, \indexletter}) &=  \sum_{s = 0}^{+\infty} \frac{(-\normalizedvelocityletter_{\populationindex} \spacestep)^s}{s!} \partial_{x}^s \distribution^{\populationindex, \collided}(t, x_{\maxlevel, \indexletter}) \\
        &= \distribution^{\populationindex, \collided} - \normalizedvelocityletter_{\populationindex} \spacestep \partial_x \distribution^{\populationindex, \collided} + \frac{\normalizedvelocityletter_{\populationindex}^2 \spacestep^2}{2} \partial_{xx} \distribution^{\populationindex, \collided} - \frac{\normalizedvelocityletter_{\populationindex}^3 \spacestep^3}{6} \partial_{x}^3 \distribution^{\populationindex, \collided} + \dots, \nonumber
    \end{align}
    where the argument is omitted when understood.
    This expansion (possibly truncated) is the basic brick of the equivalent equations \cite{dubois2008equivalent}, which are used to study the consistency of the reference lattice Boltzmann scheme.
    In our analysis, we shall only be interested in the right hand side of \eqref{eq:TaylorExpansion}, since the left hand side is not affected by the adaptive grid\footnote{It would be affected for a level dependent time-step in the spirit of \cite{filippova1998grid}.} as we have considered a unique time step imposed by the finest resolution \eqref{eq:TimeStepDefinition}.
    We call the development on the right hand side of \eqref{eq:StreamPhase} ``target expansion''.
    Each power of $\spacestep$ has a different role in determining the physics approximated by the lattice Boltzmann scheme. 
    In particular:
    \begin{itemize}
        \item $- \normalizedvelocityletter_{\populationindex} \spacestep \partial_x \distribution^{\populationindex, \collided}$ is what we might call an ``inertial'' term, because it yields inertial contributions in the approximated model at leading order.
        \item $\frac{\normalizedvelocityletter_{\populationindex}^2 \spacestep^2}{2} \partial_{xx} \distribution^{\populationindex, \collided}$ might be called ``diffusive'' term since it results in dissipative terms in the approximated model at order $\spacestep$.
        \item $- \frac{\normalizedvelocityletter_{\populationindex}^3 \spacestep^3}{6} \partial_{x}^3 \distribution^{\populationindex, \collided}$ might be called ``dispersive'' term in analogy with Finite Differences. Its physical meaning is less clear than for the other terms but it can have a non-negligible impact on the stability of the lattice Boltzmann method.
    \end{itemize}

    Looking at \eqref{eq:StreamPhase}, it is not immediately clear how to apply the previous analysis, because the evolution of the solution is written in term of the reconstruction operator acting on the distributions.
    The connection with the previous expansion strategy shall be achieved using what we might call ``reconstruction flattening''.
    Intuitively, considering Remark \ref{rem:polynomialExactness}, our method is expected to show less discrepancies from the reference method as the width of the prediction stencil $\predictionstencil$ grows.
    
    \subsection{Reconstruction flattening}\label{sec:ReconstructionFlattening}
    
        %\definecolor{redingoingflux}{rgb}{0.81960784313,0.21176470588,0.21176470588}
        \definecolor{blueingoingflux}{rgb}{0.,0.44705882352,0.69803921568}

        \begin{figure}[h]
        \begin{center}
            \def\svgwidth{1.\textwidth}
            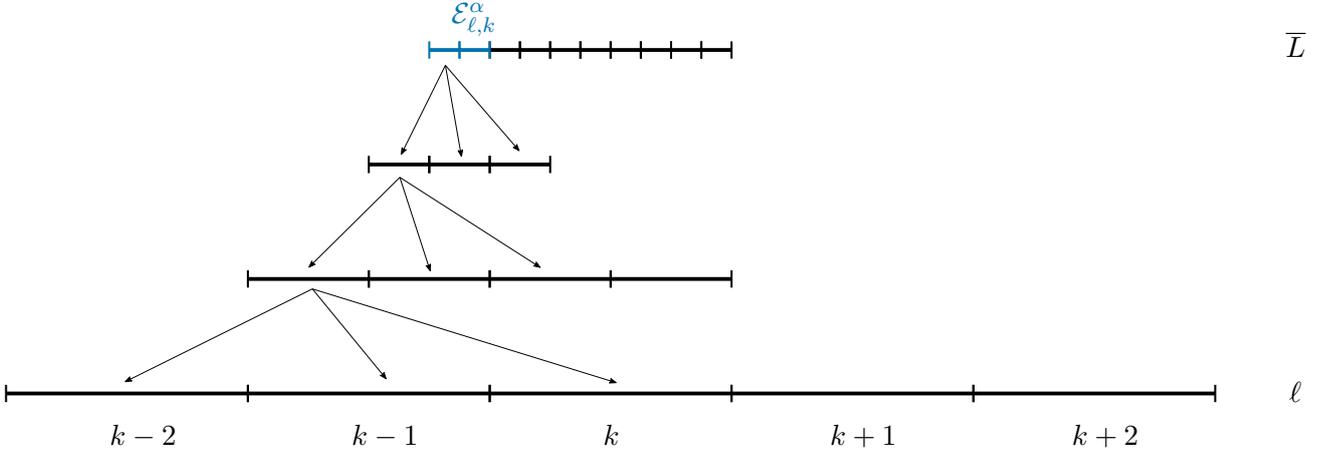
        \end{center}
        \caption{\label{fig:flattening}Example of flattening procedure with $\predictionstencil = 1$ for a velocity $\normalizedvelocityletter_{\populationindex} = 2$. The cells in blue correspond to those belonging to $\mathcal{E}_{\levelletter, \indexletter}^{\populationindex}$. The prediction operator is recursively applied (arrows) until reaching the level we are looking for, namely $\levelletter$.}
        \end{figure}
        
        In what follows, we assume that the stream stencil of the considered method implies at most two neighbors in each direction, this is $\max_{\populationindex} |\normalizedvelocityletter_{\populationindex}|\leq 2$.
        Therefore this analysis covers schemes such as the D1Q2 \cite{graille2014approximation, caetano2019result}, D1Q3 \cite{dubois2013stable} and D1Q5 \cite{bellotti2021sisc} schemes.
        However, the study can be extended to larger stencils upon considering sums in \eqref{eq:FlattenedStream} spanning on a larger set of integers.
        The generalization to $\predictionstencil \geq 2$ is achieved in the same manner.
        
        Flattening the reconstruction operator means that we are able to compute the set of weights $(C_{\leveldifference, m}^{\populationindex})_{m = -2}^{m = +2} \subset \mathbb{R}$ such that
        \begin{align}
            \average{\distribution}^{\populationindex}_{\levelletter, \indexletter} (t+\timestep) &= \average{\distribution}^{\populationindex, \collided}_{\levelletter, {\indexletter}} (t) + \frac{1}{2^{\leveldifference}} \adaptiveroundbrackets{\sum_{\overline{{\indexletter}} \in \mathcal{E}_{\levelletter, {\indexletter}}^{\populationindex}} \reconstructed{\average{\distribution}}{\populationindex, \collided}{\maxlevel, \overline{{\indexletter}}} (t) - \sum_{\overline{{\indexletter}} \in \mathcal{A}_{\levelletter, {\indexletter}}^{\populationindex}} \reconstructed{\average{\distribution}}{\populationindex, \collided}{\maxlevel, \overline{{\indexletter}}} (t)} 
            = \average{\distribution}^{\populationindex, \collided}_{\levelletter, \indexletter} (t) + \frac{1}{2^{\leveldifference}} \sum_{m = -2}^{+2} C_{\leveldifference, m}^{\populationindex} \average{\distribution}^{\populationindex, \collided}_{\levelletter, \indexletter + m} (t), \label{eq:FlattenedStream}
        \end{align}
        which corresponds to the illustration of Figure \ref{fig:flattening}.
        This means that we have condensed the computation of the total pseudo-flux at the finest level $\maxlevel$ as a weighted sum of values on five neighbors at the current level $\levelletter$.
        Recall Remark \ref{rem:UniformMesh}.
        \begin{remark}
            It is important to observe that both our multiresolution \eqref{eq:StreamPhase} and the Lax-Wendroff scheme \eqref{eq:LaxWendroff} can be put under the formalism introduced in \eqref{eq:FlattenedStream}.
        \end{remark}
        The fundamental advantage of this representation is that we can safely develop in Taylor series around the considered cell by adopting a Finite Difference point of view.
        This gives
        \begin{align}\label{eq:FlattenedExpansion}
            \sum_{s = 0}^{+\infty} \frac{\timestep^s}{s!} \partial_{t}^s \distribution^{\populationindex}(t, x_{\levelletter, \indexletter}) &= \distribution^{\populationindex, \collided}(t, x_{\levelletter, \indexletter}) + \sum_{s = 0}^{+\infty} \adaptiveroundbrackets{\frac{(\spacestep_{\levelletter})^s}{2^{\leveldifference}s!} \adaptiveroundbrackets{ \sum_{m = -2}^{+2}m^s  C_{\leveldifference, m}^{\populationindex}}  \partial_{x}^s \distribution^{\populationindex, \collided}(t, x_{\levelletter, \indexletter})}, \nonumber \\
            &= \distribution^{\populationindex, \collided}(t, x_{\levelletter, \indexletter}) + \sum_{s = 0}^{+\infty} \adaptiveroundbrackets{\frac{2^{\leveldifference(s-1)}(\spacestep)^s}{s!} \adaptiveroundbrackets{ \sum_{m = -2}^{+2}m^s  C_{\leveldifference, m}^{\populationindex}}  \partial_{x}^s \distribution^{\populationindex, \collided}(t, x_{\levelletter, \indexletter})}, \\
            &= \adaptiveroundbrackets{1 + \frac{1}{2^{\leveldifference}} \sum_{m = -2}^{+2} C_{\leveldifference, m}^{\populationindex}} \distribution^{\populationindex, \collided} + \adaptiveroundbrackets{\sum_{m = -2}^{+2} m C_{\leveldifference, m}^{\populationindex}} \spacestep \partial_x \distribution^{\populationindex, \collided} \nonumber \\
            & +  \adaptiveroundbrackets{2^{\leveldifference} \sum_{m = -2}^{+2} m^2 C_{\leveldifference, m}^{\populationindex}} \frac{\spacestep^2}{2} \partial_{xx} \distribution^{\populationindex, \collided} +  \adaptiveroundbrackets{2^{2\leveldifference} \sum_{m = -2}^{+2} m^3 C_{\leveldifference, m}^{\populationindex}} \frac{\spacestep^3}{6} \partial_{x}^3 \distribution^{\populationindex, \collided} + \dots \nonumber
        \end{align}
        where we have used that $\spacestep_{\levelletter} = 2^{\leveldifference} \spacestep$.
        The aim is to equate as most coefficients as possible between \eqref{eq:TaylorExpansion} and \eqref{eq:FlattenedExpansion} in order to have an adaptive scheme at the local level of refinement $\levelletter$ with approximated physics and stability conditions as close as possible to that of the reference scheme at level $\maxlevel$. We stress that these conditions are checked locally but we request them for any possible level.
        By doing the comparison term-by-term, we end up with the following definition.

        %{\bfseries Benjamin: pour la définition, pourquoi ne pas mettre à l'ordre p on doit avoir les égalités pour  s allant de 0 à p ?}
        \begin{definition}[Match]
            Let $\spatialdimensionality = 1$. We say that the adaptive stream phase \eqref{eq:FlattenedStream} matches that of the reference scheme at order $s \in \mathbb{N}^{\star}$ whenever
        \begin{equation}\label{eq:MatchAtOrderS}
            \sum_{m = -2}^{+2} C_{\leveldifference, m}^{\populationindex} = 0, \qquad \sum_{m = -2}^{+2} m^p C_{\leveldifference, m}^{\populationindex} = \frac{(-\normalizedvelocityletter_{\populationindex})^p}{2^{\leveldifference(p-1)}}, \quad \text{for} \quad p = 1, \dots, s.
        \end{equation}
        regardless of the level $\levelletter$, thus for any $\leveldifference \geq 0$ and for every velocity $\populationindex = 0, \dots, \velocitynumber-1$.
        %{\bfseries Benjamin: il faut peut-être bien préciser pour $\alpha\in ?$}
        \end{definition} 
        With this definition in mind, we study how this property holds for the most common multiresolution schemes, namely $\predictionstencil = 0$ and $\predictionstencil = 1$ and for the Lax-Wendroff scheme.
        
        \subsection{Haar wavelet: $\predictionstencil = 0$}
            
            We start by analyzing the simplest scheme, generated by the Haar wavelet for $\predictionstencil = 0$.
            We provide the weights for the correspondent reconstruction flattening \eqref{eq:FlattenedStream} in the following statement
        
            \begin{proposition}[Match for $\predictionstencil = 0$]
                Let $\spatialdimensionality = 1$, $\predictionstencil = 0$ and $\leveldifference \geq 0$, then the flattened weights of the adaptive stream phase given by \eqref{eq:FlattenedStream} read
                \begin{equation*}
                    C_{\leveldifference, 0}^{\populationindex} = -|\normalizedvelocityletter_{\populationindex}|, \qquad C_{\leveldifference, -\normalizedvelocityletter_{\populationindex}/|\normalizedvelocityletter_{\populationindex}|}^{\populationindex} = |\normalizedvelocityletter_{\populationindex}|,
                \end{equation*}
                and those not listed are equal to zero.
                Therefore, the adaptive stream phase matches that of the reference scheme up to order $s = 1$.
                This also writes
                \begin{equation*}
                    \sum_{m = -2}^{+2} C_{\leveldifference, m}^{\populationindex} = 0, \qquad  \sum_{m = -2}^{+2} m C_{\leveldifference, m}^{\populationindex} = -\normalizedvelocityletter_{\populationindex}.
                \end{equation*}
            \end{proposition}
            
            \begin{proof}
                    Looking at \eqref{eq:PredictionOperator} with $\predictionstencil = 0$, each stage of the reconstruction operator acts by looking for the father of the cell we predict on. Overall, for a cell at the finest level, it returns the value on its ancestor at level $\levelletter$.
                    Therefore (consider positively moving velocity for the sake of presentation)
we have for $\normalizedvelocityletter_{\populationindex} = 1$, $\mathcal{E}_{\levelletter, \indexletter}^{\populationindex} = \{2^{\leveldifference}\indexletter-1\}$,  $\mathcal{A}_{\levelletter, \indexletter}^{\populationindex} = \{2^{\leveldifference}(\indexletter+1)-1\}$  and
                        \begin{equation*}
                            \reconstructed{\average{\distribution}}{\populationindex}{\maxlevel, 2^{\leveldifference}\indexletter - 1} = \average{\distribution}^{\populationindex}_{\levelletter, \indexletter - 1}, \qquad \reconstructed{\average{\distribution}}{\populationindex}{\maxlevel, 2^{\leveldifference}(\indexletter+1) - 1} = \average{\distribution}^{\populationindex}_{\levelletter, \indexletter}.
                        \end{equation*}
For the case $\normalizedvelocityletter_{\populationindex} = 2$, we have $\mathcal{E}_{\levelletter, \indexletter}^{\populationindex} = \{2^{\leveldifference}\indexletter-1, 2^{\leveldifference}\indexletter-2\}$,  $\mathcal{A}_{\levelletter, \indexletter}^{\populationindex} = \{2^{\leveldifference}(\indexletter+1)-1, 2^{\leveldifference}(\indexletter+1)-2\}$ and
                        \begin{gather*}
                            \reconstructed{\average{\distribution}}{\populationindex}{\maxlevel, 2^{\leveldifference}\indexletter - 1} = \average{\distribution}^{\populationindex}_{\levelletter, \indexletter - 1}, \qquad \reconstructed{\average{\distribution}}{\populationindex}{\maxlevel, 2^{\leveldifference}\indexletter - 2} = \average{\distribution}^{\populationindex}_{\levelletter, \indexletter - 1}, \\
                            \reconstructed{\average{\distribution}}{\populationindex}{\maxlevel, 2^{\leveldifference}(\indexletter+1) - 1} = \average{\distribution}^{\populationindex}_{\levelletter, \indexletter}, \qquad \reconstructed{\average{\distribution}}{\populationindex}{\maxlevel, 2^{\leveldifference}(\indexletter+1) - 2} = \average{\distribution}^{\populationindex}_{\levelletter, \indexletter}.
                        \end{gather*}
                    Plugging into \eqref{eq:FlattenedStream} completes the first part of the proof.
                    The second part is an immediate consequence of the first:
                    % \begin{itemize}
                    %     \item We have $\sum_{m = -2}^{+2} C_{\leveldifference, m}^{\populationindex} = C_{\leveldifference, 0}^{\populationindex} + C_{\leveldifference, -\normalizedvelocityletter_{\populationindex}/|\normalizedvelocityletter_{\populationindex}|}^{\populationindex} = - |\normalizedvelocityletter_{\populationindex}|+ |\normalizedvelocityletter_{\populationindex}| = 0$.
                    %     \item We have $\sum_{m = -2}^{+2} mC_{\leveldifference, m}^{\populationindex} = 0 \times C_{\leveldifference, 0}^{\populationindex} -\normalizedvelocityletter_{\populationindex}/|\normalizedvelocityletter_{\populationindex}| \times C_{\leveldifference, -\normalizedvelocityletter_{\populationindex}/|\normalizedvelocityletter_{\populationindex}|}^{\populationindex} = -\normalizedvelocityletter_{\populationindex}/|\normalizedvelocityletter_{\populationindex}| \times |\normalizedvelocityletter_{\populationindex}| = - \normalizedvelocityletter_{\populationindex}$.
                    % \end{itemize}
\begin{gather*}
                    \smash{\sum_{m = -2}^{+2}} C_{\leveldifference, m}^{\populationindex} = C_{\leveldifference, 0}^{\populationindex} + C_{\leveldifference, -\normalizedvelocityletter_{\populationindex}/|\normalizedvelocityletter_{\populationindex}|}^{\populationindex} = - |\normalizedvelocityletter_{\populationindex}|+ |\normalizedvelocityletter_{\populationindex}| = 0,\\
                        \smash[b]{\sum_{m = -2}^{+2}} mC_{\leveldifference, m}^{\populationindex} = 0 \times C_{\leveldifference, 0}^{\populationindex} -\normalizedvelocityletter_{\populationindex}/|\normalizedvelocityletter_{\populationindex}| \times C_{\leveldifference, -\normalizedvelocityletter_{\populationindex}/|\normalizedvelocityletter_{\populationindex}|}^{\populationindex} = -\normalizedvelocityletter_{\populationindex}/|\normalizedvelocityletter_{\populationindex}| \times |\normalizedvelocityletter_{\populationindex}| = - \normalizedvelocityletter_{\populationindex}
                    \end{gather*}
                \end{proof}

            On the other hand, the term diffusive term at $O(\spacestep^2)$ does not match that of the target expansion.
            This can be easily seen by taking $\normalizedvelocityletter_{\populationindex} = 1$.
            Then $2^{\leveldifference} \sum_{m = -2}^{m = +2} m^2 C_{\leveldifference, m}^{\populationindex} =  2^{\leveldifference} \neq 1$.
            This has a major consequence on the applicability of the method based on the Haar wavelet, because it correctly recovers the ``inertial'' terms but fails in correctly accounting for the ``dissipative'' terms, yielding a wrong diffusion structure with respect to the equations targeted by the reference method.
            %An example of reference method which would be profoundly altered by this adaptive strategy is the D2Q9 scheme \cite{lallemand2000theory} for the incompressible Navier-Stokes system.
            The well-known D2Q9 scheme \cite{lallemand2000theory} for the incompressible Navier-Stokes system is one possible example of lattice Boltzmann scheme which would be deeply altered by the 2D generalization of this strategy.
            
        \subsection{Non-trivial wavelet: $\predictionstencil = 1$}
        
            Afterwards, we consider the case $\predictionstencil = 1$, which has been thoroughly investigated in \cite{bellotti2021sisc, bellotti2021multidimensional}.
            We are going to see that the limitations of the Haar case $\predictionstencil = 0$ can be solved by consider a larger prediction stencil: indeed -- for most of the applications -- employing $\predictionstencil = 1$ is sufficient.
            
            \begin{proposition}[Match for $\predictionstencil = 1$]\label{prop:CoefficientsGamma1}
                Let $\spatialdimensionality = 1$, $\predictionstencil = 1$ and $\leveldifference > 0$, then the flattened weights of the stream phase \eqref{eq:FlattenedStream} are given by the recurrence relations
                \begin{equation*}
                    \begin{pmatrix}
                                            C_{\leveldifference, -2}^{\populationindex} \\
                                            C_{\leveldifference, -1}^{\populationindex} \\
                                            C_{\leveldifference,  0}^{\populationindex} \\
                                            C_{\leveldifference,  1}^{\populationindex} \\
                                            C_{\leveldifference,  2}^{\populationindex}
                                           \end{pmatrix}
                = \underbrace{\begin{pmatrix}
                                        0 & -1/8 & 0 & 0 & 0 \\
                                        2 & 9/8 & 0 & -1/8 & 0 \\
                                        0 & 9/8 & 2 & 9/8 & 0 \\
                                        0 & -1/8 & 0 & 9/8 & 2 \\
                                        0 & 0 & 0 & -1/8 & 0
                                         \end{pmatrix}}_{=: \operatorial{P}}
                        \begin{pmatrix}
                                            C_{\leveldifference-1, -2}^{\populationindex} \\
                                            C_{\leveldifference-1, -1}^{\populationindex} \\
                                            C_{\leveldifference-1,  0}^{\populationindex} \\
                                            C_{\leveldifference-1,  1}^{\populationindex} \\
                                            C_{\leveldifference-1,  2}^{\populationindex}
                                           \end{pmatrix},
                \end{equation*}
                where the initialization is given by $C_{0, -\normalizedvelocityletter_{\populationindex}}^{\populationindex} = 1$ and $C_{0, 0}^{\populationindex} = -1$ and the remaining terms set to zero.
                Therefore, the adaptive stream phase matches that of the reference scheme up to order $s = 3$.
                This also writes
                \begin{gather*}
                    \sum_{m = -2}^{+2} C_{\leveldifference, m}^{\populationindex} = 0, \qquad  \sum_{m = -2}^{+2} m C_{\leveldifference, m}^{\populationindex} = -\normalizedvelocityletter_{\populationindex}, \qquad 
                    \sum_{m = -2}^{+2} m^2 C_{\leveldifference, m}^{\populationindex} =  \dfrac{\normalizedvelocityletter_{\populationindex}^2}{2^{\leveldifference}}, \qquad \sum_{m = -2}^{+2} m^3 C_{\leveldifference, m}^{\populationindex} = - \dfrac{\normalizedvelocityletter_{\populationindex}^3}{4^{\leveldifference}}.
                \end{gather*}
            \end{proposition}

                \begin{proof}
                    The initialization trivially gives the scheme \eqref{eq:StreamPhase}. Assume to know the coefficients of the flattened advection for level $\levelletter + 1$, that is for $\leveldifference - 1$.
                    We have (we omit the time $t$ for the sake of compactness)
                    \begin{align*}
                        \sum_{\overline{\indexletter} \in \mathcal{E}_{\levelletter, \indexletter}^{\populationindex}} \reconstructed{\average{\distribution}}{\populationindex, \collided}{\maxlevel, \overline{\indexletter}} - \sum_{\overline{\indexletter} \in \mathcal{A}_{\levelletter, \indexletter}^{\populationindex}} \reconstructed{\average{\distribution}}{\populationindex}{\maxlevel, \overline{\indexletter}} &= \adaptiveroundbrackets{\sum_{\overline{\indexletter} \in \mathcal{E}_{\levelletter + 1, 2 \indexletter}^{\populationindex}} \reconstructed{\average{\distribution}}{\populationindex, \collided}{\maxlevel, \overline{\indexletter}} - \sum_{\overline{\indexletter} \in \mathcal{A}_{\levelletter + 1, 2\indexletter}^{\populationindex}} \reconstructed{\average{\distribution}}{\populationindex, \collided}{\maxlevel, \overline{\indexletter}} } + \adaptiveroundbrackets{\sum_{\overline{\indexletter} \in \mathcal{E}_{\levelletter + 1, 2 \indexletter + 1}^{\populationindex}} \reconstructed{\average{\distribution}}{\populationindex, \collided}{\maxlevel, \overline{\indexletter}} - \sum_{\overline{\indexletter} \in \mathcal{A}_{\levelletter + 1, 2\indexletter + 1}^{\populationindex}} \reconstructed{\average{\distribution}}{\populationindex, \collided}{\maxlevel, \overline{\indexletter}} }, \\
                        &= \sum_{m = -2}^{+2} C_{\leveldifference - 1, m}^{\populationindex} \predicted{\average{\distribution}}{\populationindex, \collided}{\levelletter + 1, 2\indexletter + m} + \sum_{m = -2}^{+2} C_{\leveldifference - 1, m}^{\populationindex} \predicted{\average{\distribution}}{\populationindex, \collided}{\levelletter + 1, 2\indexletter + 1 + m}, \\
                        &= \sum_{m = -2}^{+2} C_{\leveldifference - 1, m}^{\populationindex} \predicted{\average{\distribution}}{\populationindex, \collided}{\levelletter + 1, 2\indexletter + m} + \sum_{m = -1}^{+3} C_{\leveldifference - 1, m-1}^{\populationindex} \predicted{\average{\distribution}}{\populationindex, \collided}{\levelletter + 1, 2\indexletter + m} \\
                        &= \sum_{m = -2}^{+3} \tilde{C}_{\leveldifference - 1, m}^{\populationindex} \predicted{\average{\distribution}}{\populationindex, \collided}{\levelletter + 1, 2\indexletter + m},
                    \end{align*}
                    with 
                    \begin{equation*}
                        \tilde{C}_{\leveldifference - 1, m}^{\populationindex} = 
                            \begin{cases}
                                C_{\leveldifference - 1, -2}^{\populationindex}, \qquad &m = -2, \\
                                C_{\leveldifference - 1, m}^{\populationindex} + C_{\leveldifference - 1, m - 1}^{\populationindex}, \qquad &m = -1, 0, 1, 2, \\
                                C_{\leveldifference - 1, 2}^{\populationindex}, \qquad &m = 3.
                            \end{cases}
                    \end{equation*}
                    Using the prediction operator
                    \begin{align*}
                        \sum_{m = -2}^{+3} \tilde{C}_{\leveldifference - 1, m}^{\populationindex} \predicted{\average{\distribution}}{\populationindex, \collided}{\levelletter + 1, 2\indexletter + m} &= \tilde{C}_{\leveldifference - 1, -2}^{\populationindex} \adaptiveroundbrackets{\distribution_{\levelletter, \indexletter - 1} + \dfrac{1}{8}\distribution_{\levelletter, \indexletter - 2} - \dfrac{1}{8}\distribution_{\levelletter, \indexletter}} + \tilde{C}_{\leveldifference - 1, -1}^{\populationindex,} \adaptiveroundbrackets{\distribution_{\levelletter, \indexletter - 1} - \dfrac{1}{8}\distribution_{\levelletter, \indexletter - 2} + \dfrac{1}{8}\distribution_{\levelletter, \indexletter}} \\
                        &+ \tilde{C}_{\leveldifference - 1, 0}^{\populationindex} \adaptiveroundbrackets{\distribution_{\levelletter, \indexletter} + \dfrac{1}{8}\distribution_{\levelletter, \indexletter - 1} - \dfrac{1}{8}\distribution_{\levelletter, \indexletter+1}} + \tilde{C}_{\leveldifference - 1, 1}^{\populationindex} \adaptiveroundbrackets{\distribution_{\levelletter, \indexletter} - \dfrac{1}{8}\distribution_{\levelletter, \indexletter - 1} + \dfrac{1}{8}\distribution_{\levelletter, \indexletter+1}} \\
                        &+ \tilde{C}_{\leveldifference - 1, 2}^{\populationindex} \adaptiveroundbrackets{\distribution_{\levelletter, \indexletter + 1} + \dfrac{1}{8}\distribution_{\levelletter, \indexletter} - \dfrac{1}{8}\distribution_{\levelletter, \indexletter+2}} + \tilde{C}_{\leveldifference - 1, 3}^{\populationindex} \adaptiveroundbrackets{\distribution_{\levelletter, \indexletter + 1} - \dfrac{1}{8}\distribution_{\levelletter, \indexletter} + \dfrac{1}{8}\distribution_{\levelletter, \indexletter+2}},
                    \end{align*}
                    so that after tedious computations, we arrive at
                    \begin{align*}
                        \sum_{m = -2}^{+3} \tilde{C}_{\leveldifference - 1, m}^{\populationindex} \predicted{\average{\distribution}}{\populationindex, \collided}{\levelletter + 1, 2\indexletter + m} &= \adaptiveroundbrackets{-\dfrac{1}{8} C_{\leveldifference - 1, -1}^{\populationindex}} \average{\distribution}_{\levelletter, \indexletter - 2}^{\populationindex, \collided} + \adaptiveroundbrackets{2 C_{\leveldifference - 1, -2}^{\populationindex} + \dfrac{9}{8} C_{\leveldifference - 1, -1}^{\populationindex} - \dfrac{1}{8} C_{\leveldifference - 1, 1}^{\populationindex}}\average{\distribution}_{\levelletter, \indexletter - 1}^{\populationindex, \collided} \\
                        &+ \adaptiveroundbrackets{\dfrac{9}{8} C_{\leveldifference - 1, -1}^{\populationindex} + 2 C_{\leveldifference - 1, 0}^{\populationindex} + \dfrac{9}{8} C_{\leveldifference - 1, 1}^{\populationindex}} \average{\distribution}_{\levelletter, \indexletter}^{\populationindex, \collided} \\
                        &+ \adaptiveroundbrackets{- \dfrac{1}{8} C_{\leveldifference - 1, -1}^{\populationindex} + \dfrac{9}{8} C_{\leveldifference - 1, 1}^{\populationindex}  + 2 C_{\leveldifference - 1, 2}^{\populationindex} } \average{\distribution}_{\levelletter, \indexletter + 1}^{\populationindex, \collided} + \adaptiveroundbrackets{-\dfrac{1}{8} C_{\leveldifference - 1, 1}^{\populationindex}} \average{\distribution}_{\levelletter, \indexletter + 2}^{\populationindex, \collided},
                    \end{align*}
                    concluding the first part of the proof.
                    Then, let us proceed by recurrence: for $\leveldifference = 0$ the thesis trivially holds. Assume that it holds for $\leveldifference - 1$.
                    \begin{itemize}
                        \item $\sum_{m = -2}^{+2} C_{\leveldifference, m}^{\populationindex} = \dots = 2 \sum_{m = -2}^{+2} C_{\leveldifference - 1, m}^{\populationindex} = 0$.
                        \item  $\sum_{m = -2}^{+2} m C_{\leveldifference, m}^{\populationindex} = \dots =   \sum_{m = -2}^{+2} m C_{\leveldifference - 1, m}^{\populationindex} = -\normalizedvelocityletter_{\populationindex}$.
                        \item $\sum_{m = -2}^{+2} m^2 C_{\leveldifference, m}^{\populationindex} = \dots = \frac{1}{2}\sum_{m = -2}^{+2} m^2 C_{\leveldifference-1, m}^{\populationindex} = \frac{1}{2} \frac{\normalizedvelocityletter_{\populationindex}^2}{2^{\leveldifference - 1}} = \frac{\normalizedvelocityletter_{\populationindex}^2}{2^{\leveldifference}}$.
                        \item $\sum_{m = -2}^{+2} m^3 C_{\leveldifference, m}^{\populationindex} = \dots = \frac{1}{4}\sum_{m = -2}^{+2} m^3 C_{\leveldifference-1, m}^{\populationindex} = -\frac{1}{4} \frac{\normalizedvelocityletter_{\populationindex}^3}{4^{\leveldifference - 1}} = -\frac{\normalizedvelocityletter_{\populationindex}^3}{4^{\leveldifference}}$,
                    \end{itemize}
                    that concludes the proof.
                \end{proof}

            Again, we cannot go further in matching the target expansion \eqref{eq:StreamPhase}, considering for example $\normalizedvelocityletter_{\populationindex} = 1$, we have that $\sum_{m = -2}^{m = +2} m^4 C_{1, m}^{\populationindex} = -7/8 \neq 1/8$.
            This means that the method for $\predictionstencil = 1$ can be successfully employed in contexts where we want to control both the ``inertial'' and the ``diffusive'' physics, like in the D2Q9 scheme for the quasi-incompressible Navier-Stokes system \cite{lallemand2000theory}.
            Moreover, since we also match the target expansion at order 3, the achievements accomplished on the reference scheme at this order are also preserved by the adaptive scheme. 
            This is a highly desirable feature for scientists who have a good understanding of their reference scheme and who would like to employ our adaptive strategy as a black-box. Even if we are performing an asymptotic analysis, we argue that this feature reduces discrepancies in terms of stability compared to the reference method.
            
            \subsection{Lax-Wendroff scheme}
            To conclude on the one-dimensional analysis, let us analyze the consistency of the Lax-Wendroff scheme.
            \begin{proposition}[Match for Lax-Wendroff]
                Let $\spatialdimensionality = 1$ and $\leveldifference \geq 0$, then the flattened weights of the stream phase given by \eqref{eq:LaxWendroff} are given by
                \begin{equation}\label{eq:LaxWendroffCoefficients}
             C_{\leveldifference, 0}^{\populationindex} = -\frac{|\normalizedvelocityletter_{\populationindex}|^2}{2^{\leveldifference}}, \qquad C_{\leveldifference, -\normalizedvelocityletter_{\populationindex}/|\normalizedvelocityletter_{\populationindex}|}^{\populationindex} = \frac{|\normalizedvelocityletter_{\populationindex}|}{2}\adaptiveroundbrackets{1 + \frac{|\normalizedvelocityletter_{\populationindex}|}{2^{\leveldifference}}}, \qquad C_{\leveldifference, \normalizedvelocityletter_{\populationindex}/|\normalizedvelocityletter_{\populationindex}|}^{\populationindex} = -\frac{|\normalizedvelocityletter_{\populationindex}|}{2}\adaptiveroundbrackets{1 - \frac{|\normalizedvelocityletter_{\populationindex}|}{2^{\leveldifference}}}.
            \end{equation}
             Therefore, the adaptive stream phase matches that of the reference scheme up to order $s = 2$.
                This also writes
        \begin{equation*}
            \sum_{m = -2}^{+2} C_{\leveldifference, m}^{\populationindex} = 0, \qquad  \sum_{m = -2}^{+2} mC_{\leveldifference, m}^{\populationindex} = -\normalizedvelocityletter_{\populationindex}, \qquad  \sum_{m = -2}^{+2} m^2 C_{\leveldifference, m}^{\populationindex} = \frac{\normalizedvelocityletter_{\populationindex}^{2}}{2^{\leveldifference}},
        \end{equation*}
        \end{proposition}
        \begin{proof}
            The first part of the claim is true by direct inspection of \eqref{eq:LaxWendroff}. The second part comes from the usual straightforward computations.
        \end{proof}

        However as expected from such a kind of scheme, the dispersive order $s=3$ is not matched, because $\sum_{m = 2}^{+2}m^{3} C_{\leveldifference, m}^{\populationindex} = -\normalizedvelocityletter_{\populationindex}^{3}/|\normalizedvelocityletter_{\populationindex}|^2 \neq - \normalizedvelocityletter_{\populationindex}^3/{4^{\leveldifference}}$.
        This is not trivial when considering Proposition \ref{prop:LWisNotMR}, because the stream phase is built using the same interpolation than the multiresolution scheme for $\predictionstencil = 1$.
        Still, what changes is the multiresolution approach employs the reconstruction operator with recursive application of the prediction operator, whereas the Lax-Wendroff scheme uses the polynomial interpolation only once.
        
        \subsection{Conclusions}
        In this section, we have seen that in the case of multiresolution scheme, the prediction stencil $\predictionstencil$ has to be taken large enough in order to match a sufficient number of desired orders. 
        In particular, the number of matched orders is equal to $2 \predictionstencil + 1$.
        Therefore, for most of the applications, $\predictionstencil = 0$ (Haar wavelet) is not enough, because it modifies the second-order terms which are frequently used to model diffusion phenomena.
        On the other hand, $\predictionstencil = 1$ is often sufficient for most of the applications and its reliability on the third-order terms is an interesting ``icing on the cake''.
        Finally, the Lax-Wendroff scheme matches sharply until order two as pointed out in \cite{fakhari2015numerics}, so it successfully handles diffusive terms but can lead to different stability constraints and oscillatory behaviors for the scheme, due to its intrinsic dispersive nature.

\section{Equivalent equations analysis in 2D}\label{sec:Equivalent2D}
    
    So far, we have analyzed our LBM-MR schemes with the help of the equivalent equations once the reconstruction flattening is done.
    In this section, we succinctly show how the previous analysis can be extended to the multidimensional case $\spatialdimensionality \geq 2$ by exploiting the tensorial product structure of the prediction operator, in order to simplify the computations.
    As for the previous section, the analysis holds as long as $\max_{\populationindex} |\vectorial{\normalizedvelocityletter}_{\populationindex}|_{\infty} \leq 2$.
    For the sake of presentation, we present the case $\spatialdimensionality = 2$.
    Therefore, the advection phase reads, in a Finite Differences formalism
    \begin{equation*}
        \distribution^{\populationindex} (t + \timestep, \vectorial{x}_{\maxlevel, \vectorial{\indexletter}}) = \distribution^{\populationindex, \collided} (t, \vectorial{x}_{\maxlevel, \vectorial{\indexletter} - \vectorial{\normalizedvelocityletter}_{\populationindex}}) =  \distribution^{\populationindex, \collided} (t, \vectorial{x}_{\maxlevel, \indexletter} - \vectorial{\normalizedvelocityletter}_{\populationindex} \spacestep),
    \end{equation*}
    thus a Taylor expansion -- assuming that we are allowed to commute derivatives along different axis by virtue of the Schwarz theorem -- yields
    \begin{align*}
        \sum_{s = 0}^{+\infty} \frac{\timestep^s}{s!} \partial_{t}^s \distribution^{\populationindex}(t, \vectorial{x}_{\maxlevel, \indexletter}) &= \sum_{s = 0}^{+\infty} \sum_{p = 0}^{+\infty} \frac{(\spacestep)^{s+p} (-\normalizedvelocityletter_{\populationindex, x})^s (-\normalizedvelocityletter_{\populationindex, y})^p}{s! ~ p!} \partial_x^{s} \partial_y^p \distribution^{\populationindex, \collided}(t, \vectorial{x}_{\maxlevel, \indexletter}), \\
        &= \distribution^{\populationindex, \collided} - \normalizedvelocityletter_{\populationindex, x} \spacestep \partial_x \distribution^{\populationindex, \collided}  - \normalizedvelocityletter_{\populationindex, y} \spacestep \partial_y \distribution^{\populationindex, \collided} \\ 
        &+ \frac{\normalizedvelocityletter_{\populationindex, x}^2 \spacestep^2}{2} \partial_{xx} \distribution^{\populationindex, \collided} + \normalizedvelocityletter_{\populationindex, x} \normalizedvelocityletter_{\populationindex, y} \spacestep^2 \partial_{xy} \distribution^{\populationindex, \collided} + \frac{\normalizedvelocityletter_{\populationindex, y}^2 \spacestep^2}{2} \partial_{yy} \distribution^{\populationindex, \collided}\\
        &- \frac{\normalizedvelocityletter_{\populationindex, x}^3 \spacestep^3}{6} \partial_{x}^3 \distribution^{\populationindex, \collided} - \frac{\normalizedvelocityletter_{\populationindex, x}^2 \normalizedvelocityletter_{\populationindex, y} \spacestep^3}{2} \partial_{xxy}\distribution^{\populationindex, \collided} - \frac{\normalizedvelocityletter_{\populationindex, x} \normalizedvelocityletter_{\populationindex, y}^2 \spacestep^3}{2} \partial_{xyy}\distribution^{\populationindex, \collided} - \frac{\normalizedvelocityletter_{\populationindex, y}^3 \spacestep^3}{6} \partial_{y}^3 \distribution^{\populationindex, \collided} \\
        &+ \dots
    \end{align*}
    This is the new target expansion analogous to \eqref{eq:TaylorExpansion}.
    Now, the reconstruction flattening reads
    \begin{equation}\label{eq:FlattenedStream2D}
        \average{\distribution}_{\levelletter, \vectorial{\indexletter}}^{\populationindex} (t+\timestep) = \average{\distribution}_{\levelletter, \vectorial{\indexletter}}^{\populationindex, \collided} (t) + \frac{1}{4^{\leveldifference}} \sum_{m = -2}^{+2} \sum_{n = -2}^{+2} C_{\leveldifference, m, n}^{\populationindex} \average{\distribution}_{\levelletter, \vectorial{\indexletter} + m \hat{\vectorial{\imath}} + n \hat{\vectorial{\jmath}}}^{\populationindex, \collided} (t),
    \end{equation}
    where $\hat{\vectorial{\imath}}$ and $\hat{\vectorial{\jmath}}$ are respectively the versor of the $x$ and $y$ Cartesian axis.
    Observe that the coefficients $(C_{\leveldifference, m, n}^{\populationindex})_{m, n=-2}^{m, n=+2}$ are not directly linked to their equivalents for $\spatialdimensionality = 1$. Nevertheless, the recurrence relations they satisfy are inherited from the one-dimensional case because of the construction of the prediction operator by tensor product.
    Expanding the flattened scheme as for $\spatialdimensionality = 1$ provides
    \begin{multline*}
        \sum_{s = 0}^{+\infty} \frac{\timestep^s}{s!} \partial_{t}^s \distribution^{\populationindex}(t, \vectorial{x}_{\maxlevel, \indexletter})  \\
        \shoveleft{= \distribution^{\populationindex, \collided}(t, \vectorial{x}_{\maxlevel, \indexletter}) + \sum_{s = 0}^{+\infty}\sum_{p = 0}^{+\infty} \adaptiveroundbrackets{\frac{2^{\leveldifference(s+p-2)}(\spacestep)^{s+p}}{s! ~ p!} \adaptiveroundbrackets{ \sum_{m = -2}^{+2} \sum_{n = -2}^{+2} m^s n^p  C_{\leveldifference, m, n}^{\populationindex}}  \partial_{x}^{s} \partial_y^p \distribution^{\populationindex, \collided}(t, \vectorial{x}_{\maxlevel, \indexletter})} }\\
        \shoveleft{= \adaptiveroundbrackets{1 + \frac{1}{4^{\leveldifference}} \sum_{m = -2}^{+2} \sum_{n = -2}^{+2} C_{\leveldifference, m, n}^{\populationindex}} \distribution^{\populationindex, \collided}}\\
        + \adaptiveroundbrackets{\frac{1}{2^{\leveldifference}}\sum_{m = -2}^{+2} \sum_{n = -2}^{+2} m C_{\leveldifference, m, n}^{\populationindex}} \spacestep \partial_x \distribution^{\populationindex, \collided} + \adaptiveroundbrackets{\frac{1}{2^{\leveldifference}} \sum_{m = -2}^{+2} \sum_{n = -2}^{+2} n C_{\leveldifference, m, n}^{\populationindex}} \spacestep \partial_y \distribution^{\populationindex, \collided} \\
        +  \adaptiveroundbrackets{ \sum_{m, n} m^2 C_{\leveldifference, m, n}^{\populationindex}} \frac{\spacestep^2}{2} \partial_{xx} \distribution^{\populationindex, \collided} +  \adaptiveroundbrackets{ \sum_{m,n} mn C_{\leveldifference, m, n}^{\populationindex}} \spacestep^2 \partial_{xy} \distribution^{\populationindex, \collided} +  \adaptiveroundbrackets{ \sum_{m, n} n^2 C_{\leveldifference, m, n}^{\populationindex}} \frac{\spacestep^2}{2} \partial_{yy} \distribution^{\populationindex, \collided}\\
        + \dots
    \end{multline*}
    Thus

    \begin{definition}[Match]
            Let $\spatialdimensionality = 2$. We say that the adaptive stream phase \eqref{eq:FlattenedStream2D} matches that of the reference scheme at order $s \in \mathbb{N}^{\star}$ whenever
        \begin{equation*}
            \sum_{m, n = -2}^{+2} C_{\leveldifference, m, n}^{\populationindex} = 0, \qquad \sum_{m, n = -2}^{+2} m^{p_x}n^{p_x} C_{\leveldifference, m, n}^{\populationindex} = \frac{(-\normalizedvelocityletter_{\populationindex, x})^{p_x} (-\normalizedvelocityletter_{\populationindex, y})^{p_y}}{2^{\leveldifference(p_x + p_y - 2)}}, \quad \text{for} \quad p_x, p_y = 1, \dots, s.
        \end{equation*}
        regardless of the level $\levelletter$, thus for any $\leveldifference \geq 0$ and for every velocity $\populationindex = 0, \dots, \velocitynumber-1$.
        %{\bfseries Benjamin: il faut peut-être bien préciser pour $\alpha\in ?$}
        \end{definition} 
        
    We focus on the case $\predictionstencil = 1$ and thus the matrix $\operatorial{P}$ is given by Proposition \ref{prop:CoefficientsGamma1}.
    Thanks to the construction of the prediction operator as tensorial product, introducing the following order for the coefficients $\vectorial{C}_{\leveldifference}^{\populationindex} := (C_{\leveldifference, -2, -2}^{\populationindex}, C_{\leveldifference, -1, -2}^{\populationindex}, \dots, C_{\leveldifference, 2, -2}^{\populationindex}, C_{\leveldifference, -2, -1}^{\populationindex}, \dots, C_{\leveldifference, 2, -1}^{\populationindex}, \dots, C_{\leveldifference, 2, 2}^{\populationindex})^T  \in \mathbb{R}^{25}$.
    Inside this vector, the element $C_{\leveldifference, m, n}^{\populationindex}$ has place $5(n+2)+(m+2)$ (from now on the indices of vectors and matrices start from zero).
    We obtain that
    \begin{equation*}
        \vectorial{C}_{\leveldifference}^{\populationindex} = (\operatorial{P} \otimes \operatorial{P}) \vectorial{C}_{\leveldifference - 1}^{\populationindex}.
    \end{equation*}
    Taking $\operatorial{A} \in \mathbb{R}^{m \times n}$ and  $\operatorial{B} \in \mathbb{R}^{p \times q}$, we have that $(\operatorial{A}\otimes \operatorial{B})_{i,j} = A_{i \sslash p, j \sslash q} \times B_{i \% p, j \% q}$, where $\sslash$ denotes the integer division and $\%$ the reminder of such division.
    Therefore
    \begin{equation*}
        (\operatorial{P} \otimes \operatorial{P})_{i, j} = P_{i \sslash 5, j \sslash 5} \times P_{i \% 5, j \% 5}.
    \end{equation*}
    
    We shall make use of the following Lemma
    \begin{lemma}\label{lemma:SumColumns}
        Let $j = 0, \dots, 4$, then $\sum_{m=-2}^{m=+2}m^s P_{m+2, j} = 2^{1-s}(j-2)^s$ for $s = 0, 1, 2, 3$.
    \end{lemma}
    \begin{proof}
        Use direct computations on $\operatorial{P}$ analogous to those at the end of Proposition \ref{prop:CoefficientsGamma1}.
    \end{proof}

    \begin{proposition}[Match for $\predictionstencil = 1$]
        Let $\spatialdimensionality = 2$ and $\predictionstencil = 1$, then the adaptive stream phase matches that of the reference scheme at order $s = 3$.
    \end{proposition}
    \begin{proof}
     Let $s, p = 0, \dots, 3$ and proceed recursively. For $\leveldifference = 0$ the thesis holds because we check that $\sum_{m = -2}^{m=+2} \sum_{n = -2}^{n=+2} m^s n^p C_{0, m, n}^{\populationindex} = (-\normalizedvelocityletter_{\populationindex, x})^s (-\normalizedvelocityletter_{\populationindex, y})^p$. 
     Now assume that the thesis holds for $\leveldifference - 1$, thus that 
     \begin{equation*}
        \sum_{m, n = -2}^{+2} m^s n^p C_{\leveldifference, m, n}^{\populationindex} = \frac{(-\normalizedvelocityletter_{\populationindex, x})^s (-\normalizedvelocityletter_{\populationindex, y})^p}{2^{(\leveldifference -1)(s + p - 2)}}.
     \end{equation*}
     Then
     \begin{multline*}
            \sum_{m = -2}^{+2} \sum_{n = -2}^{+2} m^s n^p C_{\leveldifference, m, n}^{\populationindex} = \sum_{m = -2}^{+2} \sum_{n = -2}^{+2} m^s n^p (\operatorial{P} \otimes \operatorial{P})_{5(n+2)+(m+2), \cdot} \cdot \vectorial{C}_{\leveldifference - 1}^{\populationindex}, \\
            = \sum_{m = -2}^{+2} \sum_{n = -2}^{+2} \sum_{j = 0}^{24} m^s n^p (\operatorial{P} \otimes \operatorial{P})_{5(n+2)+(m+2), j} \times C_{\leveldifference - 1, j \% 5 - 2, j \sslash 5 - 2}^{\populationindex}, \\
            = \sum_{m = -2}^{+2} \sum_{n = -2}^{+2} \sum_{j = 0}^{24} m^s n^p P_{[5(n+2)+(m+2)] \sslash 5, j \sslash 5} \times P_{[5(n+2)+(m+2)] \% 5, j \% 5} \times C_{\leveldifference - 1, j \% 5 - 2, j \sslash 5 - 2}^{\populationindex}.
        \end{multline*}
    We observe that for $m, n = -2, \dots, 2$ we have $[5(n+2)+(m+2)] \sslash 5 = n + 2$ and $[5(n+1)+(m+2)] \% = m+2$, so we can rewrite
    \begin{align*}
            \sum_{m = -2}^{+2} \sum_{n = -2}^{+2} m^s n^p C_{\leveldifference, m, n}^{\populationindex} &= \sum_{m = -2}^{+2} \sum_{n = -2}^{+2} \sum_{j = 0}^{24} m^s n^p P_{n+2, j \sslash 5} \times P_{m+2, j \% 5} \times C_{\leveldifference - 1, j \% 5 - 2, j \sslash 5 - 2}^{\populationindex}, \\
            &= \sum_{j = 0}^{24} C_{\leveldifference - 1, j \% 5 - 2, j \sslash 5 - 2}^{\populationindex} \times \adaptiveroundbrackets{\sum_{m = -2}^{+2} m^s P_{m+2, j \% 5}} \times  \adaptiveroundbrackets{\sum_{n = -2}^{+2} n^p P_{n+2, j \sslash 5}}, \\
            &= \sum_{j = 0}^{24} C_{\leveldifference - 1, j \% 5 - 2, j \sslash 5 - 2}^{\populationindex} \times \frac{1}{2^{s-1}} (j \% 5 - 2)^s \times  \frac{1}{2^{p-1}} (j \sslash 5 - 2)^p, \\
            &= \frac{1}{2^{s+p-2}}\sum_{j = 0}^{24} (j \% 5 - 2)^s (j \sslash 5 - 2)^p C_{\leveldifference - 1, j \% 5 - 2, j \sslash 5 - 2}^{\populationindex} , \\
            &= \frac{1}{2^{s+p-2}} \sum_{m = -2}^{2}\sum_{n = -2}^{2} m^s n^p C_{\leveldifference - 1, m, n}^{\populationindex} = \frac{(-\normalizedvelocityletter_{\populationindex, x})^s (-\normalizedvelocityletter_{\populationindex, y})^p}{2^{\leveldifference(s + p - 2)}}
        \end{align*}
        where we used Lemma \ref{lemma:SumColumns}, we performed a change of indices and used the recurrence hypothesis.
    \end{proof}
    
    \begin{remark}
        It is interesting to notice that the method constructed by tensor product matches crossed terms until order $9$ (take $s = p = 3$) but cannot match pure $x$ or $y$ terms to orders larger than $3$.
    \end{remark}
    
    To summarize, this shows how to generalize the study of the adaptive LBM-MR scheme when $\spatialdimensionality \geq 2$ to recover the same results than the study for $\spatialdimensionality = 1$. Again, for most of the applications, one needs to consider $\predictionstencil \geq 1$.

\section{Numerical simulations}\label{sec:NumericalSimulations}

    In the previous sections, the expansions we have performed were formal and valid for smooth solutions in the limit of small $\spacestep_{\levelletter}$ for any considered level $\levelletter$.
    The aim of the following numerical simulations is to assess the previous approach by showing that it provides a useful tool to \emph{a priori} study the behavior of the adaptive scheme.
    The following stream schemes are considered:
    \begin{itemize}
        \item the multiresolution scheme for $\predictionstencil = 0$ given by \eqref{eq:StreamPhase}, which matches the reference scheme up to first order;
        \item the multiresolution scheme for $\predictionstencil = 1$ given by \eqref{eq:StreamPhase}, which matches the reference scheme up to third order;
        \item the Lax-Wendroff-like scheme \eqref{eq:LaxWendroff} proposed by \cite{fakhari2014finite}, which matches the reference scheme up to second order.
    \end{itemize}

    \begin{table}[h]\caption{\label{tab:TestCasesResumee}Summary of the test cases for the analysis of the stream phase without use of mesh adaptation within the  LBM-MR-LC framework.}
        \centering
            \begin{footnotesize}
            \begin{tabular}{ccccc}
            Number and Section & $\spatialdimensionality$ & Equation & Ref. scheme & Configuration \\
            \toprule[2pt]
            1 - \ref{sec:LinearAdvection} & 1 & Linear advection eq. & D1Q2 & Fixed $\maxlevel-\minlevel$, increasing $\maxlevel$ \\
                &&$\partial_t u + \partial_x (V u) = 0$&&\\
            \midrule
            2 - \ref{sec:LinearAdvectionDiffusion} & 1 & Linear advection-diffusion eq. & D1Q3 & Fixed $\maxlevel$, decreasing $\minlevel$ \\
            &&$\partial_t u + \partial_x (V u) - \nu \partial_{xx}u = 0$&&\\
            \midrule
            3 - \ref{sec:ViscousBurgersEquation} & 1 & Viscous Burgers eq. & D1Q3 & Fixed $\maxlevel$, decreasing $\minlevel$ \\
            &&$\partial_t u + \partial_x (u^2/2) - \nu \partial_{xx}u = 0$&&\\
            \midrule
            4 - \ref{sec:2DAdvectionDiffusionEquation} & 2 & Linear advection-diffusion eq. & D2Q9 & Fixed $\maxlevel$, decreasing $\minlevel$ \\
            &&$\partial_t u + \nabla \cdot (\vectorial{V} u) - \nu \Delta u = 0$&&\\
            \bottomrule
            \end{tabular}
            \end{footnotesize}
    \end{table}
    
    A first batch of tests (spanning Sections \ref{sec:LinearAdvection}, \ref{sec:LinearAdvectionDiffusion}, \ref{sec:ViscousBurgersEquation} and \ref{sec:2DAdvectionDiffusionEquation}) is performed on uniform meshes which are corsened until reaching the lowest authorized level $\minlevel$, where the adaptive scheme is utilized. Indeed, the main focus of this work is not to evaluate the quality of the grid adaptation with respect to the parameter $\epsilon$. This has been the subject of our previous papers \cite{bellotti2021sisc} and \cite{bellotti2021multidimensional}.
    Basides the fact that this setting is relevant according to Remark \ref{rem:UniformMesh} and because the match properties hold uniformly in $\levelletter$, it also provides a worst case scenario to undoubtedly prove the resilience of our numerical strategy. This could be the case when one performs mesh adaptation yet selecting a very large threshold $\epsilon$, so to that the smoothness of the solution allows to coarsen the grid everywhere.
    As a matter of fact, similar scenarios can also take place when the mesh is updated using some stiff numerical solution (for example a phase-field variable as in \cite{fakhari2016mass}) but we still want to achieve a good accuracy in the coarsely meshed areas for the non-stiff variables (for example the velocity field in the incompressible Navier-Stokes system \cite{nguessan2021}).
    In these tests, the ``standard'' leaves collision given by \eqref{eq:CollisionPhase} is used and the study of the different collision strategies is postponed to Section \ref{sec:StudyCollisionOperator}.
    The summary of the four configurations we test is given on Table \ref{tab:TestCasesResumee}: they include both the 1D and the 2D framework with linear and non-linear equations.

    When the reference scheme is expected to convergence as $\maxlevel \to +\infty$, we consider a fixed number of coarsening steps $\maxlevel - \minlevel$ and we increase the maximal level $\maxlevel$ to observe convergence. On the other hand, when the reference scheme is not convergent to the solution of the target equation, the maximum level $\maxlevel$ is fixed and the number of coarsenings $\maxlevel - \minlevel$ is increased.
    The first kind of study aims at evaluating the possible ``interference'' of the adaptive strategy with the order of convergence of the reference scheme and precisely show that the theoretical analysis allows to study such phenomenon. On the other hand, the second kind of study tries to quantify the effect of the adaptive scheme compared to the error of the reference scheme and shows that the previous analysis allows to construct a comparative evaluation of the various methods.

    We monitor the following quantities (which are all normalized $\ell^1$ norms) at the final time $T$:
    \begin{itemize}
        \item $\text{E}_{\text{ref}}$: error of the reference scheme with respect to the exact solution of the problem. This is intrinsic to the numerical method and sometimes converging as $\spacestep \to 0$.
        \item $\text{E}_{\text{adap}}^{\minlevel}$: error of the adaptive scheme with respect to the exact solution measured at level $\minlevel$.
        \item $\text{E}_{\text{adap}}^{\maxlevel}$: error of the adaptive scheme with respect to the exact solution measured at level $\maxlevel$, using the reconstruction operator.
        \item $\text{D}_{\text{adap}}$: difference between the reference and adaptive scheme, where the adaptive datum has been reconstructed at finest level in order to compare it with the solution of the reference scheme. It is converging as $\leveldifference_{\text{min}} := \maxlevel - \minlevel \to 0$.
    \end{itemize}
    The objective is to keep $\text{D}_{\text{adap}} \ll \text{E}_{\text{ref}} $ regardless of the fact that $\text{E}_{\text{ref}}$ converges, so that the error of the adaptive scheme is largely dominated by that of the reference scheme. Generally speaking, this is the aim of the multiresolution analysis, see \cite{bellotti2021sisc}.
    By the triangle inequality, the following control on the error of the adaptive method holds $\text{E}_{\text{adap}}^{\maxlevel} \leq \text{E}_{\text{ref}} + \text{D}_{\text{adap}}$: the error of the adaptive method is the result of two contributions, the error of the reference scheme (which in principle cannot be alleviated) and the difference between the behavior of the adaptive and the reference scheme (to be alleviated by increasing $\predictionstencil$).
    
    A second kind of test, see Section \ref{sec:MeshAdaptationTest}, is the only place in the paper where we actually consider spatially dynamically adapted grids \emph{via} the multiresolution procedure.
    The aim is twofold.
    On one hand, we want to show that multiresolution is able to correctly capture the evolution of steep solutions arising from non-linear equations, even when the mesh at the initial time is heavily coarsened due to the regularity of the initial datum of the problem.
    On the other hand, we show the importance of using time-adaptive meshes once the smoothness of the solution decreases during the simulation. Indeed, we see that in this case the smoothness assumption to perform the theoretical analysis is no longer valid and using uniform coarsened meshes produces unpredictable behaviors. Multiresolution strongly alleviates the issue by following the steep zones of the solution.
    
    A third kind of test (Section \ref{sec:StudyCollisionOperator}) comes back to uniform coarsened meshes and its scope is to compare the different collision strategies presented in Section \ref{sec:DifferentCollisionStrategies} in a non-linear context, which is the one where one can potentially observe discrepancies between different strategies.

    \begin{remark}[Number of conserved moments]
    For the sake of this work, all the schemes have only one conserved variable $\average{u} = \sum_{\populationindex} \average{\distribution}^{\populationindex}$, but all the previous study is independent of the number of conserved variables, since it pertains to the stream phase of the method. 
    \end{remark}

    \FloatBarrier

    \subsection{1D Linear advection equation}\label{sec:LinearAdvection}
        The aim of this test case is to validate our analysis in a case where, on one hand, we know that the reference scheme converges as $\spacestep \to 0$ because the relaxation parameters do not depend on $\spacestep$ and, on the other hand, the equilibria are linear thus the collision strategy does not alter the quality of the method.
        We expect that all the tested methods match enough terms in order to maintain the convergence of the reference method. However, two interwoven phenomena can take place. The first is a modification of the convergence rate because of $\text{D}_{\text{adap}}$. Second, at some point, the term $\text{D}_{\text{adap}}$ can become non-negligible with respect to $\text{E}_{\text{ref}}$.
        
    \subsubsection{The problem and the reference scheme}

    To this end, consider the linear advection equation with constant velocity $V\in \mathbb{R}$ (see Table \ref{tab:TestCasesResumee}) with solution
    \begin{equation}\label{eq:LinearAdvectionEquation}
        u(t, x) = \frac{1}{(4 \pi \nu t_0)^{1/2}} \text{exp} \adaptiveroundbrackets{-\frac{|x- V t |^2}{4\nu t_0}}.
    \end{equation}
    We simulate using the following parameters: $T = 2$, $t_0 = 1$, $V = 0.5$, $\nu = 0.005$. 
    For the sake of the computation, we consider a bounded domain $[-3, 3]$ with copy boundary conditions.
    The numerical scheme we adapt is a D1Q2 with velocities $\normalizedvelocityletter_0 = 1$ and $\normalizedvelocityletter_1 = -1$ with change of basis and relaxation matrix given by
    
    \begin{equation*}
        \operatorial{M} = \adaptiveroundbrackets{\begin{matrix}
                                                   1 & 1 \\
                                                   \latticevelocity & -\latticevelocity \\
                                                 \end{matrix}}, \qquad
        \operatorial{S} = \text{diag}(0, s).
    \end{equation*}
    Taking $m^{1, \text{eq}} = V m^0$, it has been theoretically shown \cite{dellacherie2014construction} (conditionally for the $\ell^{\infty}$ norm and unconditionally for the $\ell^2$ norm) and numerically verified  \cite{graille2014approximation} that the scheme converges linearly towards the solution \eqref{eq:LinearAdvectionEquation} for $s \in ]0, 2[$ and quadratically for $s = 2$ as $\spacestep \to 0$. In this test, we take $\latticevelocity = 1$.

    \subsubsection{Results and discussion}

    The test is conducted fixing the difference between the maximum level $\maxlevel$ and the minimum level $\minlevel$, at which we perform the computation, either at $\leveldifference_{\text{min}} = 2$ or $\leveldifference_{\text{min}} = 6$ for two different relaxation parameters, namely $s = 1$ and $s = 2$.
    We progressively increase the maximum level $\maxlevel$ to observe the convergence of the reference scheme towards the solution and to study how the adaptive schemes behave in such a situation.
    
    \begin{figure}
       \begin{center}
            \includegraphics[width=1.\textwidth]{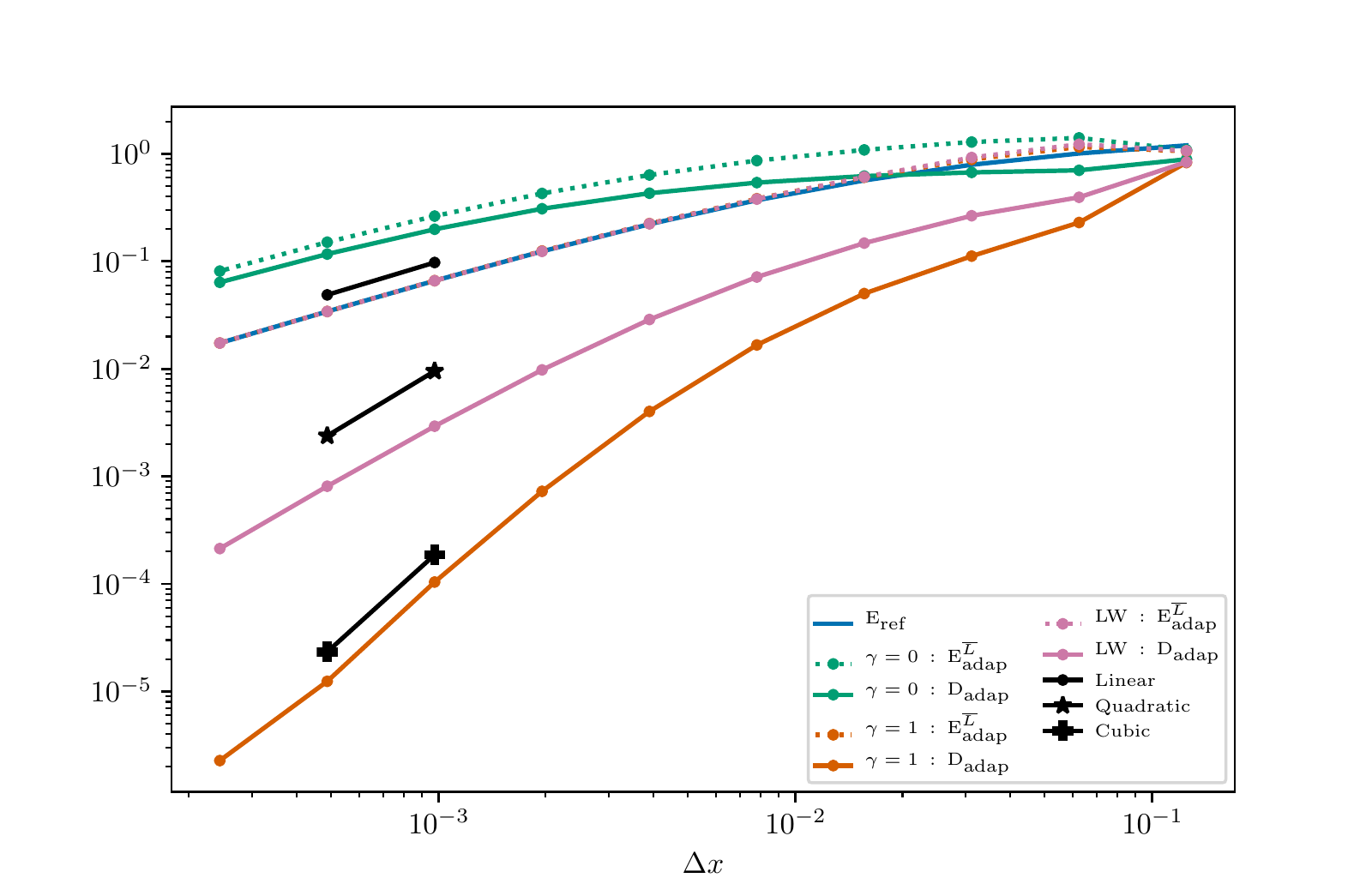}
        \end{center}\caption{\label{fig:d1q2_linear_advection-1}Results for the test 1 with the 1D Linear advection equation taking $\leveldifference_{\text{min}} = 2$ and $s = 1$.}
    \end{figure} 
    
    \begin{figure}
       \begin{center}
            \includegraphics[width=1.\textwidth]{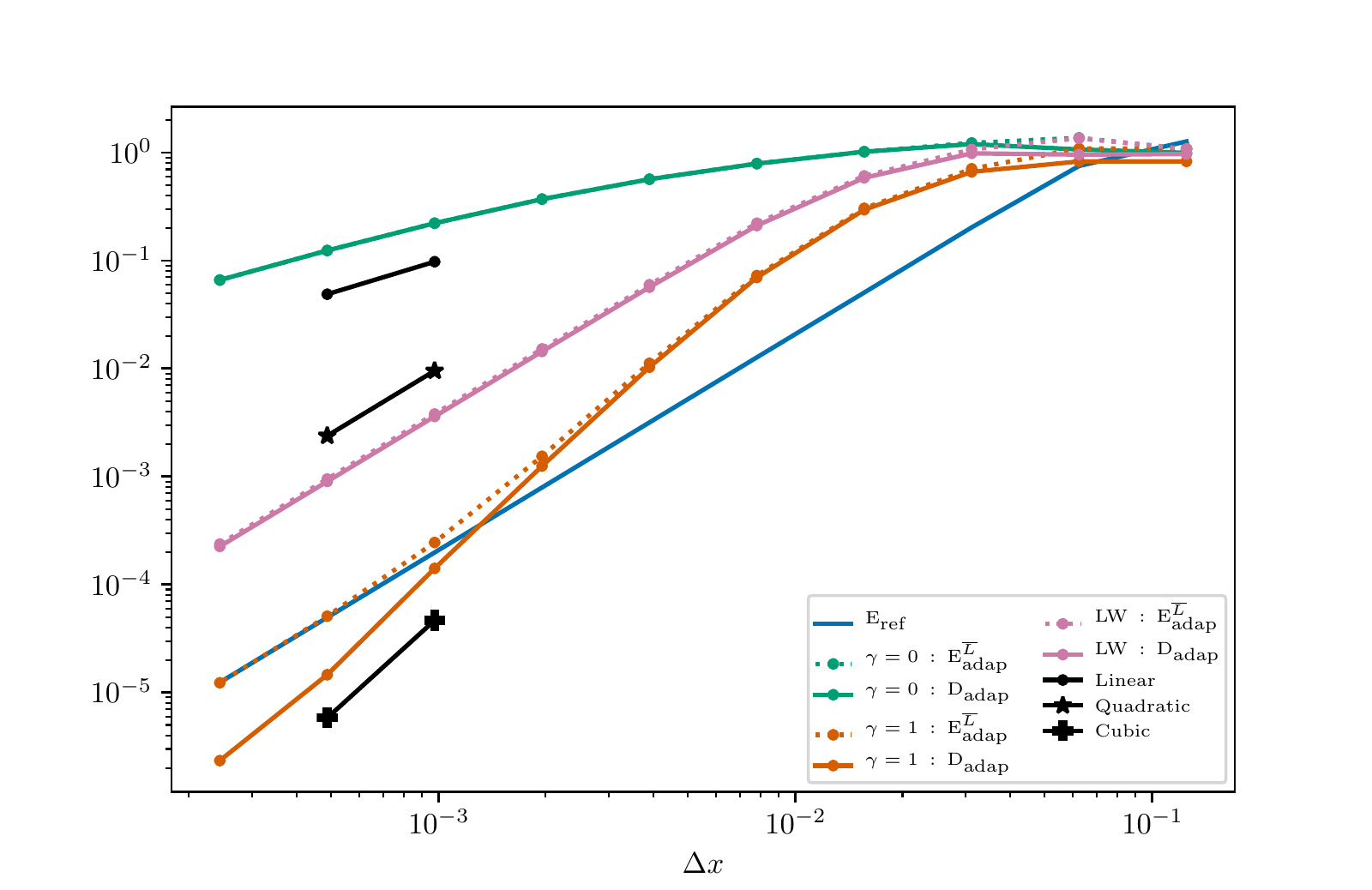}
        \end{center}\caption{\label{fig:d1q2_linear_advection-2}Results for the test 1 with the 1D Linear advection equation taking $\leveldifference_{\text{min}} = 2$ and $s = 2$.}
    \end{figure}

\begin{table}\caption{\label{tab:LinearAdvection_s_1_dL_2}Test 1 for the 1D Linear advection equation taking $\leveldifference_{\text{min}} = 2$ and $s = 1$.}
    \centering
        \begin{tiny}
        \begin{tabular}{ccccccccccc}%{|c|c|ccc|ccc|ccc|}
              & & \multicolumn{3}{|c|}{Haar $\predictionstencil = 0$} & \multicolumn{3}{|c|}{$\predictionstencil = 1$} & \multicolumn{3}{|c}{Lax-Wendroff} \\
            \toprule[2pt]
            $\maxlevel$ & $\text{E}_{\text{ref}}$ & $\text{E}_{\text{adap}}^{\minlevel}$ & $\text{E}_{\text{adap}}^{\maxlevel}$ & $\text{D}_{\text{adap}}$ & $\text{E}_{\text{adap}}^{\minlevel}$ & $\text{E}_{\text{adap}}^{\maxlevel}$ & $\text{D}_{\text{adap}}$ & $\text{E}_{\text{adap}}^{\minlevel}$ & $\text{E}_{\text{adap}}^{\maxlevel}$ & $\text{D}_{\text{adap}}$\\
            \midrule
3	&  \texttt{1.20e+0} \phantom{(0.25)}	& \texttt{1.21e+0}	& \texttt{1.09e+0}	& \texttt{8.92e-1} \phantom{(0.25)}	    & \texttt{8.41e-1}	& \texttt{1.06e+0}	& \texttt{8.31e-1} \phantom{(0.25)}	    & \texttt{8.58e-1}	& \texttt{1.07e+0}	& \texttt{8.40e-1} \phantom{(0.25)}	       \\      
4	& \texttt{1.01e+0} 	         (0.25)		& \texttt{1.43e+0}	& \texttt{1.41e+0}	& \texttt{7.04e-1}	        (0.34)		& \texttt{1.07e+0}	& \texttt{1.16e+0}	& \texttt{2.30e-1}	        (1.85)		& \texttt{1.13e+0}	& \texttt{1.23e+0}	& \texttt{3.95e-1}	        (1.09) \\
5	& \texttt{7.93e-1} 	         (0.35)		& \texttt{1.24e+0}	& \texttt{1.29e+0}	& \texttt{6.73e-1}	        (0.06)		& \texttt{9.27e-1}	& \texttt{8.89e-1}	& \texttt{1.12e-1}	        (1.04)		& \texttt{9.55e-1}	& \texttt{9.26e-1}	& \texttt{2.66e-1}	        (0.57) \\
6	& \texttt{5.67e-1} 	         (0.48)		& \texttt{1.08e+0}	& \texttt{1.09e+0}	& \texttt{6.22e-1}	        (0.11)		& \texttt{6.21e-1}	& \texttt{6.07e-1}	& \texttt{5.02e-2}	        (1.16)		& \texttt{6.19e-1}	& \texttt{6.11e-1}	& \texttt{1.48e-1}	        (0.84) \\
7	& \texttt{3.71e-1} 	         (0.61)		& \texttt{8.68e-1}	& \texttt{8.66e-1}	& \texttt{5.41e-1}	        (0.20)		& \texttt{3.86e-1}	& \texttt{3.83e-1}	& \texttt{1.67e-2}	        (1.59)		& \texttt{3.81e-1}	& \texttt{3.80e-1}	& \texttt{7.16e-2}	        (1.05) \\
8	& \texttt{2.22e-1} 	         (0.74)		& \texttt{6.37e-1}	& \texttt{6.37e-1}	& \texttt{4.31e-1}	        (0.33)		& \texttt{2.26e-1}	& \texttt{2.25e-1}	& \texttt{4.02e-3}	        (2.05)		& \texttt{2.24e-1}	& \texttt{2.24e-1}	& \texttt{2.88e-2}	        (1.31) \\
9	& \texttt{1.24e-1} 	         (0.84)		& \texttt{4.29e-1}	& \texttt{4.29e-1}	& \texttt{3.09e-1}	        (0.48)		& \texttt{1.25e-1}	& \texttt{1.25e-1}	& \texttt{7.26e-4}	        (2.47)		& \texttt{1.25e-1}	& \texttt{1.24e-1}	& \texttt{9.79e-3}	        (1.56) \\
10	& \texttt{6.61e-2} 	         (0.91)		& \texttt{2.64e-1}	& \texttt{2.64e-1}	& \texttt{1.99e-1}	        (0.64)		& \texttt{6.62e-2}	& \texttt{6.62e-2}	& \texttt{1.04e-4}	        (2.80)		& \texttt{6.62e-2}	& \texttt{6.61e-2}	& \texttt{2.93e-3}	        (1.74) \\
11	& \texttt{3.42e-2} 	         (0.95)		& \texttt{1.51e-1}	& \texttt{1.51e-1}	& \texttt{1.17e-1}	        (0.77)		& \texttt{3.42e-2}	& \texttt{3.42e-2}	& \texttt{1.24e-5}	        (3.06)		& \texttt{3.42e-2}	& \texttt{3.42e-2}	& \texttt{8.10e-4}	        (1.86) \\
12	& \texttt{1.74e-2} 	         (0.97)		& \texttt{8.13e-2}	& \texttt{8.13e-2}	& \texttt{6.39e-2}	        (0.87)		& \texttt{1.74e-2}	& \texttt{1.74e-2}	& \texttt{2.27e-6}	        (2.46)		& \texttt{1.74e-2}	& \texttt{1.74e-2}	& \texttt{2.13e-4}	        (1.92) \\
            \bottomrule
        \end{tabular}
        \end{tiny}
\end{table}

            \begin{table}\caption{\label{tab:LinearAdvection_s_2_dL_2}Test 1 for the 1D Linear advection equation taking $\leveldifference_{\text{min}} = 2$ and $s = 2$.}
        \begin{center}
            \begin{tiny}
            \begin{tabular}{ccccccccccc}
              & & \multicolumn{3}{|c|}{Haar $\predictionstencil = 0$} & \multicolumn{3}{|c|}{$\predictionstencil = 1$} & \multicolumn{3}{|c}{Lax-Wendroff} \\
            \toprule[2pt]
            $\maxlevel$ & $\text{E}_{\text{ref}}$ & $\text{E}_{\text{adap}}^{\minlevel}$ & $\text{E}_{\text{adap}}^{\maxlevel}$ & $\text{D}_{\text{adap}}$ & $\text{E}_{\text{adap}}^{\minlevel}$ & $\text{E}_{\text{adap}}^{\maxlevel}$ & $\text{D}_{\text{adap}}$ & $\text{E}_{\text{adap}}^{\minlevel}$ & $\text{E}_{\text{adap}}^{\maxlevel}$ & $\text{D}_{\text{adap}}$\\
            \midrule
3	& \texttt{1.27e+0} \phantom{(0.75)}		& \texttt{1.15e+0}	& \texttt{1.08e+0}	& \texttt{9.96e-1} \phantom{(-0.16)}			    & \texttt{9.20e-1}	& \texttt{1.09e+0}	& \texttt{8.29e-1} \phantom{(0.75)}	    & \texttt{8.56e-1}	& \texttt{1.08e+0}	& \texttt{9.70e-1}	\phantom{(-0.04)}        \\     
4	& \texttt{7.53e-1}	        (0.75)		& \texttt{1.37e+0}	& \texttt{1.37e+0}	& \texttt{1.07e+0}	        (-0.10)		            & \texttt{9.20e-1}	& \texttt{1.09e+0}	& \texttt{8.29e-1}	        (0.00)		& \texttt{1.22e+0}	& \texttt{1.36e+0}	& \texttt{9.59e-1}	(\phantom{-}0.02) \\
5	& \texttt{2.03e-1}	        (1.89)		& \texttt{1.20e+0}	& \texttt{1.23e+0}	& \texttt{1.20e+0}	        (-0.16)		            & \texttt{7.54e-1}	& \texttt{7.09e-1}	& \texttt{6.64e-1}	        (0.32)		& \texttt{1.09e+0}	& \texttt{1.07e+0}	& \texttt{9.85e-1}	(-0.04) \\
6	& \texttt{5.06e-2}	        (2.01)		& \texttt{1.01e+0}	& \texttt{1.02e+0}	& \texttt{1.02e+0}	        (\phantom{-}0.24)		& \texttt{3.00e-1}	& \texttt{3.05e-1}	& \texttt{2.96e-1}	        (1.17)		& \texttt{6.13e-1}	& \texttt{6.09e-1}	& \texttt{5.84e-1}	(\phantom{-}0.75) \\
7	& \texttt{1.27e-2}	        (2.00)		& \texttt{7.91e-1}	& \texttt{7.92e-1}	& \texttt{7.91e-1}	        (\phantom{-}0.36)		& \texttt{7.48e-2}	& \texttt{7.27e-2}	& \texttt{6.98e-2}	        (2.09)		& \texttt{2.23e-1}	& \texttt{2.22e-1}	& \texttt{2.11e-1}	(\phantom{-}1.47) \\
8	& \texttt{3.17e-3}	        (2.00)		& \texttt{5.67e-1}	& \texttt{5.67e-1}	& \texttt{5.67e-1}	        (\phantom{-}0.48)		& \texttt{1.17e-2}	& \texttt{1.12e-2}	& \texttt{1.03e-2}	        (2.77)		& \texttt{5.99e-2}	& \texttt{5.98e-2}	& \texttt{5.67e-2}	(\phantom{-}1.90) \\
9	& \texttt{7.92e-4}	        (2.00)		& \texttt{3.71e-1}	& \texttt{3.71e-1}	& \texttt{3.71e-1}	        (\phantom{-}0.61)		& \texttt{1.68e-3}	& \texttt{1.54e-3}	& \texttt{1.25e-3}	        (3.03)		& \texttt{1.52e-2}	& \texttt{1.52e-2}	& \texttt{1.44e-2}	(\phantom{-}1.98) \\
10	& \texttt{1.98e-4}	        (2.00)		& \texttt{2.22e-1}	& \texttt{2.22e-1}	& \texttt{2.22e-1}	        (\phantom{-}0.74)		& \texttt{2.75e-4}	& \texttt{2.45e-4}	& \texttt{1.41e-4}	        (3.15)		& \texttt{3.80e-3}	& \texttt{3.80e-3}	& \texttt{3.60e-3}	(\phantom{-}2.00) \\
11	& \texttt{4.95e-5}	        (2.00)		& \texttt{1.24e-1}	& \texttt{1.24e-1}	& \texttt{1.24e-1}	        (\phantom{-}0.84)		& \texttt{5.57e-5}	& \texttt{5.09e-5}	& \texttt{1.46e-5}	        (3.27)		& \texttt{9.49e-4}	& \texttt{9.49e-4}	& \texttt{9.00e-4}	(\phantom{-}2.00) \\
12	& \texttt{1.24e-5}	        (2.00)		& \texttt{6.61e-2}	& \texttt{6.61e-2}	& \texttt{6.61e-2}	        (\phantom{-}0.91)		& \texttt{1.29e-5}	& \texttt{1.23e-5}	& \texttt{2.34e-6}	        (2.64)		& \texttt{2.37e-4}	& \texttt{2.37e-4}	& \texttt{2.25e-4}	(\phantom{-}2.00) \\
                \hline
            \end{tabular}
            \end{tiny}
        \end{center}
    \end{table}
    
\begin{table}\caption{\label{tab:LinearAdvection_s_1_dL_6}Test 1 for the 1D Linear advection equation taking $\leveldifference_{\text{min}} = 6$ and $s = 1$.}
        \begin{center}
            \begin{tiny}
            \begin{tabular}{ccccccccccc}
              & & \multicolumn{3}{|c|}{Haar $\predictionstencil = 0$} & \multicolumn{3}{|c|}{$\predictionstencil = 1$} & \multicolumn{3}{|c}{Lax-Wendroff} \\
            \toprule[2pt]
            $\maxlevel$ & $\text{E}_{\text{ref}}$ & $\text{E}_{\text{adap}}^{\minlevel}$ & $\text{E}_{\text{adap}}^{\maxlevel}$ & $\text{D}_{\text{adap}}$ & $\text{E}_{\text{adap}}^{\minlevel}$ & $\text{E}_{\text{adap}}^{\maxlevel}$ & $\text{D}_{\text{adap}}$ & $\text{E}_{\text{adap}}^{\minlevel}$ & $\text{E}_{\text{adap}}^{\maxlevel}$ & $\text{D}_{\text{adap}}$\\
            \midrule
7	& \texttt{3.71e-1}	\phantom{(0.74)}	& \texttt{1.20e+0}	& \texttt{1.08e+0}	& \texttt{1.05e+0}  \phantom{(-0.32)}		& \texttt{7.52e-1}	& \texttt{1.07e+0}	& \texttt{1.00e+0}	\phantom{(-0.12)}		& \texttt{1.22e+0}	& \texttt{1.14e+0}	& \texttt{1.08e+0}  \phantom{(-0.73)}        \\ 	
8	& \texttt{2.22e-1}	         (0.74)		& \texttt{1.44e+0}	& \texttt{1.41e+0}	& \texttt{1.31e+0}	          (-0.32)		& \texttt{1.03e+0}	& \texttt{1.28e+0}	& \texttt{1.09e+0}	          (-0.12)		& \texttt{1.83e+0}	& \texttt{1.95e+0}	& \texttt{1.79e+0}	          (-0.73) \\
9	& \texttt{1.24e-1}	         (0.84)		& \texttt{1.25e+0}	& \texttt{1.30e+0}	& \texttt{1.23e+0}	(\phantom{-}0.10)		& \texttt{7.71e-1}	& \texttt{7.37e-1}	& \texttt{6.23e-1}	(\phantom{-}0.81)		& \texttt{1.62e+0}	& \texttt{1.65e+0}	& \texttt{1.55e+0}	(\phantom{-}0.21) \\
10	& \texttt{6.61e-2}	         (0.91)		& \texttt{1.10e+0}	& \texttt{1.11e+0}	& \texttt{1.06e+0}	(\phantom{-}0.21)		& \texttt{2.12e-1}	& \texttt{2.03e-1}	& \texttt{1.53e-1}	(\phantom{-}2.03)		& \texttt{7.59e-1}	& \texttt{7.61e-1}	& \texttt{7.13e-1}	(\phantom{-}1.12) \\
11	& \texttt{3.42e-2}	         (0.95)		& \texttt{8.88e-1}	& \texttt{8.85e-1}	& \texttt{8.58e-1}	(\phantom{-}0.30)		& \texttt{4.71e-2}	& \texttt{4.47e-2}	& \texttt{1.89e-2}	(\phantom{-}3.01)		& \texttt{2.34e-1}	& \texttt{2.33e-1}	& \texttt{2.18e-1}	(\phantom{-}1.71) \\
12	& \texttt{1.74e-2}	         (0.97)		& \texttt{6.57e-1}	& \texttt{6.57e-1}	& \texttt{6.41e-1}	(\phantom{-}0.42)		& \texttt{1.90e-2}	& \texttt{1.80e-2}	& \texttt{1.94e-3}	(\phantom{-}3.28)		& \texttt{6.17e-2}	& \texttt{6.16e-2}	& \texttt{5.79e-2}	(\phantom{-}1.91) \\
                \hline
            \end{tabular}
            \end{tiny}
        \end{center}
    \end{table}

    \begin{table}\caption{\label{tab:LinearAdvection_s_2_dL_6}Test 1 for the 1D Linear advection equation  taking $\leveldifference_{\text{min}} = 6$ and $s = 2$.}
        \begin{center}
            \begin{tiny}
            \begin{tabular}{ccccccccccc}
              & & \multicolumn{3}{|c|}{Haar $\predictionstencil = 0$} & \multicolumn{3}{|c|}{$\predictionstencil = 1$} & \multicolumn{3}{|c}{Lax-Wendroff} \\
            \toprule[2pt]
            $\maxlevel$ & $\text{E}_{\text{ref}}$ & $\text{E}_{\text{adap}}^{\minlevel}$ & $\text{E}_{\text{adap}}^{\maxlevel}$ & $\text{D}_{\text{adap}}$ & $\text{E}_{\text{adap}}^{\minlevel}$ & $\text{E}_{\text{adap}}^{\maxlevel}$ & $\text{D}_{\text{adap}}$ & $\text{E}_{\text{adap}}^{\minlevel}$ & $\text{E}_{\text{adap}}^{\maxlevel}$ & $\text{D}_{\text{adap}}$\\
            \midrule
7	& \texttt{1.27e-2} \phantom{(2.00)}		& \texttt{1.20e+0}	& \texttt{1.08e+0}	& \texttt{1.08e+0}	 \phantom{(-0.38)}		& \texttt{7.47e-1}	& \texttt{1.07e+0}	& \texttt{1.07e+0}	\phantom{(-0.29)}		& \texttt{1.22e+0}	& \texttt{1.14e+0}	& \texttt{1.14e+0}	\phantom{(-0.80)}        \\ 
8	& \texttt{3.17e-3}	        (2.00)		& \texttt{1.44e+0}	& \texttt{1.41e+0}	& \texttt{1.41e+0}	          (-0.38)		& \texttt{1.05e+0}	& \texttt{1.31e+0}	& \texttt{1.31e+0}	          (-0.29)		& \texttt{1.86e+0}	& \texttt{1.98e+0}	& \texttt{1.98e+0}              (-0.80) \\
9	& \texttt{7.92e-4}	        (2.00)		& \texttt{1.25e+0}	& \texttt{1.30e+0}	& \texttt{1.30e+0}	(\phantom{-}0.12)		& \texttt{7.84e-1}	& \texttt{7.65e-1}	& \texttt{7.65e-1}	(\phantom{-}0.77)		& \texttt{1.68e+0}	& \texttt{1.72e+0}	& \texttt{1.72e+0}	(\phantom{-}0.20)  \\ 
10	& \texttt{1.98e-4}	        (2.00)		& \texttt{1.10e+0}	& \texttt{1.10e+0}	& \texttt{1.10e+0}	(\phantom{-}0.23)		& \texttt{2.00e-1}	& \texttt{1.94e-1}	& \texttt{1.94e-1}	(\phantom{-}1.98)		& \texttt{8.05e-1}	& \texttt{8.07e-1}	& \texttt{8.07e-1}	(\phantom{-}1.09)  \\ 
11	& \texttt{4.95e-5}	        (2.00)		& \texttt{8.84e-1}	& \texttt{8.82e-1}	& \texttt{8.82e-1}	(\phantom{-}0.32)		& \texttt{2.33e-2}	& \texttt{2.22e-2}	& \texttt{2.22e-2}	(\phantom{-}3.13)		& \texttt{2.41e-1}	& \texttt{2.41e-1}	& \texttt{2.41e-1}	(\phantom{-}1.74)  \\ 
12	& \texttt{1.24e-5}	        (2.00)		& \texttt{6.53e-1}	& \texttt{6.53e-1}	& \texttt{6.53e-1}	(\phantom{-}0.43)		& \texttt{2.57e-3}	& \texttt{2.12e-3}	& \texttt{2.12e-3}	(\phantom{-}3.39)		& \texttt{6.09e-2}	& \texttt{6.11e-2}	& \texttt{6.11e-2}	(\phantom{-}1.98)  \\ 
                \hline
            \end{tabular}
            \end{tiny}
        \end{center}
    \end{table}
    
    The results are provided on Table \ref{tab:LinearAdvection_s_1_dL_2} and \ref{tab:LinearAdvection_s_2_dL_2} for $\leveldifference_{\text{min}} = 2$ and on Tables \ref{tab:LinearAdvection_s_1_dL_6} and \ref{tab:LinearAdvection_s_2_dL_6} for $\leveldifference_{\text{min}} = 6$. Moreover, they are also presented in a more effective way on Figures \ref{fig:d1q2_linear_advection-1} and \ref{fig:d1q2_linear_advection-2}.
    In all of them, we observe the expected behavior of the reference scheme, converging linearly for $s = 1$ and quadratically for $s = 2$.
    Let us comment on the behavior of each adaptive strategy:
    \begin{itemize}
        \item \textbf{Multiresolution scheme for $\predictionstencil = 0$}. In the case where the reference method is first-order convergent ($s = 1$), we observe that $\text{E}_{\text{adap}}^{\maxlevel}$ is also first-order convergent but, especially for $\leveldifference_{\text{min}} = 6$, $\text{D}_{\text{adap}}$ dominates against $\text{E}_{\text{ref}}$, which is the reason why the green full line and the green dashed lines are not superposed in Figure \ref{fig:d1q2_linear_advection-1}.
        Moreover, this is the reason why for $s = 2$, despite the fact that $\text{E}_{\text{ref}}$ converges quadratically, $\text{E}_{\text{adap}}^{\maxlevel}$ converges only linearly due to the limitations imposed by the convergence ratio of $\text{D}_{\text{adap}}$.
        %$\text{E}_{\text{adap}}^{\maxlevel} \leq \text{E}_{\text{ref}} + \text{D}_{\text{adap}}.$
        \item \textbf{Multiresolution scheme for $\predictionstencil = 1$}. We observe that in any case the convergence rate of $\text{D}_{\text{adap}}$ is third-order. Therefore, we always obtain the same convergence rate of  $\text{E}_{\text{ref}}$ for $\text{E}_{\text{adap}}^{\maxlevel}$.
        \item \textbf{Lax-Wendroff scheme}. Since the convergence rate of $\text{D}_{\text{adap}}$ is second-order, we always observe no alteration of the convergence rate of the reference scheme by the adaptive method.
    \end{itemize}
    Finally, observe that for any method the convergence rates of $\text{D}_{\text{adap}}$ are erratic for small $\maxlevel$ allegedly because Taylor expansions are not fully legitimate for coarse resolutions.
    These numerical experiments confirm the theoretical study and show that the adaptive method should match enough terms to avoid alterations of the convergence rates of the reference scheme.
    In particular, if the reference method converges at order $s$, the adaptive stream phase must match at least at order $s$ as well.

    %\FloatBarrier

    \subsection{1D Linear advection-diffusion equation}\label{sec:LinearAdvectionDiffusion}
    
    In this test case, the choice of collision phase still does not play any role because of the linearity.
    However, the reference method is not longer convergent because the relaxation parameters are adjusted to recover the right diffusion terms.
    Despite this lack of convergence, the aim of the simulations is to show that the LBM-MR scheme for $\predictionstencil = 0$ is not accurate enough to reproduce the physics of the reference algorithm.
    Moreover, we shall observe that the Lax-Wendroff scheme correctly accounts for the diffusion terms but can introduce spurious oscillations due to the lack of matching of the third-order terms. This is not the case for the LBM-MR scheme with $\predictionstencil = 1$.
    
    \subsubsection{The problem and the reference scheme}
    
    Consider the solution of the linear advection-diffusion equation (see Table \ref{tab:TestCasesResumee}) where $V \in \mathbb{R}$ is the advection velocity and $\nu > 0$ is the diffusion coefficient
    \begin{equation}\label{eq:LinearAdvectionDiffusionEquation}
        u(t, x) = \frac{1}{(4 \pi \nu (t_0 + t))^{1/2}} \text{exp} \adaptiveroundbrackets{-\frac{|x - Vt|^2}{4\nu (t_0 + t)}}.
    \end{equation}
    The parameters of the problem and the computational domain are the same than in the previous Section \ref{sec:LinearAdvection}.
    We consider a D1Q3 scheme with velocities $\normalizedvelocityletter_0 = 0, \normalizedvelocityletter_1 = 1$ and $\normalizedvelocityletter_2 = -1$ with change of basis and relaxation matrix given by
    \begin{equation*}
        \operatorial{M} = \adaptiveroundbrackets{\begin{matrix}
                                                   1 & 1 & 1 \\
                                                   0 & \latticevelocity & -\latticevelocity \\
                                                   0 & \latticevelocity^2/2 & \latticevelocity^2/2 \\
                                                 \end{matrix}}, \qquad
        \operatorial{S} = \text{diag}(0, s_v, s_w).
    \end{equation*}

    The scheme is determined by $\latticevelocity = 1$, the equilibria $\momentum^{1, \text{eq}} = V\momentum^0$ and $ \momentum^{2, \text{eq}} = \kappa \momentum^0$, with $\kappa = 0.5$ and $s_{v} = (1/2 + \lambda \nu /(\spacestep (2\kappa - V^2)))^{-1}$ and $s_w = 1$.
    %For the sake of the computation, we consider a bounded domain $[-3, 3]$.
    It is in general not convergent, because we are making $s_v \to 0$ as $\spacestep \to 0$ in order to obtain the right diffusion structure.

    \subsubsection{Results and discussion}
    
    Since we are no longer interested in convergence, we take a fixed $\maxlevel = 11$ and we vary the minimum level at which we perform computations.
    The results are given on Table \ref{tab:LinearAdvectionDiffusion} and the solution at final time $T$ for some $\leveldifference_{\text{min}}$ is shown on Figure \ref{fig:linear_advection_diffusion}.
    We observe that:
    \begin{itemize}
        \item \textbf{Multiresolution scheme for $\predictionstencil = 0$}. The inertial term (advection) is correctly represented in accordance with the theoretical analysis for any $\leveldifference_{\text{min}}$, because the packet is transported at the right velocity until reaching the point $x = 1$. 
        Nevertheless, the dissipative (diffusion) term is not correct because we have an excess of diffusion as long as $\leveldifference_{\text{min}}$ grows. This was predicted by the theoretical framework.
        Even if the dispersive term is not matched either, this does not affect the stability of the method because of the large amount of available numerical diffusion.
        As expected, $\text{D}_{\text{adap}}$ is not negligible compared to $\text{E}_{\text{ref}}$ for any $\leveldifference_{\text{min}} \geq 1$.

        \item \textbf{Multiresolution scheme for $\predictionstencil = 1$}. We observe that both the inertial (advection) and dissipative (diffusion) term are correctly represented.
        Moreover, according to our intuition, since the dispersive terms are untouched compared to the target expansion \eqref{eq:TaylorExpansion}, the adaptive method remains stable when the reference method is stable.
        Looking at the errors more carefully, we see that the additional error $\text{D}_{\text{adap}}$ is negligible compared to $\text{E}_{\text{ref}}$ for $\leveldifference_{\text{min}} < 6$ or $7$.
        \item \textbf{Lax-Wendroff scheme}. It correctly matches the inertial and dissipative phenomena. 
        Nevertheless, as $\leveldifference_{\text{min}}$ grows, we observe the formation of spurious oscillations and the packet is not perfectly centered at $x = 1$. This is allegedly due to the modification of the third-order dispersion as theoretically observed.
        For the tested case, this is not enough to induce instabilities.
        This shows that this adaptive method can be subjected to instabilities even when the reference method is stable because of the modification of the dispersion.
    The additional error $\text{D}_{\text{adap}}$ is negligible compared to $\text{E}_{\text{ref}}$ for $\leveldifference_{\text{min}} < 3$ or $4$.
    \end{itemize}

    \begin{table}\caption{\label{tab:LinearAdvectionDiffusion}Test 2 for the 1D Linear advection diffusion equation taking $\maxlevel = 11$ and performing the computation using a mesh at level $\minlevel$ as indicated.}
        \begin{center}
            \begin{footnotesize}
            \begin{tabular}{cccccccccc}
                  & \multicolumn{3}{|c|}{Haar $\predictionstencil = 0$} & \multicolumn{3}{|c|}{$\predictionstencil = 1$} & \multicolumn{3}{|c}{Lax-Wendroff} \\
            \toprule[2pt]
                $\leveldifference_{\text{min}}$ & $\text{E}_{\text{adap}}^{\minlevel}$ & $\text{E}_{\text{adap}}^{\maxlevel}$ & $\text{D}_{\text{adap}}$ & $\text{E}_{\text{adap}}^{\minlevel}$ & $\text{E}_{\text{adap}}^{\maxlevel}$ & $\text{D}_{\text{adap}}$ & $\text{E}_{\text{adap}}^{\minlevel}$ & $\text{E}_{\text{adap}}^{\maxlevel}$ & $\text{D}_{\text{adap}}$\\
            \midrule
                0 &	\texttt{1.94e-2} &	\texttt{1.94e-2} &	\texttt{0.00e+0} &	\texttt{1.94e-2} &	\texttt{1.94e-2} &	\texttt{0.00e+0} &	\texttt{1.94e-2} &	\texttt{1.94e-2} &	\texttt{0.00e+0} \\
                1 &	\texttt{2.30e-2} &	\texttt{2.30e-2} &	\texttt{1.55e-2} &	\texttt{1.94e-2} &	\texttt{1.94e-2} &	\texttt{7.88e-7} &	\texttt{1.94e-2} &	\texttt{1.94e-2} &	\texttt{3.63e-5} \\
                2 &	\texttt{4.68e-2} &	\texttt{4.68e-2} &	\texttt{4.52e-2} &	\texttt{1.94e-2} &	\texttt{1.94e-2} &	\texttt{3.41e-6} &	\texttt{1.92e-2} &	\texttt{1.92e-2} &	\texttt{1.82e-4} \\
                3 &	\texttt{9.92e-2} &	\texttt{9.92e-2} &	\texttt{9.94e-2} &	\texttt{1.94e-2} &	\texttt{1.94e-2} &	\texttt{1.31e-5} &	\texttt{1.87e-2} &	\texttt{1.87e-2} &	\texttt{7.63e-4} \\
                4 &	\texttt{1.91e-1} &	\texttt{1.91e-1} &	\texttt{1.92e-1} &	\texttt{1.94e-2} &	\texttt{1.94e-2} &	\texttt{5.40e-5} &	\texttt{1.65e-2} &	\texttt{1.65e-2} &	\texttt{3.09e-3} \\
                5 &	\texttt{3.33e-1} &	\texttt{3.33e-1} &	\texttt{3.34e-1} &	\texttt{1.93e-2} &	\texttt{1.93e-2} &	\texttt{2.78e-4} &	\texttt{8.18e-3} &	\texttt{8.32e-3} &	\texttt{1.24e-2} \\
                6 &	\texttt{5.25e-1} &	\texttt{5.24e-1} &	\texttt{5.26e-1} &	\texttt{1.84e-2} &	\texttt{1.84e-2} &	\texttt{1.74e-3} &	\texttt{3.11e-2} &	\texttt{3.16e-2} &	\texttt{5.03e-2} \\
                7 &	\texttt{7.51e-1} &	\texttt{7.47e-1} &	\texttt{7.48e-1} &	\texttt{1.10e-2} &	\texttt{1.07e-2} &	\texttt{1.89e-2} &	\texttt{1.93e-1} &	\texttt{1.96e-1} &	\texttt{2.15e-1} \\
                \hline
            \end{tabular}
            \end{footnotesize}
        \end{center}
    \end{table}
    
    \begin{figure}
       \begin{center}
            \includegraphics[width=1.0\textwidth]{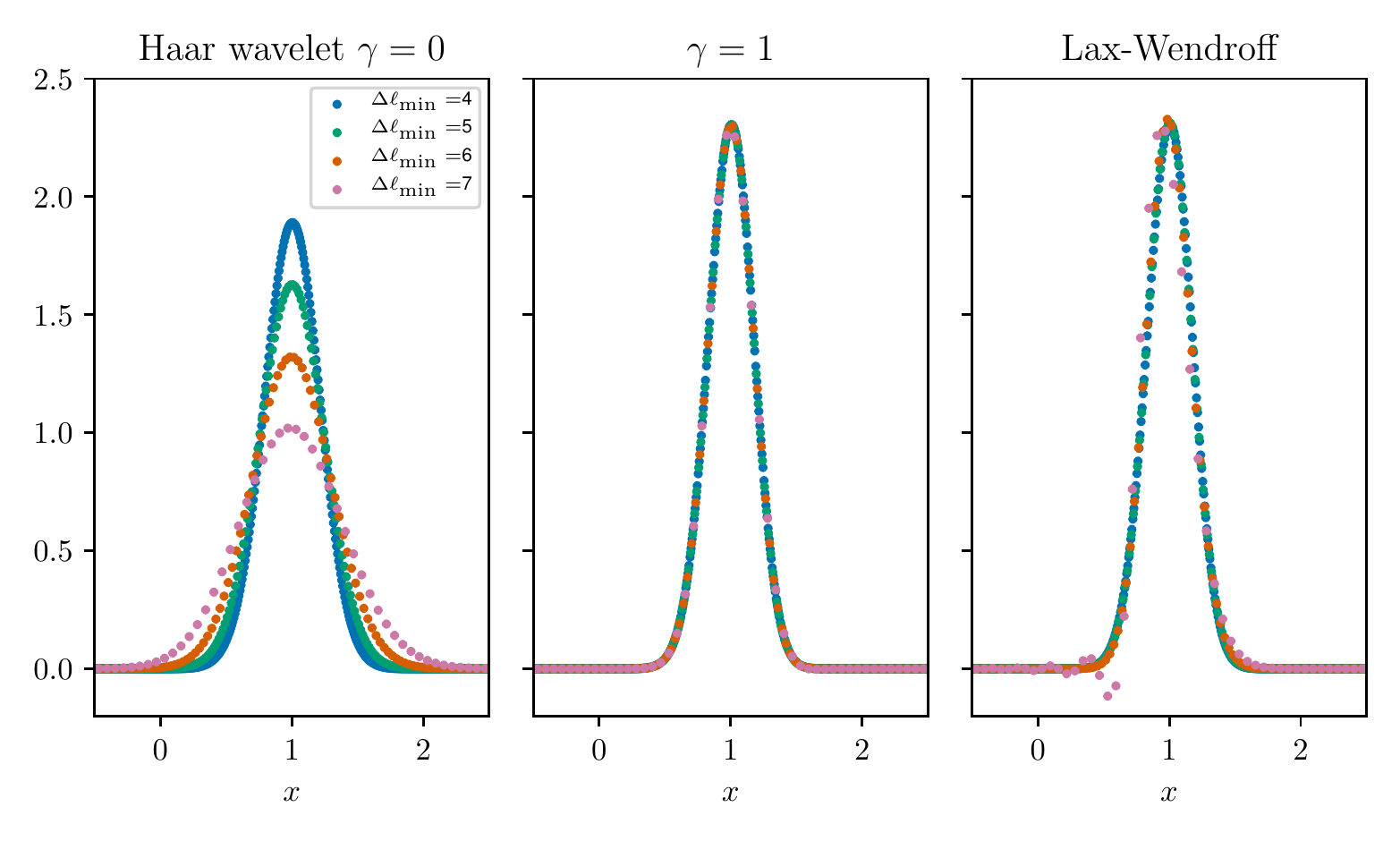}
        \end{center}\caption{\label{fig:linear_advection_diffusion}Solution in test 2 (on a chosen sector of the domain $[-3, 3]$) at the final time $T$ of the 1D linear advection equation for different $\leveldifference_{\text{min}}$ and schemes.}
    \end{figure} 
    
    Once again, the theoretical analysis is fully corroborated by the numerical behavior of the schemes and show that, even in this quite simple framework, LBM-MR schemes for $\predictionstencil \geq 1$ are the most reliable ones.
    
    \FloatBarrier

    \subsection{1D Viscous Burgers equation}\label{sec:ViscousBurgersEquation}
    
    We now turn to a non-linear problem. In this case, the choice of model for the collision phase is no longer negligible and shall be studied at the very end of the work in Section \ref{sec:StudyCollisionOperator}.
    For the moment, we aim at proving that our previous analysis is still meaningful in this context upon verification of the smoothness assumption.
    Indeed, we see that in the case of singularities the previous analysis is no longer well-grounded due to the lack of smoothness and we thus understand the strong interest of dynamic mesh adaptation using multiresolution, discussed in Section \ref{sec:MeshAdaptationTest}.

    \subsubsection{The problem and the reference scheme}
    Consider the solution of the viscous Burgers equation (see Table \ref{tab:TestCasesResumee}) given by
    \begin{equation*}
        u(t, x) = \sqrt{\frac{4\nu}{t}} \frac{\int_{-\infty}^{+\infty} \eta ~  \text{exp}\left [-\frac{1}{4\nu}\text{erf} \adaptiveroundbrackets{\frac{x}{\sqrt{4\nu t_0}} - \sqrt{\frac{t}{t_0}}\eta} \right ] e^{-\eta^2} \text{d}\eta}{\int_{-\infty}^{+\infty}  \text{exp}\left [-\frac{1}{4\nu}\text{erf} \adaptiveroundbrackets{\frac{x}{\sqrt{4\nu t_0}} - \sqrt{\frac{t}{t_0}}\eta} \right ] e^{-\eta^2} \text{d}\eta},
    \end{equation*}    
    where the exact solution has been obtained following \cite{landajuela2011burgers} and the integrals with weight $e^{-\eta^2}$ shall be approximated by Gauss-Hermite formul\ae{} with one hundred points.
    We consider $T = 1$, $t_0 = 1$ and either $\nu = 0.05$ (large diffusion) or $\nu = 0.005$ (small diffusion) for the continuum problem. The computational domain is $[-3, 3]$ with copy boundary conditions.
    We again use a D1Q3 as in Section \ref{sec:LinearAdvectionDiffusion}.
    The scheme is determined by $\latticevelocity = 4$, the equilibria $\momentum^{1, \text{eq}} = (\momentum^0)^2/2$ and $ \momentum^{2, \text{eq}} = (\momentum^0)^3/6 + \kappa \momentum^0/2$, with $\kappa = 4$ (large diffusion) and $\kappa = 1$ (small diffusion) and $s_{v} = (1/2 + \lambda \nu /(\spacestep \kappa))^{-1}$ and $s_w = 1$.
    
    \subsubsection{Results and discussion}

    \begin{table}\caption{\label{tab:ViscousBurgersLargeDiffusion}Test 3 for the 1D viscous Burgers equation taking $\maxlevel = 11$ and $\nu = 0.05$ (large diffusion) and performing the computation using a mesh at level $\minlevel$ as indicated.}
        \begin{center}
            \begin{footnotesize}
            \begin{tabular}{cccccccccc}
                  & \multicolumn{3}{|c|}{Haar $\predictionstencil = 0$} & \multicolumn{3}{|c|}{$\predictionstencil = 1$} & \multicolumn{3}{|c}{Lax-Wendroff} \\
            \toprule[2pt]
                $\leveldifference_{\text{min}}$ & $\text{E}_{\text{adap}}^{\minlevel}$ & $\text{E}_{\text{adap}}^{\maxlevel}$ & $\text{D}_{\text{adap}}$ & $\text{E}_{\text{adap}}^{\minlevel}$ & $\text{E}_{\text{adap}}^{\maxlevel}$ & $\text{D}_{\text{adap}}$ & $\text{E}_{\text{adap}}^{\minlevel}$ & $\text{E}_{\text{adap}}^{\maxlevel}$ & $\text{D}_{\text{adap}}$\\
            \midrule
0 &	\texttt{1.23e-2} &	\texttt{1.23e-2} &	\texttt{0.00e+0} &	\texttt{1.23e-2} &	\texttt{1.23e-2} &	\texttt{0.00e+0} &	\texttt{1.23e-2} &	\texttt{1.23e-2} &	\texttt{0.00e+0} \\
1 &	\texttt{1.24e-2} &	\texttt{1.24e-2} &	\texttt{9.99e-4} &	\texttt{1.23e-2} &	\texttt{1.23e-2} &	\texttt{1.88e-7} &	\texttt{1.23e-2} &	\texttt{1.23e-2} &	\texttt{1.60e-6} \\
2 &	\texttt{1.27e-2} &	\texttt{1.27e-2} &	\texttt{2.99e-3} &	\texttt{1.23e-2} &	\texttt{1.23e-2} &	\texttt{9.34e-7} &	\texttt{1.23e-2} &	\texttt{1.23e-2} &	\texttt{8.02e-6} \\
3 &	\texttt{1.41e-2} &	\texttt{1.41e-2} &	\texttt{6.95e-3} &	\texttt{1.23e-2} &	\texttt{1.23e-2} &	\texttt{3.89e-6} &	\texttt{1.23e-2} &	\texttt{1.23e-2} &	\texttt{3.37e-5} \\
4 &	\texttt{1.94e-2} &	\texttt{1.94e-2} &	\texttt{1.48e-2} &	\texttt{1.23e-2} &	\texttt{1.23e-2} &	\texttt{1.57e-5} &	\texttt{1.22e-2} &	\texttt{1.22e-2} &	\texttt{1.36e-4} \\
5 &	\texttt{3.26e-2} &	\texttt{3.25e-2} &	\texttt{3.00e-2} &	\texttt{1.23e-2} &	\texttt{1.23e-2} &	\texttt{6.30e-5} &	\texttt{1.19e-2} &	\texttt{1.19e-2} &	\texttt{5.48e-4} \\
6 &	\texttt{6.04e-2} &	\texttt{6.03e-2} &	\texttt{5.90e-2} &	\texttt{1.23e-2} &	\texttt{1.23e-2} &	\texttt{2.60e-4} &	\texttt{1.08e-2} &	\texttt{1.09e-2} &	\texttt{2.20e-3} \\
7 &	\texttt{1.13e-1} &	\texttt{1.13e-1} &	\texttt{1.12e-1} &	\texttt{1.22e-2} &	\texttt{1.22e-2} &	\texttt{1.18e-3} &	\texttt{8.32e-3} &	\texttt{8.62e-3} &	\texttt{9.08e-3} \\
                \hline
            \end{tabular}
            \end{footnotesize}
        \end{center}
    \end{table}

    \begin{table}\caption{\label{tab:ViscousBurgersSmallDiffusion}Test 3 for the 1D viscous Burgers equation taking $\maxlevel = 11$ and $\nu = 0.005$ (small diffusion) and performing the computation using a mesh at level $\minlevel$ as indicated.}
        \begin{center}
            \begin{footnotesize}
            \begin{tabular}{cccccccccc}
                  & \multicolumn{3}{|c|}{Haar $\predictionstencil = 0$} & \multicolumn{3}{|c|}{$\predictionstencil = 1$} & \multicolumn{3}{|c}{Lax-Wendroff} \\
            \toprule[2pt]
                $\leveldifference_{\text{min}}$ & $\text{E}_{\text{adap}}^{\minlevel}$ & $\text{E}_{\text{adap}}^{\maxlevel}$ & $\text{D}_{\text{adap}}$ & $\text{E}_{\text{adap}}^{\minlevel}$ & $\text{E}_{\text{adap}}^{\maxlevel}$ & $\text{D}_{\text{adap}}$ & $\text{E}_{\text{adap}}^{\minlevel}$ & $\text{E}_{\text{adap}}^{\maxlevel}$ & $\text{D}_{\text{adap}}$\\
                \hline
0 &	\texttt{5.31e-3} &	\texttt{5.31e-3} &	\texttt{0.00e+0} &	\texttt{5.31e-3} &	\texttt{5.31e-3} &	\texttt{0.00e+0} &	\texttt{5.31e-3} &	\texttt{5.31e-3} &	\texttt{0.00e+0} \\
1 &	\texttt{4.96e-3} &	\texttt{4.96e-3} &	\texttt{1.16e-3} &	\texttt{5.31e-3} &	\texttt{5.31e-3} &	\texttt{3.47e-6} &	\texttt{5.29e-3} &	\texttt{5.29e-3} &	\texttt{2.72e-5} \\
2 &	\texttt{4.62e-3} &	\texttt{4.61e-3} &	\texttt{3.41e-3} &	\texttt{5.31e-3} &	\texttt{5.31e-3} &	\texttt{2.34e-5} &	\texttt{5.21e-3} &	\texttt{5.22e-3} &	\texttt{1.38e-4} \\
3 &	\texttt{6.81e-3} &	\texttt{6.78e-3} &	\texttt{7.76e-3} &	\texttt{5.28e-3} &	\texttt{5.30e-3} &	\texttt{1.41e-4} &	\texttt{4.90e-3} &	\texttt{4.92e-3} &	\texttt{6.17e-4} \\
4 &	\texttt{1.48e-2} &	\texttt{1.47e-2} &	\texttt{1.64e-2} &	\texttt{5.28e-3} &	\texttt{5.31e-3} &	\texttt{8.63e-4} &	\texttt{4.54e-3} &	\texttt{4.58e-3} &	\texttt{3.54e-3} \\
5 &	\texttt{3.23e-2} &	\texttt{3.20e-2} &	\texttt{3.34e-2} &	\texttt{5.92e-3} &	\texttt{6.14e-3} &	\texttt{6.08e-3} &	\texttt{1.43e-2} &	\texttt{1.48e-2} &	\texttt{1.70e-2} \\
6 &	\texttt{6.67e-2} &	\texttt{6.49e-2} &	\texttt{6.57e-2} &	\texttt{2.91e-2} &	\texttt{3.36e-2} &	\texttt{3.37e-2} &	\texttt{9.97e-2} &	\texttt{1.05e-1} &	\texttt{1.04e-1} \\
7 &	\texttt{1.07e-1} &	\texttt{1.24e-1} &	\texttt{1.25e-1} &	\texttt{2.55e-1} &	\texttt{2.45e-1} &	\texttt{2.42e-1} &	\texttt{8.03e-1} &	\texttt{8.18e-1} &	\texttt{8.19e-1} \\
                \hline
            \end{tabular}
            \end{footnotesize}
        \end{center}
    \end{table}

    \begin{figure}
       \begin{center}
            \includegraphics[width=1.0\textwidth]{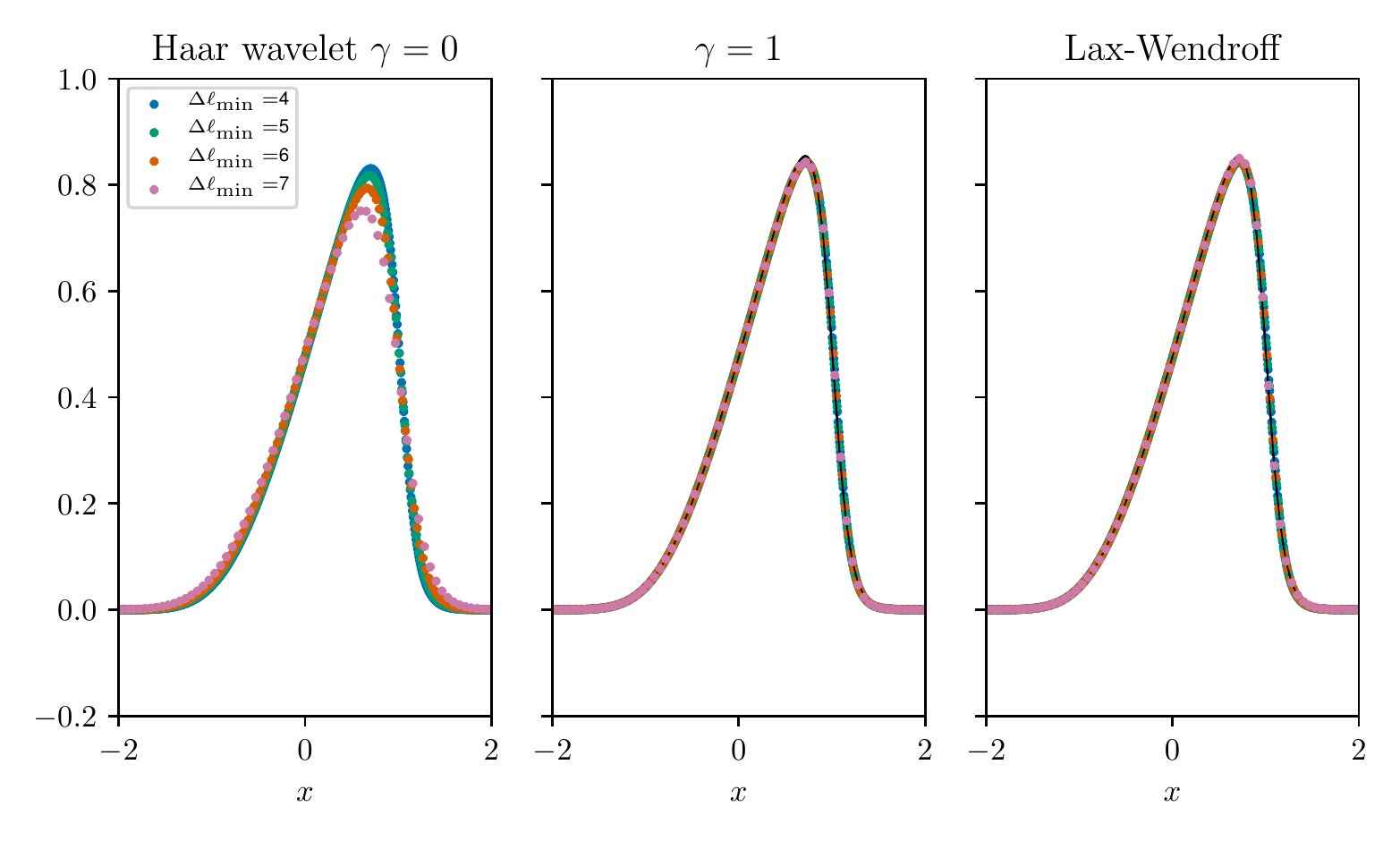}
        \end{center}\caption{\label{fig:viscous_burgers_large}Solution in test 3 (on a chosen sector of the domain $[-3, 3]$) at the final time $T$ of the 1D viscous Burgers equation with $\nu = 0.05$ (large diffusion) for different $\leveldifference_{\text{min}}$ and schemes.}
    \end{figure} 
    
    \begin{figure}
       \begin{center}
            \includegraphics[width=1.0\textwidth]{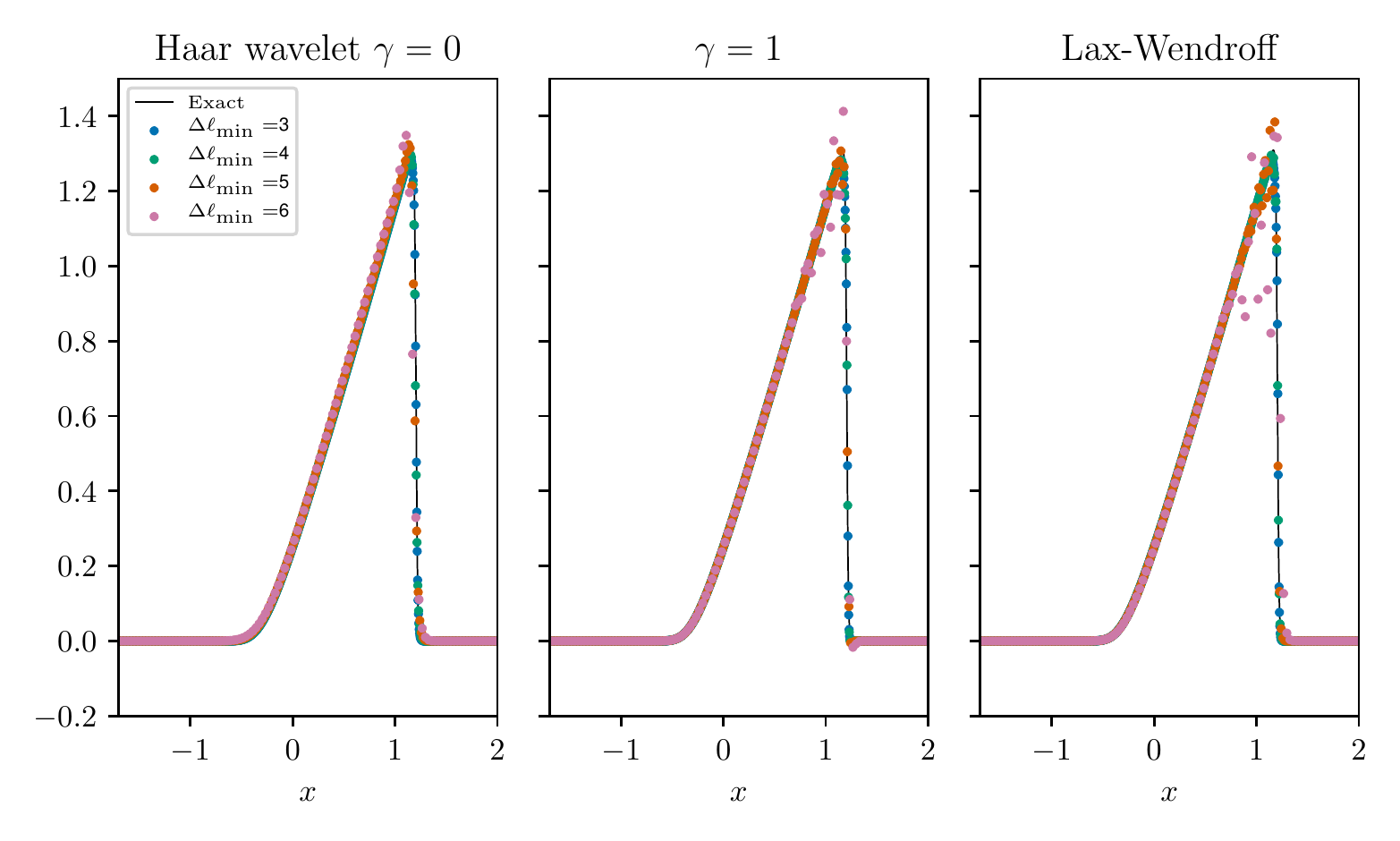}
        \end{center}\caption{\label{fig:viscous_burgers_small}Solution in test 3 (on a chosen sector of the domain $[-3, 3]$) at the final time $T$ of the 1D viscous Burgers equation with $\nu = 0.005$ (small diffusion) for different $\leveldifference_{\text{min}}$ and schemes.}
    \end{figure}

    We first perform the same kind of test than in Section \ref{sec:LinearAdvectionDiffusion} and the results are on Table \ref{tab:ViscousBurgersSmallDiffusion} and \ref{tab:ViscousBurgersLargeDiffusion} and Figure \ref{fig:viscous_burgers_large} and \ref{fig:viscous_burgers_small}.
    We point out the following facts:

    \begin{itemize}
        \item \textbf{Multiresolution scheme for $\predictionstencil = 0$}. In the case of large diffusion, we see that the method adds much numerical diffusion yielding unreliable results. On the other hand, for the small diffusion, the result seems good from a graphic point of view because of the secondary role of diffusion on the shape of the solution. 
        However, the discrepancies from the reference scheme are important in both cases.
        Compared to the other strategies, the difference with respect to the reference algorithm is larger, as expected: starting from $\leveldifference_{\text{min}} = 3$ (for large diffusion) and $\leveldifference_{\text{min}} = 2$ (for small diffusion), the term $\text{D}_{\text{adap}}$ can no longer be neglected.

        \item \textbf{Multiresolution scheme for $\predictionstencil = 1$}. 
        The plots of the solution show that it agrees well with the expected one. We can notice a slight crushing of the solution for the large diffusion which can be considered a fourth order effect.
        In the case of small diffusion, despite the fact that the reference scheme does not oscillate close to the steep zone of the solution, we see that the adaptive method does so for large $\leveldifference_{\text{min}}$. This cannot be due to third-order terms, since they are matched, so one may argue that these are fifth-order effects (not likely) or the consequence of the fact that we are no longer allowed to perform Taylor expansions either because the spatial step is too large or because the solution is not smooth enough. 
        Indeed, as already said, this is the proof that using a fixed coarsened mesh is not a good approach to deal with moving singularities. This context calls for dynamically adapted meshes and error control.
        Still, the difference with the reference scheme is minimized for this choice of reconstruction and the impact of the adaptive scheme can be neglected  for any $\leveldifference_{\text{min}}$ for the large diffusion and until $\leveldifference_{\text{min}} = 4$ for the small diffusion.
        \item \textbf{Lax-Wendroff scheme}. The result for the large diffusion seems reliable, even if the method slightly ``overshoots'' the exact solution allegedly due to the third-order mismatch. On the other hand, the test with small diffusion clearly shows that the method perturbs the dispersive terms at third order, inducing oscillations near the kinky zones of the solution.
        Compared to the reference solution, the behavior of the Lax-Wendroff scheme is situated half-way between those for $\predictionstencil = 0$ and $\predictionstencil = 1$, as expected.
        in particular, the impact of the adaptive stream is negligible until $\leveldifference_{\text{min}} = 7$ (for large diffusion) and $\leveldifference_{\text{min}} = 3$ (for small diffusion).
    \end{itemize}
    For each stream strategy, we see that $\text{D}_{\text{adap}}$ stops to be negligible compared  $\text{E}_{\text{ref}}$ earlier for the small diffusion than for the large. This is coherent with the fact that the solution develops more high-frequency modes.
    Moreover, one limiting factor of the theoretical analysis are the implicit smoothness assumptions on the solutions. 
    In the case where the solution is close to singular and especially for large $\leveldifference_{\text{min}}$, the behavior of the numerical scheme on a uniform coarsened mesh deviates from the theoretical predictions because the smoothness assumption is not valid. 
    From a multiresolution perspective, the lack of smoothness translates into the fact that the details of the solution at the finest level $\maxlevel$ are not small close to the blowup.
    
    \FloatBarrier

    \subsection{2D Linear advection-diffusion equation}\label{sec:2DAdvectionDiffusionEquation}
    
        Before ending with a demonstration with mesh adaptation and a study on the influence of the collision phase on the quality of the numerical solution, we repeat the test of Section \ref{sec:LinearAdvectionDiffusion} for $\spatialdimensionality = 2$, to corroborate the extension of the previous analysis to the multidimensional setting done in Section \ref{sec:Equivalent2D}.
        We selected a quite ``rich'' numerical model in terms of degrees of freedom to show the generality of our analysis.

        \subsubsection{The problem and the reference scheme}
        The target equation is the same than \eqref{eq:LinearAdvectionDiffusionEquation} just taking $\spatialdimensionality = 2$ (see Table \ref{tab:TestCasesResumee}), therefore the solution is
        \begin{equation*}
            u(t, \vectorial{x}) = \frac{1}{4 \pi \nu (t_0 + t)} \text{exp} \adaptiveroundbrackets{-\frac{|\vectorial{x} - \vectorial{V}t|^2}{4\nu (t_0 + t)}}.
        \end{equation*}

        The scheme we use is the D2Q9 with velocities given by
        \begin{equation*}
            \normalizedvelocityletter_{\populationindex} = 
            \begin{cases}
                (0, 0), \qquad &\populationindex = 0, \\
                \adaptiveroundbrackets{\cos{\adaptiveroundbrackets{\frac{\pi}{2}(\populationindex - 1)}}, \sin{\adaptiveroundbrackets{\frac{\pi}{2}(\populationindex - 1)}}},  \qquad &\populationindex = 1,2,3,4, \\
                \adaptiveroundbrackets{\cos{\adaptiveroundbrackets{\frac{\pi}{2}(\populationindex - 5)+\frac{\pi}{4}}}, \sin{\adaptiveroundbrackets{\frac{\pi}{2}(\populationindex - 5) + \frac{\pi}{4}}}},  \qquad &\populationindex = 5,6,7,8,
            \end{cases}
        \end{equation*}
        with the moments by Lallemand and Luo \cite{lallemand2000theory} relaxing with $\operatorial{S} = \text{diag}(0, s, s, 1, 1, 1, 1, 1, 1)$
        \begin{equation*}
            \operatorial{M} = \adaptiveroundbrackets{
            \begin{matrix}
        1 &  1  &  1   &  1  &  1  &    1   &   1   &   1   &   1   \\
        0 & \lambda  &  0   & -\lambda &  0  &   \lambda   &  -\lambda  &  -\lambda  &  \lambda   \\
        0 &  0  &  \lambda  &  0  & -\lambda &   \lambda   &  \lambda   &  -\lambda  &  -\lambda  \\
        -4\lambda^2 & -\lambda^2  & -\lambda^2   & -\lambda^2  & -\lambda^2  &  2\lambda^2  & 2\lambda^2  & 2\lambda^2  & 2\lambda^2  \\
        0 &  -2\lambda^3  &  0   &  2\lambda^3  &  0  &   \lambda^3   & -\lambda^3   & -\lambda^3   &  \lambda^3   \\
        0 &  0  &  -2\lambda^3   &  0  &  2\lambda^3  &   \lambda^3   &  \lambda^3   & -\lambda^3   & -\lambda^3   \\
        4\lambda^4 &  -2\lambda^4  &  -2\lambda^4   &  -2\lambda^4 &  -2\lambda^4  &   \lambda^4   &  \lambda^4   &  \lambda^4   &  \lambda^4   \\
        0 & \lambda^2  & -\lambda^2  & \lambda^2  & -\lambda^2 &    0   &   0   &   0   &   0   \\
        0 &  0  &  0   &  0  &  0  &   \lambda^2   & -\lambda^2   &  \lambda^2   & -\lambda^2  
            \end{matrix}},
        \end{equation*}
        with $s = (1/2 + 3 \nu/(\latticevelocity \spacestep))^{-1}$ to enforce the diffusivity.
        The equilibria are based on the second-order expansion of the Maxwellian as in \cite{fakhari2016mass}
        \begin{align*}
            m^{1, \text{eq}} = V_x m^0, \quad m^{2, \text{eq}} = V_y m^0, \quad m^{3, \text{eq}} = (-2 \latticevelocity^2 + 3|\vectorial{V}|^2) m^0, \\
            m^{4, \text{eq}} = -\latticevelocity^2 V_x m^0, \quad m^{5, \text{eq}} = -\latticevelocity^2 V_y m^0, \quad m^{6, \text{eq}} = (\latticevelocity^4 - 3 \latticevelocity^2 |\vectorial{V}|^2)m^0, \\
            m^{7, \text{eq}} = (V_x^2 - V_y^2)m^0, \quad m^{8, \text{eq}} = V_x V_y m^0.
        \end{align*}
        
        The parameters of the problem are $T = 0.5$, $t_0 = 1$, $V_x = 0.5$, $V_y = 0.5$, $\nu = 0.005$, $\latticevelocity = 1$. The computation is done on the bounded domain $[-1/2, 1]^2$.

    \subsubsection{Results and discussion}
    
            \begin{figure}
       \begin{center}
            \includegraphics[width=1.0\textwidth]{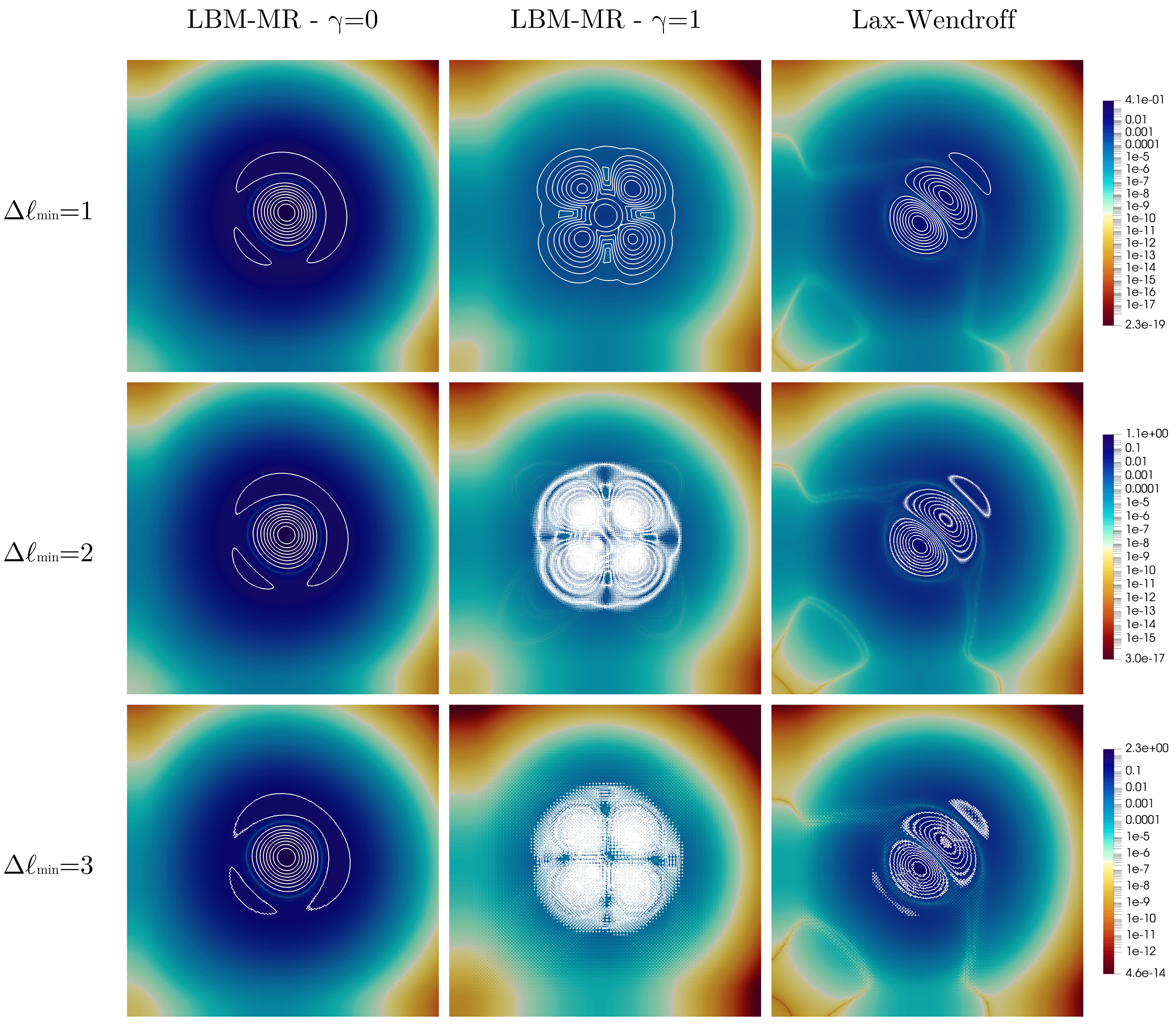}
        \end{center}\caption{\label{fig:2d_difference}Spatial patterns of $\text{D}_{\text{adap}}$ in test 4 at the final time $T$ with logarithmic (to compare $\predictionstencil = 0, 1$ and Lax-Wendroff) color-scale for each choice of $\leveldifference_{\text{min}}$. We also plot ten isocontours to show the main trends which are hidden behind the logarithmic scale.}
    \end{figure} 

    \begin{figure}
       \begin{center}
            \includegraphics[width=1.0\textwidth]{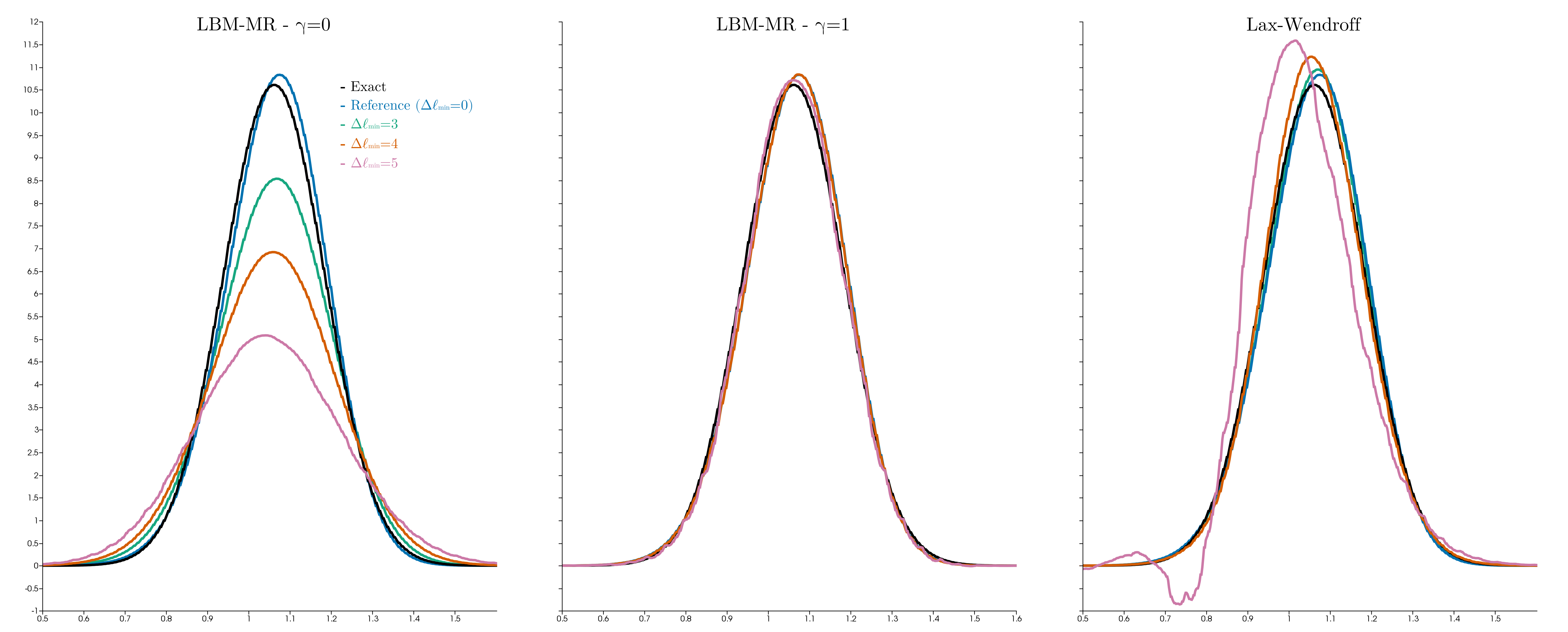}
        \end{center}\caption{\label{fig:2d_solutions}Solution of test 4 on a diagonal cut at the final time $T$ of the 2D linear advection diffusion equation for different $\leveldifference_{\text{min}}$ and schemes.}
    \end{figure}

        \begin{table}\caption{\label{tab:LinearAdvectionDiffusion2D}Test 4 for the 2D Linear advection diffusion equation taking $\maxlevel = 9$ and performing the computation using a mesh at level $\minlevel$ as indicated.}
        \begin{center}
            \begin{footnotesize}
            \begin{tabular}{cccccccccc}
                  & \multicolumn{3}{|c|}{Haar $\predictionstencil = 0$} & \multicolumn{3}{|c|}{$\predictionstencil = 1$} & \multicolumn{3}{|c}{Lax-Wendroff} \\
            \toprule[2pt]
                $\leveldifference_{\text{min}}$ & $\text{E}_{\text{adap}}^{\minlevel}$ & $\text{E}_{\text{adap}}^{\maxlevel}$ & $\text{D}_{\text{adap}}$ & $\text{E}_{\text{adap}}^{\minlevel}$ & $\text{E}_{\text{adap}}^{\maxlevel}$ & $\text{D}_{\text{adap}}$ & $\text{E}_{\text{adap}}^{\minlevel}$ & $\text{E}_{\text{adap}}^{\maxlevel}$ & $\text{D}_{\text{adap}}$\\
                \hline
0 &	\texttt{4.86e-2} &	\texttt{4.86e-2} &	\texttt{0.00e+0} &	\texttt{4.86e-2} &	\texttt{4.86e-2} & \texttt{0.00e+0} & \texttt{4.86e-2} & \texttt{4.86e-2} & \texttt{0.00e+0} \\
1 &	\texttt{4.61e-2} &	\texttt{4.61e-2} &	\texttt{2.79e-2} &	\texttt{4.86e-2} &	\texttt{4.86e-2} & \texttt{9.42e-5} & \texttt{4.80e-2} & \texttt{4.80e-2} & \texttt{8.20e-4} \\
2 &	\texttt{7.60e-2} &	\texttt{7.58e-2} &	\texttt{8.06e-2} &	\texttt{4.86e-2} &	\texttt{4.87e-2} & \texttt{3.89e-4} & \texttt{4.55e-2} & \texttt{4.56e-2} & \texttt{4.09e-3} \\
3 &	\texttt{1.65e-1} &	\texttt{1.64e-1} &	\texttt{1.75e-1} &	\texttt{4.83e-2} &	\texttt{4.87e-2} & \texttt{1.62e-3} & \texttt{3.66e-2} & \texttt{3.71e-2} & \texttt{1.71e-2} \\
4 &	\texttt{3.19e-1} &	\texttt{3.16e-1} &	\texttt{3.29e-1} &	\texttt{4.64e-2} &	\texttt{4.82e-2} & \texttt{7.49e-3} & \texttt{3.69e-2} & \texttt{4.01e-2} & \texttt{6.90e-2} \\
5 &	\texttt{5.47e-1} &	\texttt{5.38e-1} &	\texttt{5.51e-1} &	\texttt{3.69e-2} &	\texttt{4.99e-2} & \texttt{4.94e-2} & \texttt{2.27e-1} & \texttt{2.39e-1} & \texttt{2.82e-1} \\
6 &	\texttt{8.22e-1} &	\texttt{8.16e-1} &	\texttt{8.26e-1} &	\texttt{4.44e-1} &	\texttt{4.74e-1} & \texttt{5.14e-1} & \texttt{9.29e-1} & \texttt{1.00e+0} & \texttt{1.04e+0} \\
                \hline
            \end{tabular}
            \end{footnotesize}
        \end{center}
    \end{table}
    
    We perform the same kind of test than for the unidimensional problem, by taking $\maxlevel = 9$ and using different minimum levels $\minlevel$. 
    The full results are on Table \ref{tab:LinearAdvectionDiffusion2D} and some plots of the solution in Figures \ref{fig:2d_difference} and \ref{fig:2d_solutions}.
    We observe the following facts:
    \begin{itemize}
        \item \textbf{Multiresolution scheme for $\predictionstencil = 0$}. As one expects, the diffusion term is not correctly handled. This results in a non-negligible additional error $\text{D}_{\text{adap}}$ in Table \ref{tab:LinearAdvectionDiffusion2D} and Figure \ref{fig:2d_solutions} clearly shows that the packet is crushed way too rapidly. It is also interesting to notice that since the most important contribution to $\text{D}_{\text{adap}}$ is an additional isotropic diffusion, the structure of $\text{D}_{\text{adap}}$ (see the white contours on Figure \ref{fig:2d_difference}) is essentially isotropic.
        \item \textbf{Multiresolution scheme for $\predictionstencil = 1$}. On the other hand, this method successfully copes with the diffusion phenomena, being able to have a negligible $\text{D}_{\text{adap}}$ until $\leveldifference_{\text{min}} = 5$. Figure \ref{fig:2d_solutions} shows a very good agreement with the expected solution and Figure \ref{fig:2d_difference} shows that the discrepancies from the reference scheme are essentially isotropic, since made up of fourth-order terms, with additional rapidly oscillatory terms when $\leveldifference_{\text{min}}$ increases. This creates the dense amount of contours we can observe.
        \item \textbf{Lax-Wendroff scheme}. Again as expected, the method does not alter the diffusion terms but $\text{D}_{\text{adap}}$ starts to be a dominant term earlier than for $\predictionstencil = 1$, namely around $\leveldifference_{\text{min}} = 3$. This can be also understood when looking at Figure \ref{fig:2d_solutions}, where one clearly notices the alteration of the third order terms which induces a dispersive effect. This can also be seen on Figure \ref{fig:2d_difference}, where the dispersive effect shows to be non-isotropic and linked with the propagation of the packet in space at finite velocity. Once $\leveldifference_{\text{min}}$ increases, we still observe, though way less intensely than for $\predictionstencil = 1$, the development of high-frequency components of $\text{D}_{\text{adap}}$.
    \end{itemize}
    
    \FloatBarrier

\subsection{Mesh adaptation}\label{sec:MeshAdaptationTest}

    The actual mesh adaptation by multiresolution is not the main subject of this work but we have seen in Section \ref{sec:ViscousBurgersEquation} that it is really needed in some situations.
    %However, we would like to provide a demonstration of the capabilities of such approach.
    In particular, we want to look at the fact that the adaptive scheme is accurate enough to allow, even if the initial mesh is quite coarsened with respect to the finest level $\maxlevel$, to progressively refine the mesh when steep gradients occur, as in the solution of the viscous Burgers equation in Section \ref{sec:ViscousBurgersEquation}.
    This is the only test of this paper where mesh adapted by multiresolution are used.
    To this end, we consider $\predictionstencil = 1$ and we apply our multiresolution algorithm \eqref{eq:DetailInequality} to perform mesh adaptation at each time step using $\epsilon = 0.0001$ and $\overline{\mu} = 2$.
    
    \subsubsection{Results and discussion}
    
    Looking at the result on Figure \ref{fig:viscous_burgers_large_gm_1_with_MR} for the large diffusion, we observe that the finest reached level (namely with $\leveldifference = 4$) is the same at the beginning and at the end of the simulation. This happens because the smoothness of the solution does not change as time advances and therefore the magnitude of the details remain essentially the same.
    In this case, we have already seen in Section \ref{sec:ViscousBurgersEquation} that a uniform coarsened mesh is sufficient to correctly follow the dynamics and that the theoretical analysis using Taylor expansions is legitimate.
    
    On the other hand in the case of small diffusion, we observe on Figure \ref{fig:viscous_burgers_small_gm_1_with_MR} that the expectations are met: the initial Gaussian datum is regular enough to coarsen the mesh everywhere. However, the adaptive scheme correctly captures the dynamics of the solution and multiresolution eventually puts more and more refined grids to correctly follow these structures. The final solution is close to be singular, we see that we also reach the finest level $\maxlevel$ close to the steep areas of the solution, where details have really changed order of magnitude.
    This demonstrates, on one side, that the numerical scheme is reliable enough to compensate possible losses of information induced by the mesh coarsening \eqref{eq:DetailInequality} \emph{via} multiresolution and, on the other side, that the interest of the theoretical analysis and the use of fixed meshes is limited to smooth solutions.
    For singular solutions, a dynamic refinement algorithm is actually needed.

    \begin{figure}
       \begin{center}
            \includegraphics[width=1.0\textwidth]{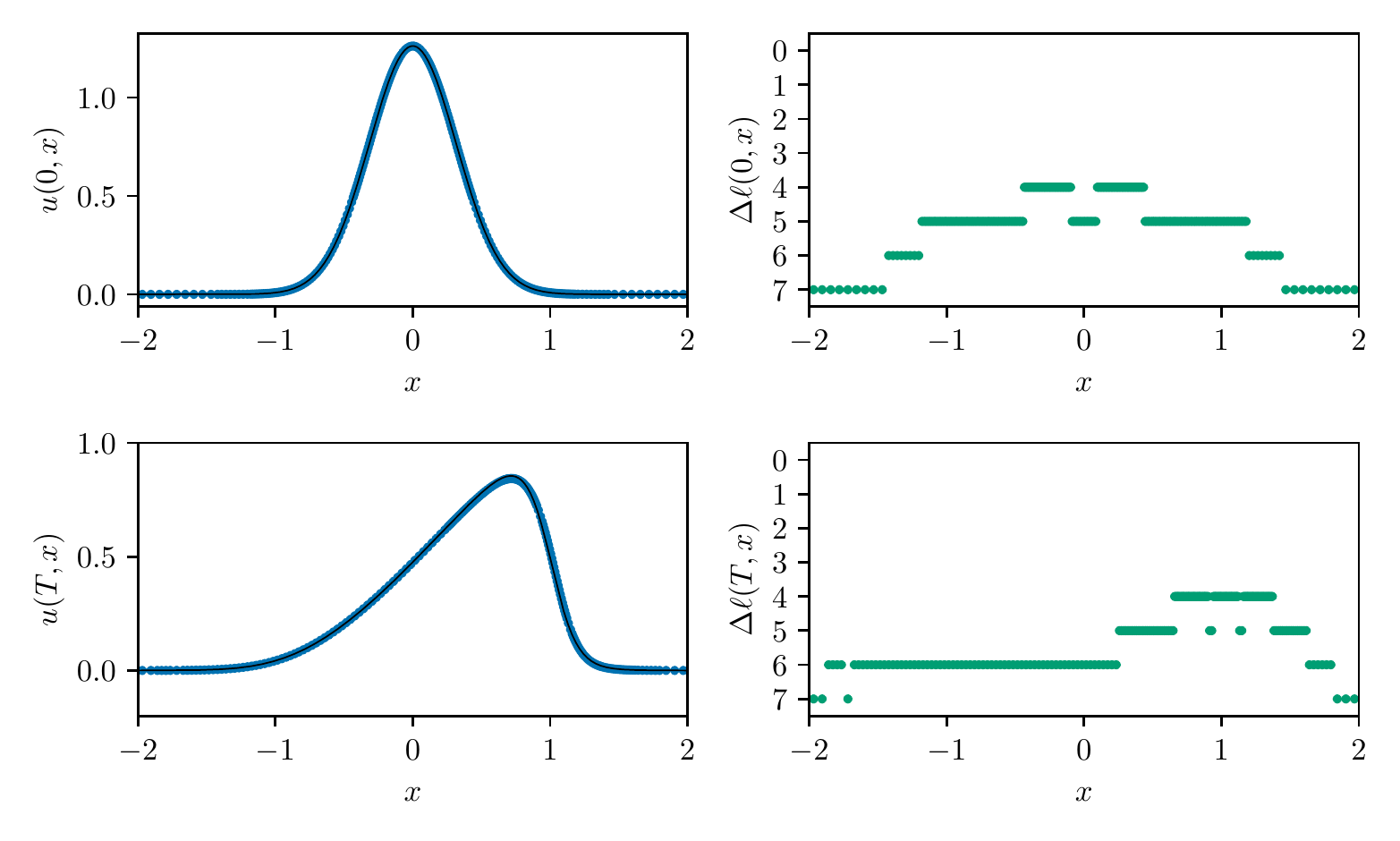}
        \end{center}\caption{\label{fig:viscous_burgers_large_gm_1_with_MR}Solution of test 4 at the initial and final time $T$ of the 1D viscous Burgers equation with $\nu = 0.05$ (large diffusion) for $\predictionstencil = 1$ using adaptive multiresolution with $\epsilon = 0.0001$.% and $\overline{\mu} = 2$. 
        ~ The exact solution is traced in black.}
    \end{figure} 
    
    \begin{figure}
       \begin{center}
            \includegraphics[width=1.0\textwidth]{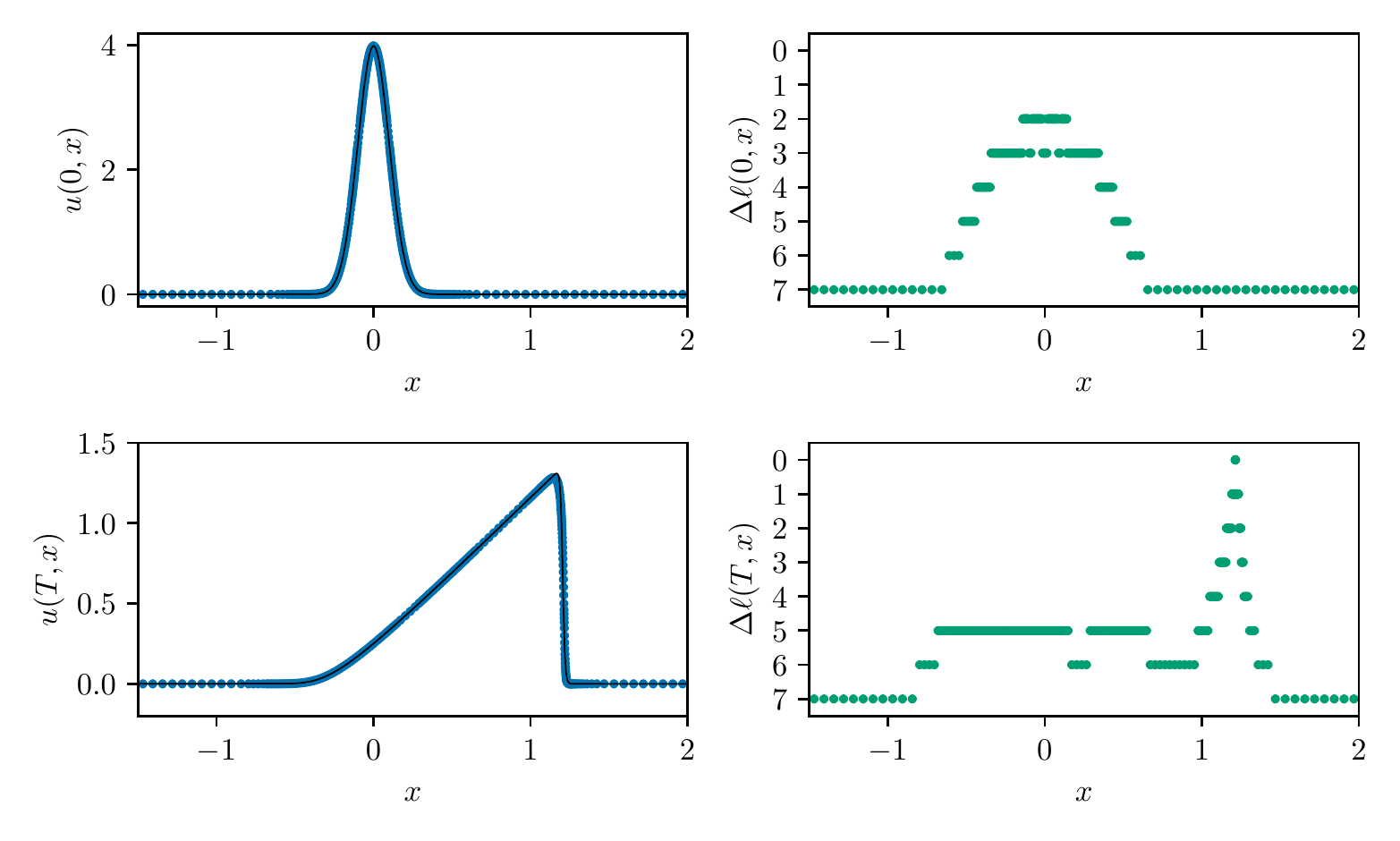}
        \end{center}\caption{\label{fig:viscous_burgers_small_gm_1_with_MR}Solution of test 4 at the initial and final time $T$ of the 1D viscous Burgers equation with $\nu = 0.005$ (small diffusion) for $\predictionstencil = 1$ using adaptive multiresolution with $\epsilon = 0.0001$.% and $\overline{\mu} = 2$. 
        ~ The exact solution is traced in black.}
    \end{figure} 
    
    \FloatBarrier

\subsection{Influence of the collision operator}\label{sec:StudyCollisionOperator}
    Before concluding, we want to study the influence of the choice of collision algorithm on the quality of the outcome of our numerical method. In particular, we want to elucidate if the collision phase must be performed on the finest level of resolution $\maxlevel$ or not.
    We repeat the first simulation of Section \ref{sec:ViscousBurgersEquation}, which was the only non-linear one we have done in this work, using $\gamma = 1$, which have proved to be the most reliable stream approximation between those we analyzed.
    This is coupled with the LBM-MR-LC, the LBM-MR-RC and the LBM-MR-PQC collision techniques with a Gauss-Legendre quadrature of order five, see \cite{abramowitz1964}. On the reference interval $[-1, 1]$, this corresponds to the quadrature points and weights $(\tilde{x}_1, \tilde{w}_1) = (-\sqrt{3/5}, 5/9)$, $(\tilde{x}_2, \tilde{w}_2) = (0, 8/9)$ and $(\tilde{x}_3, \tilde{w}_3) = (\sqrt{3/5}, 5/9)$.
    %For every method, the reconstruction polynomials and the advection phase are built using $\predictionstencil = 1$.
    
    \subsubsection{Results and discussion}

    \begin{table}\caption{\label{tab:ViscousBurgersCollisionLargeDiffusion}Test 4 for the 1D viscous Burgers equation taking $\maxlevel = 11$ and $\nu = 0.05$ (large diffusion) and performing the computation using a mesh at level $\minlevel$ as indicated. This is a comparison between different collision strategies, contrarily to Table \ref{tab:ViscousBurgersLargeDiffusion} (some results coincide).}
        \begin{center}
            \begin{footnotesize}
            \begin{tabular}{cccccccccc}
                  & \multicolumn{3}{|c|}{LBM-MR-LC} & \multicolumn{3}{|c|}{LBM-MR-RC} & \multicolumn{3}{|c}{LBM-MR-PQC} \\
            \toprule[2pt]
                $\leveldifference_{\text{min}}$ & $\text{E}_{\text{adap}}^{\minlevel}$ & $\text{E}_{\text{adap}}^{\maxlevel}$ & $\text{D}_{\text{adap}}$ & $\text{E}_{\text{adap}}^{\minlevel}$ & $\text{E}_{\text{adap}}^{\maxlevel}$ & $\text{D}_{\text{adap}}$ & $\text{E}_{\text{adap}}^{\minlevel}$ & $\text{E}_{\text{adap}}^{\maxlevel}$ & $\text{D}_{\text{adap}}$\\
                \hline
0 &		\texttt{1.23e-2} &	\texttt{1.23e-2} &	\texttt{0.00e+0} &			\texttt{1.23e-2} &	\texttt{1.23e-2} &	\texttt{0.00e+0} &			\texttt{1.23e-2} &	\texttt{1.23e-2} &	\texttt{5.18e-8} \\
1 &		\texttt{1.23e-2} &	\texttt{1.23e-2} &	\texttt{1.88e-7} &			\texttt{1.23e-2} &	\texttt{1.23e-2} &	\texttt{1.14e-7} &			\texttt{1.23e-2} &	\texttt{1.23e-2} &	\texttt{1.27e-7} \\
2 &		\texttt{1.23e-2} &	\texttt{1.23e-2} &	\texttt{9.34e-7} &			\texttt{1.23e-2} &	\texttt{1.23e-2} &	\texttt{5.70e-7} &			\texttt{1.23e-2} &	\texttt{1.23e-2} &	\texttt{5.76e-7} \\
3 &		\texttt{1.23e-2} &	\texttt{1.23e-2} &	\texttt{3.89e-6} &			\texttt{1.23e-2} &	\texttt{1.23e-2} &	\texttt{2.40e-6} &			\texttt{1.23e-2} &	\texttt{1.23e-2} &	\texttt{2.41e-6} \\
4 &		\texttt{1.23e-2} &	\texttt{1.23e-2} &	\texttt{1.57e-5} &			\texttt{1.23e-2} &	\texttt{1.23e-2} &	\texttt{9.78e-6} &			\texttt{1.23e-2} &	\texttt{1.23e-2} &	\texttt{9.79e-6} \\
5 &		\texttt{1.23e-2} &	\texttt{1.23e-2} &	\texttt{6.30e-5} &			\texttt{1.23e-2} &	\texttt{1.23e-2} &	\texttt{4.06e-5} &			\texttt{1.23e-2} &	\texttt{1.23e-2} &	\texttt{4.06e-5} \\
6 &		\texttt{1.23e-2} &	\texttt{1.23e-2} &	\texttt{2.60e-4} &			\texttt{1.23e-2} &	\texttt{1.23e-2} &	\texttt{1.86e-4} &			\texttt{1.23e-2} &	\texttt{1.23e-2} &	\texttt{1.86e-4} \\
7 &		\texttt{1.22e-2} &	\texttt{1.22e-2} &	\texttt{1.18e-3} &			\texttt{1.22e-2} &	\texttt{1.23e-2} &	\texttt{9.97e-4} &			\texttt{1.22e-2} &	\texttt{1.23e-2} &	\texttt{9.98e-4} \\
                \hline
            \end{tabular}
            \end{footnotesize}
        \end{center}
    \end{table}
    
        \begin{table}\caption{\label{tab:ViscousBurgersCollisionSmallDiffusion}Test 4 for the 1D viscous Burgers equation taking $\maxlevel = 11$ and $\nu = 0.005$ (small diffusion) and performing the computation using a mesh at level $\minlevel$ as indicated. This is a comparison between different collision strategies, contrarily to Table \ref{tab:ViscousBurgersSmallDiffusion} (some results coincide).}
        \begin{center}
            \begin{footnotesize}
            \begin{tabular}{cccccccccc}
                  & \multicolumn{3}{|c|}{LBM-MR-LC} & \multicolumn{3}{|c|}{LBM-MR-RC} & \multicolumn{3}{|c}{LBM-MR-PQC} \\
            \toprule[2pt]
                $\leveldifference_{\text{min}}$ & $\text{E}_{\text{adap}}^{\minlevel}$ & $\text{E}_{\text{adap}}^{\maxlevel}$ & $\text{D}_{\text{adap}}$ & $\text{E}_{\text{adap}}^{\minlevel}$ & $\text{E}_{\text{adap}}^{\maxlevel}$ & $\text{D}_{\text{adap}}$ & $\text{E}_{\text{adap}}^{\minlevel}$ & $\text{E}_{\text{adap}}^{\maxlevel}$ & $\text{D}_{\text{adap}}$\\
                \hline
0 &		\texttt{5.31e-3} &	\texttt{5.31e-3} &	\texttt{0.00e+0} &			\texttt{5.31e-3} &	\texttt{5.31e-3} &	\texttt{0.00e+0} &			\texttt{5.31e-3} &	\texttt{5.31e-3} &	\texttt{1.19e-6} \\
1 &		\texttt{5.31e-3} &	\texttt{5.31e-3} &	\texttt{3.47e-6} &			\texttt{5.31e-3} &	\texttt{5.31e-3} &	\texttt{2.79e-6} &			\texttt{5.31e-3} &	\texttt{5.31e-3} &	\texttt{3.02e-6} \\
2 &		\texttt{5.31e-3} &	\texttt{5.31e-3} &	\texttt{2.34e-5} &			\texttt{5.31e-3} &	\texttt{5.31e-3} &	\texttt{2.28e-5} &			\texttt{5.31e-3} &	\texttt{5.31e-3} &	\texttt{2.29e-5} \\
3 &		\texttt{5.28e-3} &	\texttt{5.30e-3} &	\texttt{1.41e-4} &			\texttt{5.28e-3} &	\texttt{5.28e-3} &	\texttt{1.43e-4} &			\texttt{5.28e-3} &	\texttt{5.28e-3} &	\texttt{1.43e-4} \\
4 &		\texttt{5.28e-3} &	\texttt{5.31e-3} &	\texttt{8.63e-4} &			\texttt{5.29e-3} &	\texttt{5.27e-3} &	\texttt{8.93e-4} &			\texttt{5.29e-3} &	\texttt{5.27e-3} &	\texttt{8.93e-4} \\
5 &		\texttt{5.92e-3} &	\texttt{6.14e-3} &	\texttt{6.08e-3} &			\texttt{5.75e-3} &	\texttt{5.83e-3} &	\texttt{5.73e-3} &			\texttt{5.75e-3} &	\texttt{5.84e-3} &	\texttt{5.76e-3} \\
6 &		\texttt{2.91e-2} &	\texttt{3.36e-2} &	\texttt{3.37e-2} &			\texttt{2.67e-2} &	\texttt{3.11e-2} &	\texttt{3.14e-2} &			\texttt{2.67e-2} &	\texttt{3.12e-2} &	\texttt{3.15e-2} \\
7 &		\texttt{2.55e-1} &	\texttt{2.45e-1} &	\texttt{2.42e-1} &			\texttt{2.37e-1} &	\texttt{2.27e-1} &	\texttt{2.23e-1} &			\texttt{2.32e-1} &	\texttt{2.22e-1} &	\texttt{2.19e-1} \\
                \hline
            \end{tabular}
            \end{footnotesize}
        \end{center}
    \end{table}
    
    The result in the case of large diffusion is given in Table \ref{tab:ViscousBurgersCollisionLargeDiffusion}. We observe that for this smooth solution, the additional error $\text{D}_{\text{adap}}$ induced by the collision performed on the leaves is slightly larger (about 1.5 times) than those for the reconstructed collision and the interpolated collision. Still, all the errors have the same order of magnitude.
    On the other hand, the interpolated method behaves almost like the reconstructed method except for $\leveldifference_{\text{min}} = 0$ where it does not have $\text{D}_{\text{adap}} = 0$ because the method does not perfectly coincide with the reference one.
    
    Regarding the case of small diffusion in Table \ref{tab:ViscousBurgersCollisionSmallDiffusion}, we obtain similar results, with the collision performed on the leaves (LBM-MR-LC) showing marginally larger additional errors.
    This test shows that the LBM-MR-LC strategy can be regarded as reliable even for functions which do not behave polynomially, as the solution we considered, and which can present steep fronts.
    %Moreover, notice that the equilibria we used had up to third-degree non-linearities.
    If one desires a slightly more qualitative collision strategy without significantly increase the computational cost, one may consider the predict-and-quadrate strategy.
    Overall, the reconstructed collision, although guaranteeing the most accurate results, is generally not a viable choice due to its cost and the fact that it provides performances which are marginally better than the other cheaper strategies. 
    
    \FloatBarrier

\section{Conclusions}\label{sec:Conclusions}

    In this original contribution, we have shown how to apply the classical analysis based on the equivalent equations introduced by Dubois \cite{dubois2008equivalent} to the LBM-MR schemes developed in \cite{bellotti2021sisc} and \cite{bellotti2021multidimensional}.
    This is based on the so-called ``reconstruction flattening'' procedure, which is thoroughly studied in the case of $\spatialdimensionality = 1$ and extended to the case $\spatialdimensionality = 2$. 
    
    Therefore, we are able to analyze the consistency of these methods with the target equations as for standard LBM methods and to find the maximal order of compliance of the adaptive scheme with the desired physics. In particular, our analysis has shown that the scheme based on the Haar wavelet is not accurate enough to handle the typical applications for which LBM schemes are designed, namely the simulation of models involving both transport and diffusion terms.
    The Lax-Wendroff scheme by \cite{fakhari2014finite} provides the minimal setting to utilize the most common LBM algorithms but it could yield unpredictable behaviors of dispersive nature, which could threaten the stability of the method.
    The multiresolution scheme for $\predictionstencil \geq 1$ proves to be the most reliable of the schemes we have analyzed, both in terms of consistency and stability.
    
    The analysis is valid for locally smooth solutions. This assumption is always met on grids adapted \emph{via} multiresolution because it guarantees a certain level of regularity of the solution at the local grid level.
    
    Our analysis has been validated using numerical simulations to solve scalar conservation laws, both linear and non-linear, in 1D and 2D, showing excellent agreement between the empirical behavior of the schemes and our asymptotic analysis. When the smoothness assumption was no longer valid on the uniform coarsened mesh, the adaptive setting has proved to be the right one to restore the validity of this assumption and to conclude on the reliability of the strategy through the previous analysis. 
    Moreover, we compared the outcome of our method against that of well-known works in literature \cite{fakhari2014finite, fakhari2015numerics, fakhari2016mass}, showing that, even if for a slightly larger computational cost, our method (for $\predictionstencil = 1$) is consistently more reliable.
    Eventually, we have studied the role of different collision strategies on non-linear problems to conclude that the cheap leaves collision is in general a good choice and accuracy is only marginally affected.
    
    Finally, let us mention the fact that we have worked on uniform grids in the present contribution in order to reproduce the local environment around a given cell on general dynamically adaptive grids for smooth solutions. However, within this framework, travelling waves or shock waves leading to a lower level of regularity usually remain propagated at the finest level of mesh as in the previous example (see Section \ref{sec:MeshAdaptationTest}) and thus never experience going through level jumps. However, when the mesh is fixed in advance and a level jump is present, the solution obtained with lattice Boltzmann methods is known to experience artificial reflections and spurious phenomena;  this is known to be a challenging problem. In fact, the formalism proposed in the present contribution allows to tackle this issue as well but requires a somewhat different setting which we have proposed in another short contribution  \cite{bellotti2021compterendus}. We can show that the reflected part of the wave is of small amplitude and can be controlled, as it is proportional to $\spacestep^4$, when selecting $\predictionstencil = 1$.
        
\section*{Acknowledgments}

The authors deeply thank Pierre Lallemand and Fran\c{c}ois Dubois for the stimulating discussions during the ``groupe de travail Schémas de Boltzmann sur réseau'' held at the IHP Paris in spring 2021.
Thomas Bellotti acknowledges a PhD funding (year 2019) from  \'Ecole polytechnique.
    
\section*{Reproducibility}

    The numerical code used for all the simulations are written in sequential \texttt{C++} as a part of the code \texttt{SAMURAI} (\textbf{S}tructured \textbf{Adaptive} mesh and \textbf{MU}lti-\textbf{R}esolution based on \textbf{A}lgebra of \textbf{I}ntervals) and can be found at \url{https://github.com/hpc-maths} along with its documentation.

\bibliographystyle{acm}
\bibliography{biblio.bib}

\end{document}

%% file: in_and_out_sets.pdf_tex
%% Creator: Inkscape inkscape 0.92.3, www.inkscape.org
%% PDF/EPS/PS + LaTeX output extension by Johan Engelen, 2010
%% Accompanies image file 'in_and_out_sets.pdf' (pdf, eps, ps)
%%
%% To include the image in your LaTeX document, write
%%   \input{<filename>.pdf_tex}
%%  instead of
%%   \includegraphics{<filename>.pdf}
%% To scale the image, write
%%   \def\svgwidth{<desired width>}
%%   \input{<filename>.pdf_tex}
%%  instead of
%%   \includegraphics[width=<desired width>]{<filename>.pdf}
%%
%% Images with a different path to the parent latex file can
%% be accessed with the `import' package (which may need to be
%% installed) using
%%   \usepackage{import}
%% in the preamble, and then including the image with
%%   \import{<path to file>}{<filename>.pdf_tex}
%% Alternatively, one can specify
%%   \graphicspath{{<path to file>/}}
%% 
%% For more information, please see info/svg-inkscape on CTAN:
%%   http://tug.ctan.org/tex-archive/info/svg-inkscape
%%
\begingroup%
  \makeatletter%
  \providecommand\color[2][]{%
    \errmessage{(Inkscape) Color is used for the text in Inkscape, but the package 'color.sty' is not loaded}%
    \renewcommand\color[2][]{}%
  }%
  \providecommand\transparent[1]{%
    \errmessage{(Inkscape) Transparency is used (non-zero) for the text in Inkscape, but the package 'transparent.sty' is not loaded}%
    \renewcommand\transparent[1]{}%
  }%
  \providecommand\rotatebox[2]{#2}%
  \newcommand*\fsize{\dimexpr\f@size pt\relax}%
  \newcommand*\lineheight[1]{\fontsize{\fsize}{#1\fsize}\selectfont}%
  \ifx\svgwidth\undefined%
    \setlength{\unitlength}{342.50262425bp}%
    \ifx\svgscale\undefined%
      \relax%
    \else%
      \setlength{\unitlength}{\unitlength * \real{\svgscale}}%
    \fi%
  \else%
    \setlength{\unitlength}{\svgwidth}%
  \fi%
  \global\let\svgwidth\undefined%
  \global\let\svgscale\undefined%
  \makeatother%
  \begin{picture}(1,0.75449533)%
    \lineheight{1}%
    \setlength\tabcolsep{0pt}%
    \put(0,0){\includegraphics[width=\unitlength,page=1]{in_and_out_sets.pdf}}%
    \put(-0.00859654,0.31871248){\color[rgb]{0,0.44705882,0.69803922}\makebox(0,0)[lt]{\lineheight{1.25}\smash{\begin{tabular}[t]{l}$\mathcal{E}_{\levelletter, \vectorial{\indexletter}}^{\populationindex}$\end{tabular}}}}%
    \put(0.92399979,0.43142696){\color[rgb]{0,0.61960784,0.45098039}\makebox(0,0)[lt]{\lineheight{1.25}\smash{\begin{tabular}[t]{l}$\mathcal{A}_{\levelletter, \vectorial{\indexletter}}^{\populationindex}$\end{tabular}}}}%
    \put(0,0){\includegraphics[width=\unitlength,page=2]{in_and_out_sets.pdf}}%
    \put(0.31325725,0.43142696){\color[rgb]{0.8,0.4745098,0.65490196}\makebox(0,0)[lt]{\lineheight{1.25}\smash{\begin{tabular}[t]{l}$\vectorial{\normalizedvelocityletter}_{\populationindex}=(1,1)$\end{tabular}}}}%
  \end{picture}%
\endgroup%

%% file: flattening.pdf_tex
%% Creator: Inkscape inkscape 0.92.3, www.inkscape.org
%% PDF/EPS/PS + LaTeX output extension by Johan Engelen, 2010
%% Accompanies image file 'flattening.pdf' (pdf, eps, ps)
%%
%% To include the image in your LaTeX document, write
%%   \input{<filename>.pdf_tex}
%%  instead of
%%   \includegraphics{<filename>.pdf}
%% To scale the image, write
%%   \def\svgwidth{<desired width>}
%%   \input{<filename>.pdf_tex}
%%  instead of
%%   \includegraphics[width=<desired width>]{<filename>.pdf}
%%
%% Images with a different path to the parent latex file can
%% be accessed with the `import' package (which may need to be
%% installed) using
%%   \usepackage{import}
%% in the preamble, and then including the image with
%%   \import{<path to file>}{<filename>.pdf_tex}
%% Alternatively, one can specify
%%   \graphicspath{{<path to file>/}}
%% 
%% For more information, please see info/svg-inkscape on CTAN:
%%   http://tug.ctan.org/tex-archive/info/svg-inkscape
%%
\begingroup%
  \makeatletter%
  \providecommand\color[2][]{%
    \errmessage{(Inkscape) Color is used for the text in Inkscape, but the package 'color.sty' is not loaded}%
    \renewcommand\color[2][]{}%
  }%
  \providecommand\transparent[1]{%
    \errmessage{(Inkscape) Transparency is used (non-zero) for the text in Inkscape, but the package 'transparent.sty' is not loaded}%
    \renewcommand\transparent[1]{}%
  }%
  \providecommand\rotatebox[2]{#2}%
  \newcommand*\fsize{\dimexpr\f@size pt\relax}%
  \newcommand*\lineheight[1]{\fontsize{\fsize}{#1\fsize}\selectfont}%
  \ifx\svgwidth\undefined%
    \setlength{\unitlength}{951.79788816bp}%
    \ifx\svgscale\undefined%
      \relax%
    \else%
      \setlength{\unitlength}{\unitlength * \real{\svgscale}}%
    \fi%
  \else%
    \setlength{\unitlength}{\svgwidth}%
  \fi%
  \global\let\svgwidth\undefined%
  \global\let\svgscale\undefined%
  \makeatother%
  \begin{picture}(1,0.34129332)%
    \lineheight{1}%
    \setlength\tabcolsep{0pt}%
    \put(0,0){\includegraphics[width=\unitlength,page=1]{flattening.pdf}}%
    \put(0.46312525,0){\color[rgb]{0,0,0}\makebox(0,0)[lt]{\lineheight{1.25}\smash{\begin{tabular}[t]{l}$\indexletter$\end{tabular}}}}%
    \put(0.63865276,0){\color[rgb]{0,0,0}\makebox(0,0)[lt]{\lineheight{1.25}\smash{\begin{tabular}[t]{l}$\indexletter+1$\end{tabular}}}}%
    \put(0.82589335,0){\color[rgb]{0,0,0}\makebox(0,0)[lt]{\lineheight{1.25}\smash{\begin{tabular}[t]{l}$\indexletter+2$\end{tabular}}}}%
    \put(0.26830855,0){\color[rgb]{0,0,0}\makebox(0,0)[lt]{\lineheight{1.25}\smash{\begin{tabular}[t]{l}$\indexletter-1$\end{tabular}}}}%
    \put(0.08119934,0){\color[rgb]{0,0,0}\makebox(0,0)[lt]{\lineheight{1.25}\smash{\begin{tabular}[t]{l}$\indexletter-2$\end{tabular}}}}%
    \put(0.99408628,0.03362251){\color[rgb]{0,0,0}\makebox(0,0)[lt]{\lineheight{1.25}\smash{\begin{tabular}[t]{l}$\levelletter$\end{tabular}}}}%
    \put(0.99130449,0.30217833){\color[rgb]{0,0,0}\makebox(0,0)[lt]{\lineheight{1.25}\smash{\begin{tabular}[t]{l}$\maxlevel$\end{tabular}}}}%
    \put(0.34568612,0.32903855){\color[rgb]{0,0.44705882,0.69803922}\makebox(0,0)[lt]{\lineheight{1.25}\smash{\begin{tabular}[t]{l}$\mathcal{E}_{\levelletter, \indexletter}^{\populationindex}$\end{tabular}}}}%
  \end{picture}%
\endgroup%

%% file: 2021_LBM-MR_and_Eq_Eq.bbl
\begin{thebibliography}{10}

\bibitem{abramowitz1964}
{\sc Abramowitz, M., and Stegun, I.~A.}
\newblock {\em Handbook of mathematical functions with formulas, graphs, and
  mathematical tables}, vol.~55.
\newblock US Government printing office, 1964.

\bibitem{bellotti2021multidimensional}
{\sc Bellotti, T., Gouarin, L., Graille, B., and Massot, M.}
\newblock Multidimensional fully adaptive lattice boltzmann methods with error
  control based on multiresolution analysis.
\newblock {\em Journal of Computational Physics\/} (2021).
\newblock Submitted - Available on HAL :
  \url{https://hal.archives-ouvertes.fr/hal-03158073} and ArXiv :
  \url{https://arxiv.org/abs/2103.02903}.

\bibitem{bellotti2021sisc}
{\sc Bellotti, T., Gouarin, L., Graille, B., and Massot, M.}
\newblock Multiresolution-based mesh adaptation and error control for lattice
  boltzmann methods with applications to hyperbolic conservation laws.
\newblock {\em SIAM J. Scientific Computing\/} (2021).
\newblock Submitted - Available on HAL :
  \url{https://hal.archives-ouvertes.fr/hal-03148621} and ArXiv :
  \url{https://arxiv.org/abs/2102.12163}.

\bibitem{bellotti2021compterendus}
{\sc Bellotti, T., Gouarin, L., Graille, B., and Massot, M.}
\newblock Waves passing through grid jumps and multiresolution lattice
  boltzmann methods: fourth order analysis.
\newblock {\em Comptes Rendus - Mathématique\/} (2021).
\newblock In preparation.

\bibitem{bihari1997multiresolution}
{\sc Bihari, B.~L., and Harten, A.}
\newblock Multiresolution schemes for the numerical solution of 2-d
  conservation laws i.
\newblock {\em SIAM Journal on Scientific Computing 18}, 2 (1997), 315--354.

\bibitem{caetano2019result}
{\sc Caetano, F., Dubois, F., and Graille, B.}
\newblock A result of convergence for a mono-dimensional two-velocities lattice
  boltzmann scheme.
\newblock {\em arXiv preprint arXiv:1905.12393\/} (2019).

\bibitem{chapman1990mathematical}
{\sc Chapman, S., Cowling, T.~G., and Burnett, D.}
\newblock {\em The mathematical theory of non-uniform gases: an account of the
  kinetic theory of viscosity, thermal conduction and diffusion in gases}.
\newblock Cambridge university press, 1990.

\bibitem{cohen2003fully}
{\sc Cohen, A., Kaber, S., M{\"u}ller, S., and Postel, M.}
\newblock Fully adaptive multiresolution finite volume schemes for conservation
  laws.
\newblock {\em Mathematics of Computation 72}, 241 (2003), 183--225.

\bibitem{colella1990multidimensional}
{\sc Colella, P.}
\newblock Multidimensional upwind methods for hyperbolic conservation laws.
\newblock {\em Journal of Computational Physics 87}, 1 (1990), 171--200.

\bibitem{dellacherie2014construction}
{\sc Dellacherie, S.}
\newblock Construction and analysis of lattice boltzmann methods applied to a
  1d convection-diffusion equation.
\newblock {\em Acta Applicandae Mathematicae 131}, 1 (2014), 69--140.

\bibitem{devore1984maximal}
{\sc DeVore, R.~A., and Sharpley, R.~C.}
\newblock {\em Maximal functions measuring smoothness}, vol.~293.
\newblock American Mathematical Soc., 1984.

\bibitem{dubois2008equivalent}
{\sc Dubois, F.}
\newblock Equivalent partial differential equations of a lattice boltzmann
  scheme.
\newblock {\em Computers \& Mathematics with Applications 55}, 7 (2008),
  1441--1449.

\bibitem{dubois2013stable}
{\sc Dubois, F.}
\newblock Stable lattice boltzmann schemes with a dual entropy approach for
  monodimensional nonlinear waves.
\newblock {\em Computers \& Mathematics with Applications 65}, 2 (2013),
  142--159.

\bibitem{dupuis2003theory}
{\sc Dupuis, A., and Chopard, B.}
\newblock Theory and applications of an alternative lattice boltzmann grid
  refinement algorithm.
\newblock {\em Physical Review E 67}, 6 (2003), 066707.

\bibitem{eitel2013}
{\sc Eitel-Amor, G., Meinke, M., and Schr{\"o}der, W.}
\newblock A lattice-boltzmann method with hierarchically refined meshes.
\newblock {\em Computers \& Fluids 75\/} (2013), 127--139.

\bibitem{fakhari2016mass}
{\sc Fakhari, A., Geier, M., and Lee, T.}
\newblock A mass-conserving lattice boltzmann method with dynamic grid
  refinement for immiscible two-phase flows.
\newblock {\em Journal of Computational Physics 315\/} (2016), 434--457.

\bibitem{fakhari2014finite}
{\sc Fakhari, A., and Lee, T.}
\newblock Finite-difference lattice boltzmann method with a block-structured
  adaptive-mesh-refinement technique.
\newblock {\em Physical Review E 89}, 3 (2014), 033310.

\bibitem{fakhari2015numerics}
{\sc Fakhari, A., and Lee, T.}
\newblock Numerics of the lattice boltzmann method on nonuniform grids:
  standard lbm and finite-difference lbm.
\newblock {\em Computers \& Fluids 107\/} (2015), 205--213.

\bibitem{filippova1998grid}
{\sc Filippova, O., and H{\"a}nel, D.}
\newblock Grid refinement for lattice-bgk models.
\newblock {\em Journal of Computational physics 147}, 1 (1998), 219--228.

\bibitem{graille2014approximation}
{\sc Graille, B.}
\newblock Approximation of mono-dimensional hyperbolic systems: A lattice
  boltzmann scheme as a relaxation method.
\newblock {\em Journal of Computational Physics 266\/} (2014), 74--88.

\bibitem{harti1993discrete}
{\sc Harten, A.}
\newblock Discrete multi-resolution analysis and generalized wavelets.
\newblock {\em Applied numerical mathematics 12}, 1-3 (1993), 153--192.

\bibitem{hovhannisyan2010}
{\sc Hovhannisyan, N., and M{\"u}ller, S.}
\newblock On the stability of fully adaptive multiscale schemes for
  conservation laws using approximate flux and source reconstruction
  strategies.
\newblock {\em IMA journal of numerical analysis 30}, 4 (2010), 1256--1295.

\bibitem{kruger2017lattice}
{\sc Kr{\"u}ger, T., Kusumaatmaja, H., Kuzmin, A., Shardt, O., Silva, G., and
  Viggen, E.~M.}
\newblock The lattice boltzmann method.
\newblock {\em Springer International Publishing 10}, 978-3 (2017), 4--15.

\bibitem{lallemand2000theory}
{\sc Lallemand, P., and Luo, L.-S.}
\newblock Theory of the lattice boltzmann method: Dispersion, dissipation,
  isotropy, galilean invariance, and stability.
\newblock {\em Physical Review E 61}, 6 (2000), 6546.

\bibitem{landajuela2011burgers}
{\sc Landajuela, M.}
\newblock Burgers equation.
\newblock {\em BCAM Internship report: Basque Center for Applied Mathematics\/}
  (2011).

\bibitem{leveque2002finite}
{\sc LeVeque, R.}
\newblock {\em Finite volume methods for hyperbolic problems}, vol.~31.
\newblock Cambridge university press, 2002.

\bibitem{nguessan2021}
{\sc N’Guessan, M., Massot, M., Series, L., and Tenaud, C.}
\newblock High order time integration and mesh adaptation with error control
  for incompressible navier--stokes and scalar transport resolution on dual
  grids.
\newblock {\em Journal of Computational and Applied Mathematics 387\/} (2021),
  112542.

\bibitem{rohde2006}
{\sc Rohde, M., Kandhai, D., Derksen, J., and Van~den Akker, H.~E.}
\newblock A generic, mass conservative local grid refinement technique for
  lattice-boltzmann schemes.
\newblock {\em International journal for numerical methods in fluids 51}, 4
  (2006), 439--468.

\end{thebibliography}
